\documentclass[11pt]{article}

\usepackage{etoolbox}
\usepackage{subfigure}
\usepackage{amsmath}
\usepackage{amsfonts}
\usepackage[mathscr]{eucal}
\usepackage{amsthm}
\usepackage{amssymb}
\usepackage{enumerate}
\usepackage{tikz}
\usetikzlibrary{matrix, calc, arrows, lindenmayersystems, decorations.pathreplacing, shapes.geometric}
\usepackage{hyperref} 
\usepackage{caption}
\usepackage{soul}	
\usepackage{color}	
\usepackage[algo2e]{algorithm2e}

\DeclareMathOperator{\diam}{diam}
\DeclareMathOperator{\dist}{dist}

\providecommand{\red}[1]{{\color{black}#1}}
\providecommand{\am}[1]{{\color{blue}#1}}

\usepackage{bbm}

\begin{document}
\newcommand{\rf}[1]{(\ref{#1})}
\newcommand{\mmbox}[1]{\fbox{\ensuremath{\displaystyle{ #1 }}}}	

\newcommand{\hs}[1]{\hspace{#1mm}}
\newcommand{\vs}[1]{\vspace{#1mm}}

\newcommand{\ri}{{\mathrm{i}}}
\newcommand{\re}{{\mathrm{e}}}
\newcommand{\rd}{\mathrm{d}}

\newcommand{\Rnotsn}{\mathbb{R}}
\newcommand{\Qnotsn}{\mathbb{Q}}
\newcommand{\Nnotsn}{\mathbb{N}}
\newcommand{\Cnotsn}{\mathbb{C}}
\newcommand{\K}{{\mathbb{K}}}

\newcommand{\cA}{\mathcal{A}}
\newcommand{\cB}{\mathcal{B}}
\newcommand{\cC}{\mathcal{C}}
\newcommand{\cS}{\mathcal{S}}
\newcommand{\cD}{\mathcal{D}}
\newcommand{\cH}{\mathcal{H}}
\newcommand{\cI}{\mathcal{I}}
\newcommand{\cItilde}{\tilde{\mathcal{I}}}
\newcommand{\cIhat}{\hat{\mathcal{I}}}
\newcommand{\cIcheck}{\check{\mathcal{I}}}
\newcommand{\cIstar}{{\mathcal{I}^*}}
\newcommand{\cJ}{\mathcal{J}}
\newcommand{\cM}{\mathcal{M}}
\newcommand{\cP}{\mathcal{P}}
\newcommand{\cV}{{\mathcal V}}
\newcommand{\cW}{{\mathcal W}}
\newcommand{\scrD}{\mathscr{D}}
\newcommand{\scrS}{\mathscr{S}}
\newcommand{\scrJ}{\mathscr{J}}
\newcommand{\sD}{\mathsf{D}}
\newcommand{\sN}{\mathsf{N}}
\newcommand{\sS}{\mathsf{S}}
 \newcommand{\sT}{\mathsf{T}}
 \newcommand{\sH}{\mathsf{H}}
 \newcommand{\sI}{\mathsf{I}}
 
\newcommand{\bs}[1]{\mathbf{#1}}
\newcommand{\bb}{\mathbf{b}}
\newcommand{\bd}{\mathbf{d}}
\newcommand{\bn}{\mathbf{n}}
\newcommand{\bp}{\mathbf{p}}
\newcommand{\bP}{\mathbf{P}}
\newcommand{\bv}{\mathbf{v}}
\newcommand{\bx}{\mathbf{x}}
\newcommand{\by}{\mathbf{y}}
\newcommand{\bz}{{\mathbf{z}}}
\newcommand{\bxi}{\boldsymbol{\xi}}
\newcommand{\boldeta}{\boldsymbol{\eta}}

\newcommand{\ts}{\tilde{s}}
\newcommand{\tGamma}{{\tilde{\Gamma}}}
 \newcommand{\tbx}{\tilde{\bx}}
 \newcommand{\tbd}{\tilde{\bd}}
 \newcommand{\txi}{\xi}
 
\newcommand{\done}[2]{\dfrac{d {#1}}{d {#2}}}
\newcommand{\donet}[2]{\frac{d {#1}}{d {#2}}}
\newcommand{\pdone}[2]{\dfrac{\partial {#1}}{\partial {#2}}}
\newcommand{\pdonet}[2]{\frac{\partial {#1}}{\partial {#2}}}
\newcommand{\pdonetext}[2]{\partial {#1}/\partial {#2}}
\newcommand{\pdtwo}[2]{\dfrac{\partial^2 {#1}}{\partial {#2}^2}}
\newcommand{\pdtwot}[2]{\frac{\partial^2 {#1}}{\partial {#2}^2}}
\newcommand{\pdtwomix}[3]{\dfrac{\partial^2 {#1}}{\partial {#2}\partial {#3}}}
\newcommand{\pdtwomixt}[3]{\frac{\partial^2 {#1}}{\partial {#2}\partial {#3}}}
\newcommand{\bnabla}{\boldsymbol{\nabla}}
\newcommand{\dive}{\boldsymbol{\nabla}\cdot}
\newcommand{\curl}{\boldsymbol{\nabla}\times}
\newcommand{\Phixy}{\Phi(\bx,\by)}
\newcommand{\PhiOxy}{\Phi_0(\bx,\by)}
\newcommand{\dxPhixy}{\pdone{\Phi}{n(\bx)}(\bx,\by)}
\newcommand{\dyPhixy}{\pdone{\Phi}{n(\by)}(\bx,\by)}
\newcommand{\dxPhiOxy}{\pdone{\Phi_0}{n(\bx)}(\bx,\by)}
\newcommand{\dyPhiOxy}{\pdone{\Phi_0}{n(\by)}(\bx,\by)}

\newcommand{\eps}{\varepsilon}
\newcommand{\real}[1]{{\rm Re}\left[#1\right]} 
\newcommand{\ol}[1]{\overline{#1}}
\newcommand{\ord}[1]{\mathcal{O}\left(#1\right)}
\newcommand{\oord}[1]{o\left(#1\right)}
\newcommand{\Ord}[1]{\Theta\left(#1\right)}

\newcommand{\hsnorm}[1]{||#1||_{H^{s}(\bs{R})}}
\newcommand{\hnorm}[1]{||#1||_{\tilde{H}^{-1/2}((0,1))}}
\newcommand{\norm}[2]{\left\|#1\right\|_{#2}}
\newcommand{\normt}[2]{\|#1\|_{#2}}
\newcommand{\on}[1]{\Vert{#1} \Vert_{1}}
\newcommand{\tn}[1]{\Vert{#1} \Vert_{2}}
\newcommand{\xt}{\mathbf{x},t}
\newcommand{\PhiF}{\Phi_{\rm freq}}
\newcommand{\cone}{{c_{j}^\pm}}
\newcommand{\ctwo}{{c_{2,j}^\pm}}
\newcommand{\cthree}{{c_{3,j}^\pm}}

\theoremstyle{thmstyleone}
\newtheorem{thm}{Theorem}[section]
\newtheorem{lem}[thm]{Lemma}
\newtheorem{prop}[thm]{Proposition}
\newtheorem{cor}[thm]{Corollary}

\theoremstyle{thmstyletwo}%
\newtheorem{rem}[thm]{Remark}%

\theoremstyle{thmstylethree}%
\newtheorem{defn}[thm]{Definition}%

\raggedbottom

\newcommand{\tH}{\widetilde{H}}
\newcommand{\Hze}{H_{\rm ze}}
\newcommand{\uze}{u_{\rm ze}}
\newcommand{\dimH}{{\rm dim_H}}
\newcommand{\dimB}{{\rm dim_B}}
\newcommand{\IntClosOm}{\mathrm{int}(\overline{\Omega})}
\newcommand{\IntClosOmOne}{\mathrm{int}(\overline{\Omega_1})}
\newcommand{\IntClosOmTwo}{\mathrm{int}(\overline{\Omega_2})}
\newcommand{\Ccomp}{C^{\rm comp}}
\newcommand{\tCcomp}{\tilde{C}^{\rm comp}}
\newcommand{\uC}{\underline{C}}
\newcommand{\utC}{\underline{\tilde{C}}}
\newcommand{\oC}{\overline{C}}
\newcommand{\otC}{\overline{\tilde{C}}}
\newcommand{\capcomp}{{\rm cap}^{\rm comp}}
\newcommand{\Capcomp}{{\rm Cap}^{\rm comp}}
\newcommand{\tcapcomp}{\widetilde{{\rm cap}}^{\rm comp}}
\newcommand{\tCapcomp}{\widetilde{{\rm Cap}}^{\rm comp}}
\newcommand{\hcapcomp}{\widehat{{\rm cap}}^{\rm comp}}
\newcommand{\hCapcomp}{\widehat{{\rm Cap}}^{\rm comp}}
\newcommand{\tcap}{\widetilde{{\rm cap}}}
\newcommand{\tCap}{\widetilde{{\rm Cap}}}
\newcommand{\ccap}{{\rm cap}}
\newcommand{\ucap}{\underline{\rm cap}}
\newcommand{\uCap}{\underline{\rm Cap}}
\newcommand{\cCap}{{\rm Cap}}
\newcommand{\ocap}{\overline{\rm cap}}
\newcommand{\oCap}{\overline{\rm Cap}}
\DeclareRobustCommand
{\mathringbig}[1]{\accentset{\smash{\raisebox{-0.1ex}{$\scriptstyle\circ$}}}{#1}\rule{0pt}{2.3ex}}
\newcommand{\cirH}{\mathringbig{H}}
\newcommand{\cirHs}{\mathringbig{H}{}^s}
\newcommand{\cirHt}{\mathringbig{H}{}^t}
\newcommand{\cirHm}{\mathringbig{H}{}^m}
\newcommand{\cirHzero}{\mathringbig{H}{}^0}
\newcommand{\deO}{{\partial\Omega}}
\newcommand{\OO}{{(\Omega)}}
\newcommand{\Rn}{{(\R^n)}}
\newcommand{\Id}{{\mathrm{Id}}}
\newcommand{\gap}{\mathrm{Gap}}
\newcommand{\ggap}{\mathrm{gap}}
\newcommand{\isom}{{\xrightarrow{\sim}}}
\newcommand{\half}{{1/2}}
\newcommand{\mhalf}{{-1/2}}
\newcommand{\inter}{{\mathrm{int}}}

\newcommand{\Hsp}{H^{s,p}}
\newcommand{\Htq}{H^{t,q}}
\newcommand{\tHsp}{{{\widetilde H}^{s,p}}}
\newcommand{\SP}{\ensuremath{(s,p)}}
\newcommand{\Xsp}{X^{s,p}}

\newcommand{\dd}{{d}}\newcommand{\pp}{{p_*}}

\newcommand{\Rnn}{\R^{n_1+n_2}}
\newcommand{\Tr}{{\mathrm{Tr}}}

\renewcommand{\arraystretch}{1.7}
\renewcommand{\bs}[1]{{\boldsymbol{#1}}} %
\newcommand{\vb}[1]{\vec{\bs{#1}}}
\newcommand{\be}{\bs{e}}
\renewcommand{\bn}{\bs{n}}
\renewcommand{\bx}{\bs{x}}
\renewcommand{\by}{\bs{y}}
\newcommand{\bg}{\bs{g}}
\newcommand{\bu}{\bs{u}}
\newcommand{\bw}{\bs{w}}
\newcommand{\bA}{\bs{A}}
\newcommand{\bC}{\bs{C}}
\newcommand{\bL}{\bs{L}}
\newcommand{\bS}{\bs{S}}
\newcommand{\bT}{\bs{T}}
\newcommand{\bU}{\bs{U}}
\newcommand{\bV}{\bs{V}}
\newcommand{\vbE}{\vb{E}}
\newcommand{\bX}{\bs{X}}
\newcommand{\bgamma}{{\bs{\gamma}}}
\newcommand{\bH}{\bs{H}}
\newcommand{\bnu}{\boldsymbol{\nu}}
\newcommand{\btau}{\boldsymbol{\tau}}
\newcommand{\bseta}{\boldsymbol{\eta}}
\newcommand{\rD}{{D}}
\newcommand{\rN}{\mathrm{N}}
\newcommand{\bm}{\bs{m}}
\newcommand{\No}{{\mathbb{N}_0}}
\newcommand{\tbX}{\tilde{\bX}}
\newcommand{\tA}{\tilde{A}}
\renewcommand{\tH}{\widetilde{H}{}}
\newcommand{\tbH}{\widetilde{\bH}{}}
\newcommand{\sM}{\mathsf{M}}
\newcommand{\sE}{\mathsf{E}}
\newcommand{\cT}{\mathcal{T}}
\newcommand{\cU}{\mathcal{U}}
\newcommand{\cF}{\mathcal{F}}
\newcommand{\cL}{\mathcal{L}}
\newcommand{\cK}{\mathcal{K}}
\newcommand{\cN}{\mathcal{N}}
\newcommand{\cE}{\mathcal{E}}
\newcommand{\cR}{\mathcal{R}}
\newcommand{\tcA}{\tilde{\mathcal{A}}}
\newcommand{\tcL}{\tilde{\mathcal{L}}}
\newcommand{\tcK}{\tilde{\mathcal{K}}}
\newcommand{\bcD}{\boldsymbol{\mathcal{D}}}%
\newcommand{\vbcU}{\vec{\boldsymbol{\cU}}}
\newcommand{\vbcE}{\vec{\boldsymbol{\cE}}}
\newcommand{\vbcS}{\vec{\boldsymbol{\cS}}}
\newcommand{\bscrS}{\boldsymbol{\scrS}}
\newcommand{\dudnjump}{\left[ \pdone{u}{n}\right]}
\newcommand{\dudnjumptext}{[ \pdonetext{u}{n}]}
\newcommand{\sumpm}[1]{\overbracket[0.5pt]{\underbracket[0.5pt]{\,#1\,}}}
\renewcommand{\dudnjump}{\left[ \pdone{u^s}{n}\right]}
\renewcommand{\dudnjumptext}{[ \pdonetext{u^s}{n}]}
\newcommand{\gradd}{\vec{\rm \bf grad}}
\newcommand{\curld}{\vec{\rm \bf curl}}
\newcommand{\dived}{{\rm div}}			
\newcommand{\gradp}{{\rm \bf grad}}	%
\newcommand{\curlp}{{\rm \bf curl}}	%
\newcommand{\scurlp}{{\rm curl}}		%
\newcommand{\divep}{{\rm div}}			%
\newcommand{\gradg}{{\rm \bf grad}_\Gamma}	%
\newcommand{\curlg}{{\rm \bf curl}_\Gamma}	%
\newcommand{\scurlg}{{\rm curl}_\Gamma}	%
\newcommand{\diveg}{{\rm div}_\Gamma}		%
\newcommand{\Deltag}{{\Delta_\Gamma}}		%
\newcommand{\kap}{m}
\newcommand{\buk}{\bu_\kap}
\newcommand{\hbuk}{\hat{\bu}_\kap}
\newcommand{\bvk}{\bv_\kap}
\newcommand{\hbvk}{\hat{\bv}_\kap}
\newcommand{\bwk}{\bw_\kap}
\newcommand{\hbwk}{\hat{\bw}_\kap}
\newcommand{\hbu}{\hat{\bu}}
\newcommand{\hbv}{\hat{\bv}}
\newcommand{\hbw}{\hat{\bw}}
\newcommand{\hphi}{\hat{\phi}}
\newcommand{\hpsi}{\hat{\psi}}
\newcommand{\deG}{{\partial\Gamma}}
\newcommand{\GG}{(\Gamma)}
\newcommand{\tr}{\mathrm{tr}_\Gamma}
\newcommand{\trs}{\mathrm{tr}_{\Gamma,s}}
\newcommand{\trhalf}{\mathrm{tr}_{\Gamma,\frac{1}{2}}}
\newcommand{\IH}{\mathbb{H}}
\newcommand{\IS}{\mathbb{S}}
\newcommand{\IL}{\mathbb{L}}
\newcommand{\II}{\mathbb{I}}
\newcommand{\IV}{\mathbb{V}}
\newcommand{\IW}{\mathbb{W}}
\newcommand{\IC}{\mathbb{C}}
\newcommand{\IX}{\mathbb{X}}
\newcommand{\IP}{\mathbb{P}}
\newcommand{\IQ}{\mathbb{Q}}
\newcommand{\IY}{\mathbb{Y}}
\newcommand{\tf}{\tilde{f}}
\newcommand{\dual}[2]{\left\langle #1\,,\,#2\right\rangle}
\newcommand{\wt}[1]{\widetilde{#1}}
\newcommand{\divepR}{{\rm div}_{\Rnotsn^2}}
\newcommand{\scurlpR}{{\rm curl}_{\Rnotsn^2}}
\newcommand{\Hull}{\mathrm{Hull}}
\newcommand{\UnionHull}{\Gamma_{\Hull,h}}
\newcommand{\UnionHullbm}{(\Gamma_{\bm})_{\Hull,h}}
\newcommand{\UnionHullbmp}{(\Gamma_{\bm'})_{\Hull,h}}
\newcommand{\cosc}{c_{\rm osc}}
\definecolor{purple0}{rgb}{0.4,0,0.5}
\definecolor{orange}{rgb}{1,0.4,0}
\newcommand{\cu}[1]{{\color{purple0} #1 }}
\definecolor{orange0}{rgb}{1,0.3,0}
\newcommand{\ctodo}[1]{{\color{orange0} {\bf TODO:} #1 }}
\definecolor{green0}{rgb}{0.1,0.6,0}
\newcommand{\dnote}[1]{{\color{green0} {\bf DH:} #1 }}
\newcommand{\dimloc}{\dim_{\mathrm{loc}}}

\title{Numerical Quadrature for Singular Integrals on Fractals}

\author{A.\ Gibbs$^{\text{a}}$,
	D.\ P.\ Hewett$^{\text{a}}$, 
	A. Moiola$^{\text{b}}$\\[2pt]
	$^{\text{a}}${\footnotesize Department of Mathematics, University College London, London, United Kingdom}\\[2pt]
	$^{\text{b}}${\footnotesize Dipartimento di Matematica ``F. Casorati'', Universit\`a degli studi di Pavia, Pavia, Italy}
}

\renewcommand{\thefootnote}{\arabic{footnote}}

\maketitle

\begin{abstract}
We present and analyse numerical quadrature rules for evaluating regular and singular integrals on self-similar fractal sets. The integration domain $\mathbb{R}^n$ is assumed to be the compact attractor of an iterated function system of contracting similarities satisfying the open set condition. Integration is with respect to any ``invariant'' (also known as ``balanced'' or ``self-similar'') measure supported on $\Gamma$, including in particular the Hausdorff measure $\mathcal{H}^d$ restricted to $\Gamma$, where $d$ is the Hausdorff dimension of $\Gamma$. Both single and double integrals are considered. 
Our focus is on composite quadrature rules in which integrals over $\Gamma$ are decomposed into sums of integrals over suitable partitions of $\Gamma$ into self-similar subsets. 
For certain singular integrands of logarithmic or algebraic type we show how in the context of such a partitioning the invariance property of the measure can be exploited to express the singular integral exactly in terms of regular integrals. 
For the evaluation of these regular integrals we adopt a composite barycentre rule, which for sufficiently regular integrands exhibits second-order convergence with respect to the maximum diameter of the subsets. 
As an application we show how this approach, combined with a singularity-subtraction technique, can be used to accurately evaluate the singular double integrals that arise in Hausdorff-measure Galerkin boundary element methods for acoustic wave scattering by fractal screens. 
\end{abstract}



\section{Introduction}\label{sec:intro}

In this paper we study numerical quadrature rules for the evaluation of integrals of the form 
\begin{equation}\label{eq:single}
\int_{\Gamma} f(x) \,\rd\red{\mu}(x),
\end{equation}
and 
\begin{equation}\label{eq:double}
\int_{\Gamma}\int_{\Gamma'} f(x,y) \,\rd\red{\mu'}(y)\rd\red{\mu}(x),
\end{equation}
where $\Gamma$ and $\Gamma'$ are compact fractal subsets of $\Rnotsn^n$ - more precisely, the attractors of iterated function systems (IFSs) satisfying the open set condition (OSC) (see Subsection \ref{sec:IFS}) - \red{and $\mu$ and $\mu'$ are ``invariant'' (also known as ``balanced'' or ``self-similar'') measures on $\Gamma$ and $\Gamma'$ respectively (see Subsection \ref{sec:scaling}). A special case is where $\mu=\cH^d|_\Gamma$ and $\mu'=\cH^{d'}|_{\Gamma'}$, where} $\cH^{d}$ and $\cH^{d'}$ are Hausdorff measures, with $d$, $d'$ denoting the Hausdorff dimensions of $\Gamma$ and $\Gamma'$ (see Subsection \ref{sec:HausdorffMeasure}), \red{and $|_\Gamma$ denotes the restriction to $\Gamma$ in the sense that $\cH^d|_\Gamma(A):=\cH^d(\Gamma\cap A)$ for $A\subset\mathbb{R}^n$.}
Our particular interest is in singular integrals, where, in the case of \eqref{eq:single}, $f$ is singular at some point $\eta\in \Gamma$, and, in the case of \eqref{eq:double}, $f$ is singular on $x=y$. 

One context in which such integrals arise is in the discretization of certain boundary integral equation and volume integral equation formulations of boundary value problems for elliptic PDEs (such as the Laplace or Helmholtz equation) posed on domains with fractal boundary, for instance in the scattering of electromagnetic and acoustic waves by fractal obstacles \cite{ScreenPaper,BEMfract}, applications of which include antenna design in electrical engineering \cite{PBaRomPouCar:98,SriRanKri:11} and the quantification of the scattering effect of atmospheric ice crystals in climate modelling \cite{So:01}. 
Our main motivating example is the ``Hausdorff boundary element method (BEM)'' introduced in \cite{HausdorffBEM} for the solution of time-harmonic acoustic scattering in $\Rnotsn^{n+1}$ ($n=1,2$) by a sound-soft fractal screen $\Gamma_{\rm screen}\subset \Rnotsn^n\times \{0\}$, assumed to be the attractor of an IFS satisfying the OSC. 
The BEM proposed in \cite{HausdorffBEM} discretizes the associated single layer boundary integral equation on $\Gamma_{\rm screen}$ using an approximation space consisting of products of the relevant Hausdorff measure with piecewise-constant functions on a ``mesh'' of $\Gamma_{\rm screen}$ comprising self-similar fractal ``elements'', which are \red{subsets} of $\Gamma_{\rm screen}$ obtained as scaled, rotated and translated copies of $\Gamma_{\rm screen}$ via the IFS structure. The entries of the right-hand side vector in the Galerkin BEM system then involve integrals of the form \eqref{eq:single} \red{(with $\mu=\cH^d|_\Gamma$)}, where $\Gamma$ is an element of the mesh and $f(x)$ depends on the incident wave. The Galerkin BEM system matrix entries involve integrals of the form \eqref{eq:double} \red{(with $\mu=\cH^d|_\Gamma$ and $\mu'=\cH^{d'}|_{\Gamma'}$)}, where $\Gamma$ and $\Gamma'$ are elements of the mesh and $f(x,y)=\Phi(x,y)$, where $\Phi(x,y)$ is the fundamental solution of the Helmholtz equation in $\Rnotsn^{n+1}$, viz. 
\begin{equation}\label{eq:Helmker2d}
\Phi(x,y) = 
\begin{cases}
\frac{\ri}{4} H_0^{(1)}(k|x-y|),& n=1,\\[2mm]
\dfrac{\re^{\ri k |x-y|}}{4\pi |x-y|}, & n=2,
\end{cases}
\qquad x\neq y,
\end{equation}
where $k>0$ is the wavenumber and $H_\nu^{(1)}$ denotes the Hankel function of the first kind of order $\nu$. 
This choice of $f(x,y)$ makes \eqref{eq:double} singular when $\Gamma=\Gamma'$. 
Although not studied in \cite{HausdorffBEM}, one could also consider a collocation (as opposed to Galerkin) method for the same integral equation and approximation space, in which case the collocation matrix entries would involve integrals of the form \eqref{eq:single} \red{(with $\mu=\cH^d|_\Gamma$)}, where $\Gamma$ is an element of the mesh and $f(x)=\Phi(x,\eta)$, with $\eta$ denoting a collocation point, giving a singular integral when $\eta\in\Gamma$.
A key goal of the current paper is to present a detailed derivation and rigorous error analysis of the quadrature rules used in the implementation of the Hausdorff BEM in \cite{HausdorffBEM}. But we expect that the techniques we present will be of wider interest, since they apply \red{to general invariant measures,} to general regular integrands, and to a quite general class of singular integrands with logarithmic or algebraic singularities.  %

Our quadrature rules for \eqref{eq:single} and \eqref{eq:double} are based on decomposing integrals over $\Gamma$ into sums of integrals over suitable partitions of $\Gamma$ into self-similar \red{subsets}, generated using the IFS structure, in the same way that the Hausdorff BEM meshes are constructed in \cite{HausdorffBEM}. 
After reviewing some preliminaries in Section \ref{sec:prelim}, we start in Section \ref{sec:smooth_quad} by considering regular integrands. Applying a one-point quadrature rule on each \red{subset} leads to a composite quadrature rule, which, when the quadrature nodes are chosen as the barycentres (with respect to \red{$\mu$ or $\mu'$}) of the \red{subsets}, can achieve second-order convergence with respect to the maximum diameter of the \red{subsets} (Theorems \ref{th:MidLip1} and \ref{th:MidLip2}).
In Section \ref{sec:Phi_t} we then consider a special class of singular integrands, indexed by $t\ge0$, namely $f(x)=\Phi_t(x,\eta)$ (in the case of \eqref{eq:single}), and $f(x,y)=\Phi_t(x,y)$ with $\Gamma'=\Gamma$ (in the case of \eqref{eq:double}), where 
\begin{equation}\label{eq:Phit}
\Phi_t(x,y):=
\left\{
\begin{array}{ll}
\log |x-y|,& t=0,\\
|x-y|^{-t},& t>0.
\end{array}
\right.
\end{equation}
For these particular choices of $f$ (and for a particular choice of $\eta$ in the case of \eqref{eq:single}), we \red{use the fact (as noted in e.g.~\cite{bessis1987mellin} for the case of Cantor sets) that} the \red{invariance property of $\mu$} 
and the homogeneity property of $\Phi_t$ can be exploited to express the singular integral exactly in terms of regular integrals (Theorems \ref{th:Jsingthm} and \ref{th:sing_reform}), which can be evaluated using our composite barycentre rule, again with second-order accuracy (Corollaries \ref{cor:Jsingconvthm} and \ref{cor:Phi_t}). 
The results on $\Phi_t$ are directly relevant to Hausdorff-BEM formulations of Laplace problems analogous to the Helmholtz ones described above. 
In Section \ref{sec:Phi} we combine this \red{approach} to integrating $\Phi_t$ with a singularity-subtraction technique (cf.~\cite{Gr:82,An:80,Sc:68,AmBo:95}) to propose and analyse a second-order accurate quadrature rule for the motivating example from \cite{HausdorffBEM} discussed above, 
exploiting the fact that the singular behaviour of $\Phi(x,y)$ matches that of $\Phi_t(x,y)$ for $t=n-1$, $n=1,2$. Here the main result is Theorem \ref{thm:PhiDoubleSingular}. 
In Section \ref{sec:Numerics} we present some numerical results illustrating our theory.  \red{In the Appendix we collect some results concerning the integrability of singular functions with respect to invariant measures on IFS attractors.}

For ease of reference, we list the quadrature rules that we propose:
\begin{itemize}
\item the barycentre rule $Q_\Gamma[f]$ \eqref{eq:BaryQuad} for single regular integrals;
\item the barycentre rule $Q_{\Gamma,\Gamma'}[f]$ \eqref{eq:QGG} for double regular integrals;
\item the rule $Q_{\Gamma,t,m}^h$ \eqref{eq:SingleSingQuad} for single integrals of the singular integrand $\Phi_t$;
\item the rule $Q_{\Gamma,\Gamma,t}^h$ \eqref{eq:non_overlap_quad_gen} for double integrals of the singular integral $\Phi_t$;
\item the singularity-subtraction rule $Q_{\Gamma,\Gamma,\Phi}^h$ \eqref{eq:PhiQuad} for double integrals of the singular Helmholtz fundamental solution $\Phi$; this rule reduces to \eqref{eq:reduced} for small wavenumbers.
\end{itemize}
\red{A link to our open-source implementation of these rules, and a pointer to an interactive notebook showing example usage, is provided in Section~\ref{sec:Numerics}.} 

To put our results in the context of related work, we note that in the case $n=1$, i.e.\ when $\Gamma\subset\Rnotsn$, Gauss quadrature rules can be derived for \eqref{eq:single} (and applied to \eqref{eq:double} iteratively), as discussed e.g.\ in \cite{Ma:96,Ma:07}. For sufficiently regular integrands these rules offer superior convergence rates when compared to our composite barycentre rules.
However, one advantage of our low-order composite approach is that it is also generically applicable for $\Gamma\subset\Rnotsn^n$, $n>1$. By contrast, stable Gauss rules are not in general available for the case $n>1$ (except when $\Gamma$ is a subset of a line, in which case the $n=1$ results apply). Moreover, the case $n>1$ cannot in general be treated by taking Cartesian products of Gauss rules, since IFS attractors in $\Rnotsn^n$, $n>1$, are not in general the Cartesian product of IFS attractors in $\Rnotsn$. 
Furthermore, even when $\Gamma\subset\Rnotsn^n$ has such a Cartesian product structure, as is the case for the Cantor dust in Figure~\ref{fig:disjointness}(I), the corresponding \red{invariant} measure is not the tensor-product measure of the respective lower-dimensional \red{invariant} measures (see e.g.\ \cite[Proposition~7.1]{Fal}).
We stress, however, that if a stable Gauss rule (or any other quadrature rule for regular integrals) is available, it can be used in place of our composite barycentre rule within the context of our singularity-subtraction and \red{invariance} techniques for singular integrals, with corresponding analogues of the convergence results in Corollaries \ref{cor:Jsingconvthm} and \ref{cor:Phi_t} and Theorem \ref{thm:PhiDoubleSingular}.

\red{We also note that other quadrature rules for regular integrals on IFS attractors have been investigated in the abstract framework of uniform distributions and discrepancies in \cite{InfVol09,DrmInf12}, and in the context of so-called ``chaos games'' (i.e.\ Monte-Carlo-type algorithms)  in e.g.\ \cite{forte1998chaos}.
In both settings, convergence is typically proved for general continuous integrands, but convergence rates for smoother integrands are not provided. We present a numerical comparison between our quadrature rules and a simple chaos game approach in Section \ref{sec:Numerics}.}

We end this introduction with a comment relating to the practical evaluation of the quadrature rules presented in this paper. All our rules require knowledge of $\red{\mu}(\Gamma)$, since this quantity appears as a multiplicative factor in the formula \eqref{eq:weights_def} for the weights in the barycentre rule on which all our other quadrature rules are based. 
Somewhat surprisingly, \red{even in the special case $\mu=\cH^d|_\Gamma$} the exact value of $\cH^d(\Gamma)$ is known only for a small number of IFS attractors, including the middle third Cantor set in $\Rnotsn$ but not including the middle third Cantor dust in $\Rnotsn^2$ - we provide a more detailed commentary on the current state of knowledge regarding $\cH^d(\Gamma)$ in Remark \ref{rem:HdGamma}. 
This means that, in practice, it is in general not possible to compute \eqref{eq:single} and \eqref{eq:double} even for $f=1$!
However, this does not compromise the utility of our quadrature rules for the main motivating application of this paper, namely the implementation of the Galerkin Hausdorff BEM for acoustic scattering by a fractal screen $\Gamma_{\rm screen}$ in \cite{HausdorffBEM}\red{, and for similar possible applications to other differential and integral equations}.
This is because when one uses a BEM to solve a wave scattering problem, the physically relevant quantities such as the scattered wave field and its far-field pattern are unaffected by the choice of normalisation of the surface measure used in the BEM calculations. Working with the normalised measure $\cH_{\star}^d(\cdot):=\cH^d(\cdot)/\cH^d(\Gamma_{\rm screen})$ in the BEM application leads to integrals of the form \eqref{eq:single} and \eqref{eq:double} with $\cH^d$ replaced by $\cH_{\star}^d$. Our quadrature rules apply \textit{mutatis mutandis} to such integrals, requiring the value of $\cH_{\star}^d(\Gamma)$, but this can be computed for any subcomponent $\Gamma$ of $\Gamma_{\rm screen}$ using the IFS structure, because $\cH_{\star}^d$ has the same self-similarity scaling properties as $\cH^d$ on subcomponents of $\Gamma_{\rm screen}$ (cf.\ \eqref{eq:sim_measure} below), and $\cH_{\star}^d(\Gamma_{\rm screen})$ is known ($\cH_{\star}^d(\Gamma_{\rm screen})=1$ by the definition of $\cH_{\star}^d$). For details see 
\cite{HausdorffBEM}.

\section{Preliminaries}
\label{sec:prelim}

We begin by reviewing a number of basic results about IFS attractors and integration on them, and introduce the notation and terminology we will use throughout the paper. 

\subsection{Hausdorff measure and dimension}
\label{sec:HausdorffMeasure}

For $E\subset\Rnotsn^n$ and $\alpha\geq 0$ we recall (e.g.\ \cite[Section 3]{Fal}) the definition of the Hausdorff $\alpha$-measure of $E$, 
\[ \cH^\alpha(E)=\lim_{\delta\to 0} \bigg(\inf \sum_{i=1}^\infty \big(\diam(U_i)\big)^\alpha\bigg)\in[0,\infty)\cup\{\infty\},\]
where the infimum is over all countable covers of $E$ by 
sets $U_i\subset\Rnotsn^n$ with $\diam(U_i)\leq \delta$ for $i\in\Nnotsn$.
The Hausdorff dimension of $E$ is then defined to be 
\[ \dimH E
= \sup\big\{\alpha\in \Rnotsn^+: \cH^\alpha(E)=\infty\big\}
=\inf\big\{\alpha\in \Rnotsn^+: \cH^\alpha(E)=0\big\}\in[0,n],\]
\red{where the supremum of the empty set is taken to be zero.} Given $0< d\leq n$, we call a non-empty closed set $\Gamma\subset \Rnotsn^n$ a {\em$d$-set} if there exist $c_{2}>c_{1}>0$ such that
\begin{align}
\label{eq:dset}
c_{1}r^{d}\leq\mathcal{H}^{d}\big(\Gamma\cap B_{r}(x)\big)\leq c_{2}r^{d},\qquad x\in\Gamma,\quad0<r\leq 1,
\end{align}
where $B_r(x)$ denotes the closed ball of radius $r$ centred at $x$. 
This definition is equivalent to the definitions given in \cite[Subsubsection II.1.1]{JoWa84} and \cite[Subsection 3.1]{Triebel97FracSpec}, by \cite[Subsection 3.4]{Triebel97FracSpec}.
Condition \eqref{eq:dset} implies not only that $\dimH(\Gamma)=d$ \cite[Cor~3.6]{Triebel97FracSpec}, but moreover that $\Gamma$ is \textit{uniformly} $d$-dimensional in the sense that $\dimH(\Gamma\cap B_r(x))=d$ for every $x\in \Gamma$ and $r>0$. 
If $\Gamma$ is compact then condition \eqref{eq:dset} also gives that $0<\cH^d(\Gamma)<\infty$. 
and there exist $\tilde{c}_2>\tilde{c}_1>0$, depending only on $c_1,c_2$ %
and $\diam(\Gamma)$, such that 
\begin{align}
\label{eq:dset2}
\tilde{c}_{1}r^{d}\leq\cH^d\big(\Gamma\cap B_{r}(x)\big)\leq \tilde{c}_{2}r^{d},\qquad x\in\Gamma,\quad0<r\leq \diam(\Gamma).
\end{align}

\subsection{Iterated function systems}\label{sec:IFS}
Throughout the paper we assume that $\Gamma$ is the 
attractor of an iterated function system (IFS) of contracting similarities \red{(see e.g. \cite{hutchinson1981fractals,Fal})}, by which we mean 
a collection $\{s_1,s_2,\ldots,s_M\}$, for some $M\in\Nnotsn$, $M\geq 2$, where, for each $m=1,\ldots,M$, $s_m:\Rnotsn^n\to\Rnotsn^n$ satisfies, for some $\rho_m\in (0,1)$,
\begin{equation*}%
|s_m(x)-s_m(y)| = \rho_m|x-y|,\quad\text{for }x,y\in \Rnotsn^{n}.
\end{equation*} 
Explicitly, for each $m=1,\ldots,M$ we can write  %
\begin{equation}\label{eq:affine}
s_m(x) = \rho_mA_mx + \delta_m,
\end{equation}
for some orthogonal matrix $A_m\in\Rnotsn^{n\times n}$ and some translation $\delta_m\in\Rnotsn^n$. 
We denote by $\eta_{m}:=(I-\rho_mA_m)^{-1}\delta_m$ ($I$ being the $n\times n$ identity matrix) the fixed point of the contracting similarity $s_{m}$, i.e.\ the unique point $\eta_{m}\in \Rnotsn^n$ such that $s_m(\eta_m)=\eta_m$. 
Saying that $\Gamma$ is the attractor of the IFS means that $\Gamma$ is the unique non-empty compact set satisfying
\begin{equation*}
\Gamma = s(\Gamma), 
\end{equation*}
where 
\begin{align}
\label{eq:fullmap}
s(E) := \bigcup_{m=1}^M  s_m(E), \quad  E\subset \Rnotsn^{n}.
\end{align}
We shall also assume throughout that the \textit{open set condition} (OSC) \cite[Section 5.2]{hutchinson1981fractals} 
holds, meaning that there exists a non-empty bounded open set $O\subset \Rnotsn^{n}$ such that
\begin{align} \label{oscfirst}
s(O) \subset O \quad \mbox{and} \quad s_m(O)\cap s_{m'}(O)=\emptyset, \quad m\neq m'\in\{1,\ldots,M\}.
\end{align}
Then 
$\Gamma$ is a $d$-set (e.g.\ \cite[Thm.~4.7]{Triebel97FracSpec}), %
where $d\in (0,n]$ is the 
unique solution of %
\begin{align} \label{eq:dfirst}
\sum_{m=1}^M  (\rho_m)^d = 1.
\end{align}
Furthermore, $\cH^d(s_{m}(\Gamma))\cap s_{m'}(\Gamma))=0$ for $m\neq m'$, a property known as \textit{self-similarity} \cite[5.1(4)(ii)]{hutchinson1981fractals}.

The best-known example of an IFS attractor is the Cantor set $\Gamma\subset\Rnotsn$, 
defined by 
\begin{align}
\label{eq:CantorDef}
n=1, \quad M=2, \quad s_1(x)=\rho x, \quad s_2(x) = (1-\rho) + \rho x,
\end{align}
for some $\rho\in(0,1/2)$, the choice $\rho=1/3$ corresponding to the classical ``middle third'' case.

\subsection{Further assumptions on the IFS}
\label{sec:FurtherAssumptions}

\begin{figure}[t!]
\centering
\hspace{6pt}(I)\fbox{\includegraphics[width=.35\textwidth,clip,trim=70 20 60 10]{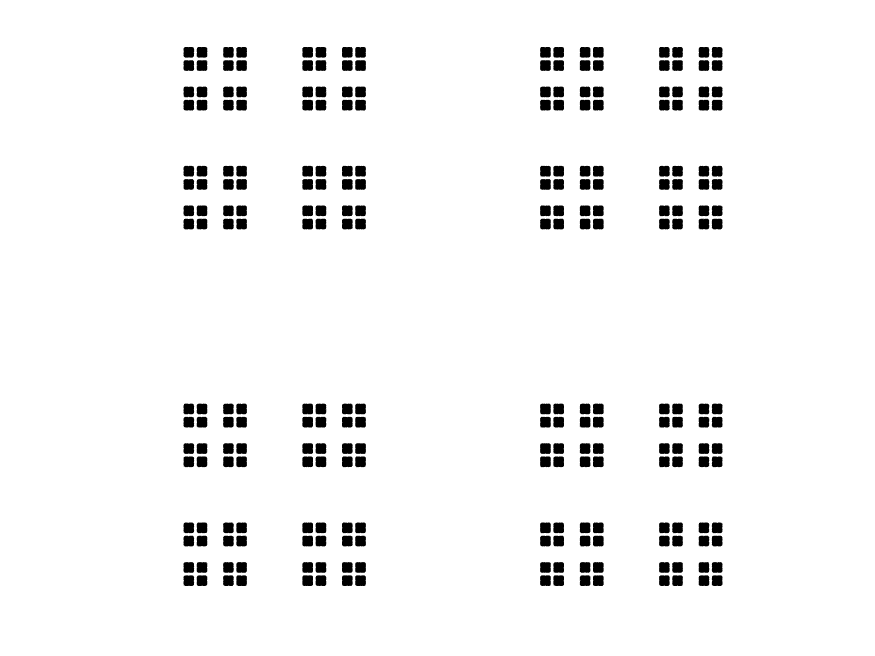}}
\fbox{\includegraphics[width=.35\textwidth,clip,trim=30 0 30 0]{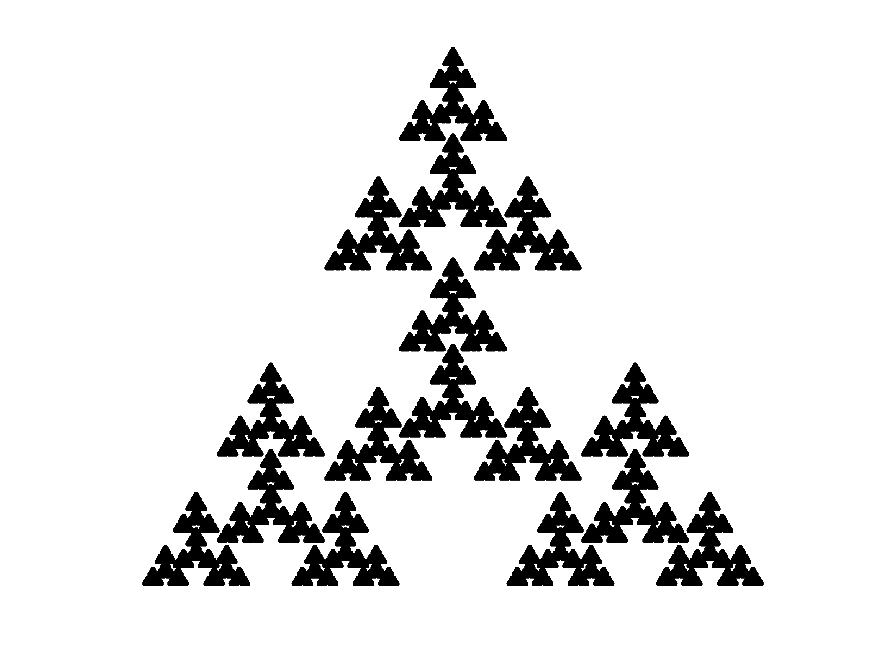}\rule{0pt}{.344\textwidth}}(II)
\\
(III)\fbox{\includegraphics[width=.35\textwidth,clip,trim=70 20 60 10]{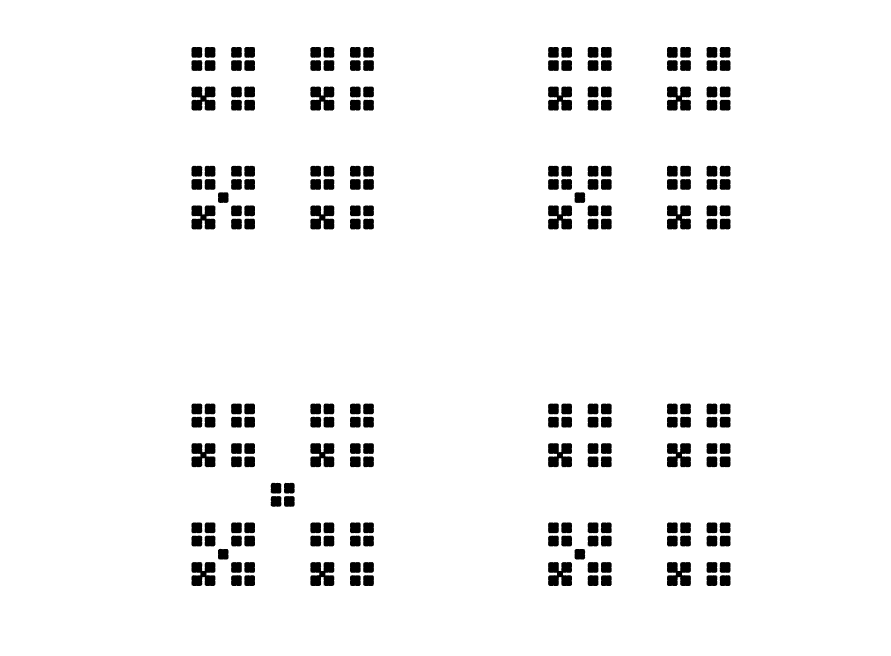}\rule{0pt}{.344\textwidth}}
\fbox{\includegraphics[width=.35\textwidth,clip,trim=70 20 60 10]{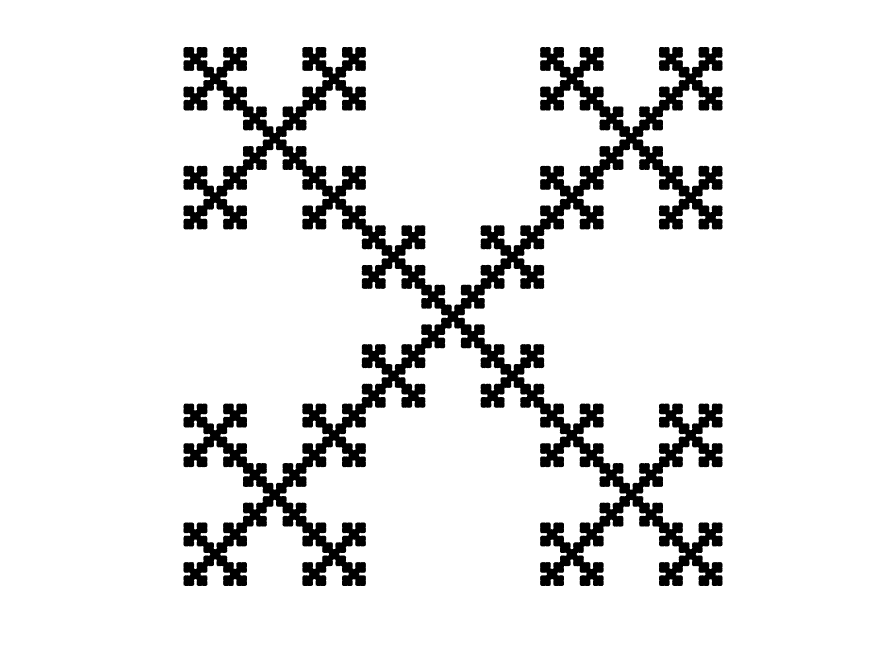}}(IV)
\caption{Four IFS attractors in $\Rnotsn^2$ with different degrees of ``disjointness''.
The formulas of all contractions are in Table~\ref{t:disjointness}. \red{Homogeneity}, disjointness and hull-disjointness were defined in 
Subsection \ref{sec:FurtherAssumptions}. 
\newline
(I) Top left: a \red{homogeneous} hull-disjoint IFS (a Cantor dust with $\rho=1/3$).
\newline
(II) Top right: a \red{homogeneous} disjoint IFS that is not hull-disjoint and such that $\eta_m\notin\Hull(\Gamma_{m'})$ for all $m\ne m'$ (recall that $\eta_m$ is the fixed point of $s_m$).
The fixed points are the three vertices and the centre of $\Hull\GG$.
This condition on the fixed points is relevant e.g.\ in Corollary~\ref{cor:Phi_t}.
\newline
(III) Bottom left: a non-\red{homogeneous} disjoint IFS that is not hull-disjoint and such that $\eta_m\in\Hull(\Gamma_{m'})$ for some $m'\ne m$.
(In particular, $s_5\GG\subset\Hull(s_3\GG)$.)
\newline
(IV) Bottom right: a \red{homogeneous} IFS attractor that is not disjoint (the Vicsek fractal).
\newline
All these examples satisfy the OSC \eqref{oscfirst}:
(I) and (IV) for the open square $O=(0,1)^2$, (II) and (III) (and (I) again) for a neighbourhood $O=\Gamma+B_\epsilon(0)$ with sufficiently small $\epsilon$.}
\label{fig:disjointness}
\end{figure}

We say that the IFS is \red{\textit{homogeneous} (as in, e.g., \cite{elton1989approximation})} if $\rho_m=\rho\in (0,1)$ for $m=1,\ldots,M$. 
In this case
\eqref{eq:dfirst} becomes
\begin{equation} \label{eq:d2}
M\rho^d =1, \quad \text{ equivalently, } \quad
d = \log(M)/\log(1/\rho).
\end{equation}
We say that the IFS is \emph{disjoint} 
\red{(as in, e.g., \cite[Defn 7.1]{barnsley2013developments})}
if 
\begin{equation}\label{eq:Rdef}
R_\Gamma:=\min_{m\neq m'}\big\{\dist \big(  s_{m}(\Gamma), s_{m'}(\Gamma)\big)  \big\}>0,
\end{equation}
which holds if and only if the open set $O$ in the OSC can be taken such that $\Gamma\subset O$ 
(e.g.~\cite{HausdorffBEM}).

We say that the IFS is \emph{hull-disjoint} if 
\begin{equation}\label{eq:RHulldef}
R_{\Gamma,\Hull}:=\min_{m\neq m'}\Big\{\dist \Big(  s_{m}\big(\Hull(\Gamma)\big), s_{m'}\big(\Hull(\Gamma)\big)\Big)  \Big\}>0,
\end{equation}
where $\Hull(E)$ denotes the convex hull of a set $E\subset\Rnotsn^n$. 
Clearly $R_{\Gamma,\Hull}\leq R_\Gamma$, so hull-disjointness implies disjointness. 
But the converse is not true -- see (II) and (III) in Figure~\ref{fig:disjointness} %
for counterexamples.

\begin{table}[tb!]\centering
\begin{tabular}{|c|c|l|l|}\hline
&$M$&Contractions $s_m$ & $d=\dimH(\Gamma)$\\\hline
(I)&4&$s_1(x)=\rho x + (0,1-\rho)$, $s_2(x)=\rho x+(1-\rho,1-\rho)$,& $\frac{\log4}{\log3}\approx 1.26$\\
&&$s_3(x)=\rho x$, $s_4(x)=\rho x + (1-\rho,0)$ with $\rho=1/3$. & \\\hline
(II)&4& $s_1(x)=\rho x$, $s_2(x)=\rho x+(1-\rho,0)$,& $\frac{\log4}{\log \frac{1}{0.41}}\approx 1.55$\\ &&
$s_3(x)=\rho x+\big(\frac12(1-\rho),\frac{\sqrt3}2(1-\rho)\big)$,&\\ && $s_4(x)=\rho x+\big(\frac12(1-\rho),\frac1{2\sqrt3}(1-\rho)\big)$, 
with $\rho=0.41$.& \\\hline
(III)&5& $s_1,\ldots,s_4$ as in (I) and $s_5(x)=\frac1{27}x+(\frac4{27},\frac4{27})$.
& $\frac{\log z}{\log3}\approx 1.28$,
\\
&&&
$z^3-4z^2-1=0$
\\\hline
(IV)&5& 
 $s_1,\ldots,s_4$ as in (I) and $s_5(x)=\rho x+(\rho,\rho)$.
 & $\frac{\log 5}{\log3}\approx 1.46$
\\\hline
\end{tabular}
\caption{The contractions corresponding to the IFS attractors in Figure \ref{fig:disjointness}.}
\label{t:disjointness}
\end{table}

\subsection{Vector index notation}
\label{sec:VectorIndex}
Our quadrature rules will be based on \red{partitioning} $\Gamma$ into \red{self-similar subsets}, via the IFS structure. To describe these \red{subsets} we adopt the vector index notation used in \cite{Jonsson98} \red{(cf.\ also \cite[Section 2.1]{hutchinson1981fractals})}. For $\ell\in \Nnotsn$ let $I_\ell:=\{1,\ldots,M\}^\ell$. Then for $E\subset\Rnotsn^n$ let $E_0:=E$, and, for $\bm=(m_1,\ldots,m_\ell)\in I_\ell$, let
\begin{equation*}%
E_{\bm} := s_{\bm}(E),\qquad s_{\bm}:=s_{m_1}\circ\ldots\circ s_{m_\ell}.
\end{equation*}
For an illustration of this notation in the case of the middle-third Cantor dust see Figure \ref{fig:CantorDustNotation}. 

\begin{figure}[t!]
\centering
\includegraphics[height=80mm]{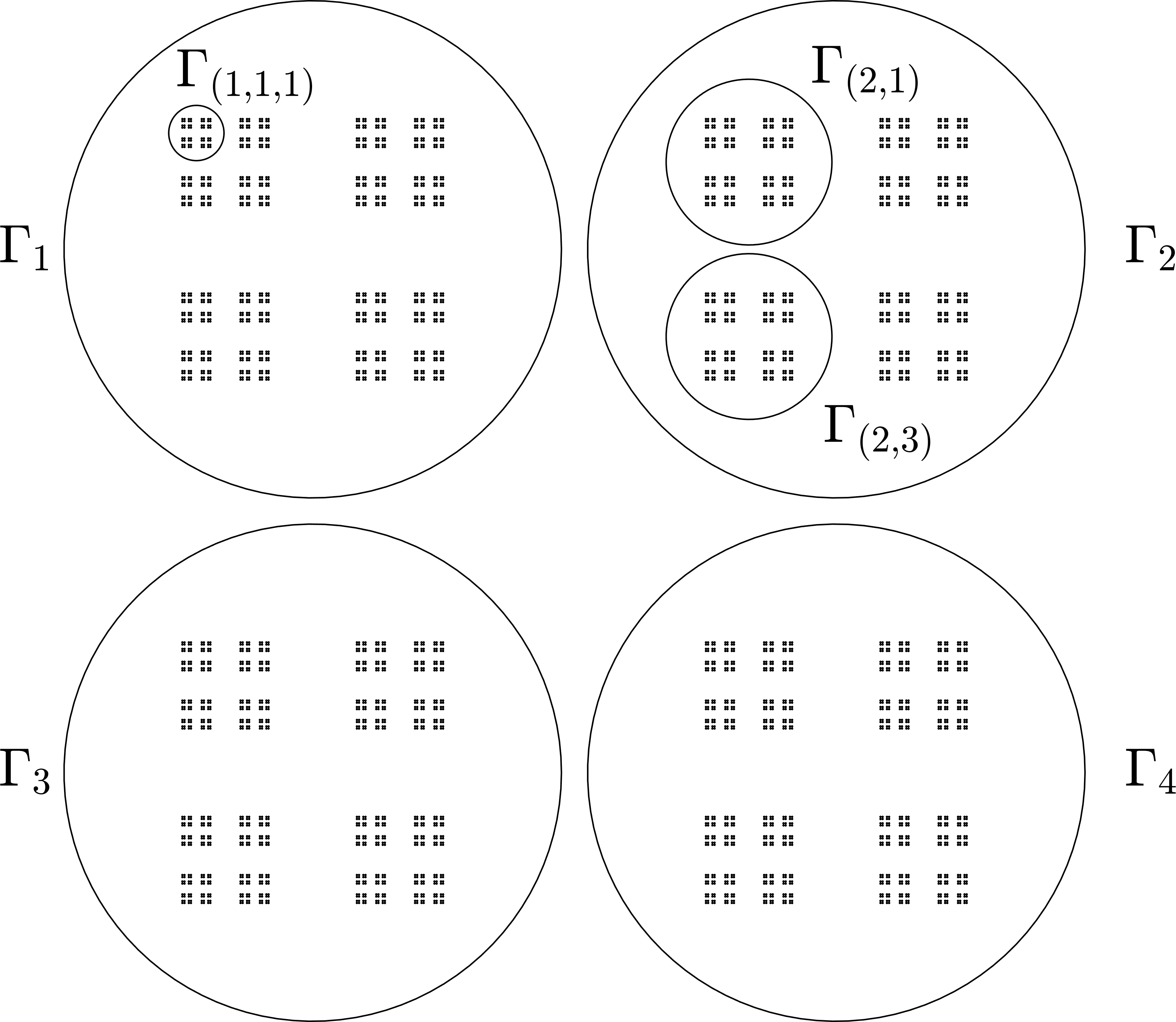}
\caption{
Illustration of the vector index notation of Subsection \ref{sec:VectorIndex}, in the case where $\Gamma\subset\Rnotsn^2$ is the middle-third Cantor dust (example (I) in Figure \ref{fig:disjointness} and Table \ref{t:disjointness}). 
The \red{subset}s $\Gamma_1$, $\Gamma_2$, $\Gamma_3$, $\Gamma_4$, $\Gamma_{(2,1)}$, $\Gamma_{(2,3)}$ and $\Gamma_{(1,1,1)}$ are circled. 
}
\label{fig:CantorDustNotation}
\end{figure}

For $\bm\in I_\ell$, $\Gamma_{\bm}$ is itself the attractor of an IFS, namely $\{s_{\bm} \circ s_m\circ (s_{\bm})^{-1}\}_{m=1}^M$, and 
\begin{align*}
\diam(\Gamma_{\bm}) = \bigg(\prod_{i=1}^\ell \rho_{m_i}\bigg)\diam(\Gamma).
\end{align*}
This implies that the ``elements'' of the ``Hausdorff-BEM'' proposed in \cite{HausdorffBEM} are themselves IFS attractors, so the quadrature rules developed here can be used in the implementation of that method.

When considering singular integrands it will be important to estimate the distance between \red{subsets} of $\Gamma$. To that end, given $\bm\neq \bn\in \cup_{\ell\in\Nnotsn}I_\ell$ we define \red{(cf.\ the partial ordering in \cite[Section 2.1]{hutchinson1981fractals})}
\begin{equation}\label{eq:lstar}
\ell_*(\bm,\bn):=\min\big\{\ell\in\{1,\ldots,\min\!\big(\!\dim(\bm),\dim(\bn)\big)\}: m_\ell\neq n_\ell\big\}\in \Nnotsn,
\end{equation}
and note the following obvious result, concerning how many indices share a common value of $\ell_*$.
\begin{lem}\label{lem:Ik_reorder}
Let $\ell\in\Nnotsn$. Given $\bm\in I_\ell$ and $\nu\in\{1,\ldots,\ell\}$, there are $M^{\ell-\nu}(M-1)$ indices $\bn\in I_\ell$ such that $\ell_*(\bm,\bn)=\nu$.
\end{lem}
We shall also use the fact that if $\Gamma$ is hull-disjoint (in the sense of \eqref{eq:RHulldef}) then for $\bm\neq \bn\in \cup_{\ell\in\Nnotsn}I_\ell$
\begin{equation}\label{eq:rprod_ineq}
\dist\big(\Hull(\Gamma_{\bm}),\Hull(\Gamma_{\bn})\big)
\geq R_{\Gamma,\Hull}\prod_{i=1}^{\ell_*(\bm,\bn)-1} \rho_{m_i}. 
\end{equation}

\subsection{\red{Invariant measures and }scaling properties}
\label{sec:scaling}
Integrals over IFS attractors have certain scaling properties that will be central to our analysis. 
\red{In the case of the Hausdorff measure $\mu=\cH^d|_\Gamma$,} by e.g.\ \cite[(3.3)]{Jonsson98} we note that for $\ell\in\Nnotsn$, $\bm=(m_1,\ldots,m_\ell)\in I_\ell$, and for any $\red{\cH^d|_\Gamma}$-measurable function $f$,
\begin{align}
\int_{\Gamma_{\bm}}f(x)\,\rd\red{\cH^d}(x) 
&= \left(\frac{\diam{\Gamma_{\bm}}}{\diam\Gamma}\right)^d\int_\Gamma f\big(s_{\bm}(x')\big) \,\rd\red{\cH^d}(x')\nonumber\\
&=\bigg(\prod_{i=1}^\ell \rho_{m_i}^d \bigg)\int_\Gamma f\big(s_{\bm}(x')\big) \,\rd\red{\cH^d}(x').\label{eq:cov}
\end{align}
In particular, \red{taking $f\equiv 1$ gives}  
\begin{equation}\label{eq:sim_measure}
\quad \cH^d(\Gamma_{\bm})
= \left(\frac{\diam{\Gamma_{\bm}}}{\diam\Gamma}\right)^d \cH^d(\Gamma)
=\left(\prod_{i=1}^\ell \rho_{m_i}^d\right) \cH^d(\Gamma) .
\end{equation}

\red{The measure $\cH^d|_\Gamma$ is just one member of a general class of finite measures on $\Gamma$ for which similar scaling results apply. Given a collection $(p_1,\ldots,p_M)$ of positive weights (or ``probabilities'') satisfying 
\begin{align}
\label{eq:weightsum}
0<p_m<1, \quad m=1,\ldots,M, \qquad \text{and} \qquad \sum_{m=1}^M p_m =1,
\end{align}
there exists (see, e.g., \cite[Sections 4 \& 5]{hutchinson1981fractals}) a Borel regular finite measure $\mu$ supported on $\Gamma$, unique up to normalisation, called an ``invariant'' \cite{hutchinson1981fractals} (also known as ``balanced'' \cite{barnsley1985iterated} or ``self-similar'' 
\cite{moran1998singularity}) measure associated to $\Gamma$ and $(p_1,\ldots,p_M)$, 
such that $\mu(A)=\sum_{m=1}^M p_m \mu(s_m^{-1}(A))$ for every measurable set $A\subset \mathbb{R}^n$. 
By \cite[Thm.~2.1]{moran1998singularity} 
the OSC implies that $\mu(s_m(\Gamma)\cap s_{m'}(\Gamma))=0$ for each $m\neq m'$, and as a consequence we find that for $\ell\in\Nnotsn$, $\bm=(m_1,\ldots,m_\ell)\in I_\ell$, and any $\mu$-measurable function $f$,
\begin{align}
\int_{\Gamma_{\bm}}f(x)\,\rd\mu(x) 
=\bigg(\prod_{i=1}^\ell p_{m_i} \bigg)\int_\Gamma f\big(s_{\bm}(x')\big) \,\rd\mu(x')\label{eq:cov_general}
\end{align}
and 
\begin{equation}\label{eq:sim_measure_general}
\quad \mu(\Gamma_{\bm})
= \bigg(\prod_{i=1}^\ell p_{m_i}\bigg)\mu(\Gamma).
\end{equation}
We shall assume henceforth that $\mu$ is an invariant measure on $\Gamma$ in this sense, for some collection of associated weights $(p_1,\ldots,p_M)$. 
The case $\mu=\cH^d|_\Gamma$ corresponds to choosing $p_m=\rho_m^d$ for $m=1,\ldots,M$, with \eqref{eq:weightsum} holding by \eqref{eq:dfirst}.
}

\subsection{\red{Partitioning} \texorpdfstring{$\Gamma$}{Gamma}}
\label{sec:Splitting}
To define our composite quadrature rules we need to specify an index set $\cI\subset\cup_{\ell\in\Nnotsn}I_\ell$ such that
\begin{align}
\label{eq:index}
\Gamma = \bigcup_{\bm\in \cI}\Gamma_{\bm}\quad &\text{ and } \quad \red{\mu}(\Gamma_{\bm}\cap \Gamma_{\bm'})=0 \text{ for }\bm\neq \bm'\in \cI,\nonumber\\
\quad &\text{so that }\sum_{\bm\in \cI} \red{\mu}(\Gamma_{\bm})= \red{\mu}(\Gamma).
\end{align}
One approach is to choose \red{subsets} with a fixed level of refinement, i.e.\ to take $\cI=I_\ell$ for some fixed $\ell\in\Nnotsn$. However, in the case of a non-\red{homogeneous} IFS, \red{subsets} chosen in this way may differ significantly in size. 
An alternative approach is to choose \red{subsets} with approximately equal diameter, taking $\cI=L_h(\Gamma)$ for some fixed $h>0$, where
\begin{align}
L_{h}(\Gamma) := \big\{&\bm=(m_1,\ldots,m_\ell)\in\cup_{\ell'\in\Nnotsn}I_{\ell'} :\nonumber\\ &\diam(\Gamma_{\bm})\leq h \text{ and } \diam(\Gamma_{(m_1,\ldots,m_{\ell-1})})>h\big\},\label{eq:Lh_def}
\end{align}
with $\Gamma_{(m_1,.\ldots,m_{\ell-1})}$ replaced by $\Gamma$ when $\ell=1$ and $L_h(\Gamma):=\{0\}$ when $h\geq \diam(\Gamma)$ \red{(recall our convention that $\Gamma_0=\Gamma$)}. 
See Figures~\ref{fig:Koch_decomp} and \ref{fig:Lh} for illustrations of the decomposition $L_h\GG$ for the Koch snowflake and a non-\red{homogeneous} Cantor set. 
If $\Gamma$ is \red{homogeneous} then for $0<h\leq \diam(\Gamma)$ we have $L_h(\Gamma)=I_\ell$, where 
\begin{align}
\label{eq:rho_ell}
\rho^\ell\diam(\Gamma)\leq h < \rho^{\ell-1}\diam(\Gamma), \,\,\text{ i.e. }\,\,\ell = \left\lceil \frac{\log(h/\diam(\Gamma))}{\log{\rho}} \right\rceil.
\end{align}

\begin{figure}[t!]
\centering
\subfigure[$h=0.4$, $|L_h(\Gamma)|=55$\label{subfig:Koch0p4}]{%
      \includegraphics[width=0.47\textwidth,clip,trim=90 30 90 20]{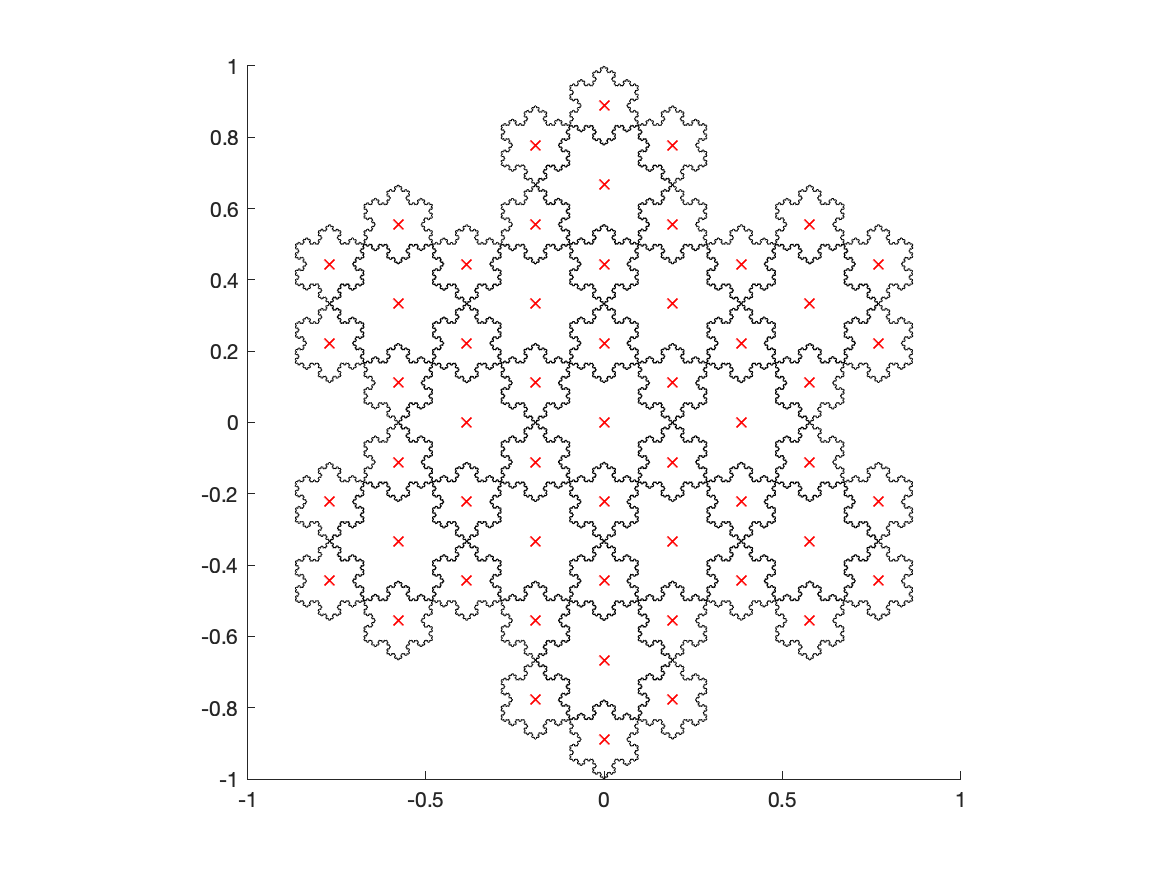}   
    }
    \subfigure[$h=0.3$, $|L_h(\Gamma)|=133$\label{subfig:Koch0p3}]{%
\includegraphics[width=0.47\textwidth,clip,trim=90 30 90 20]{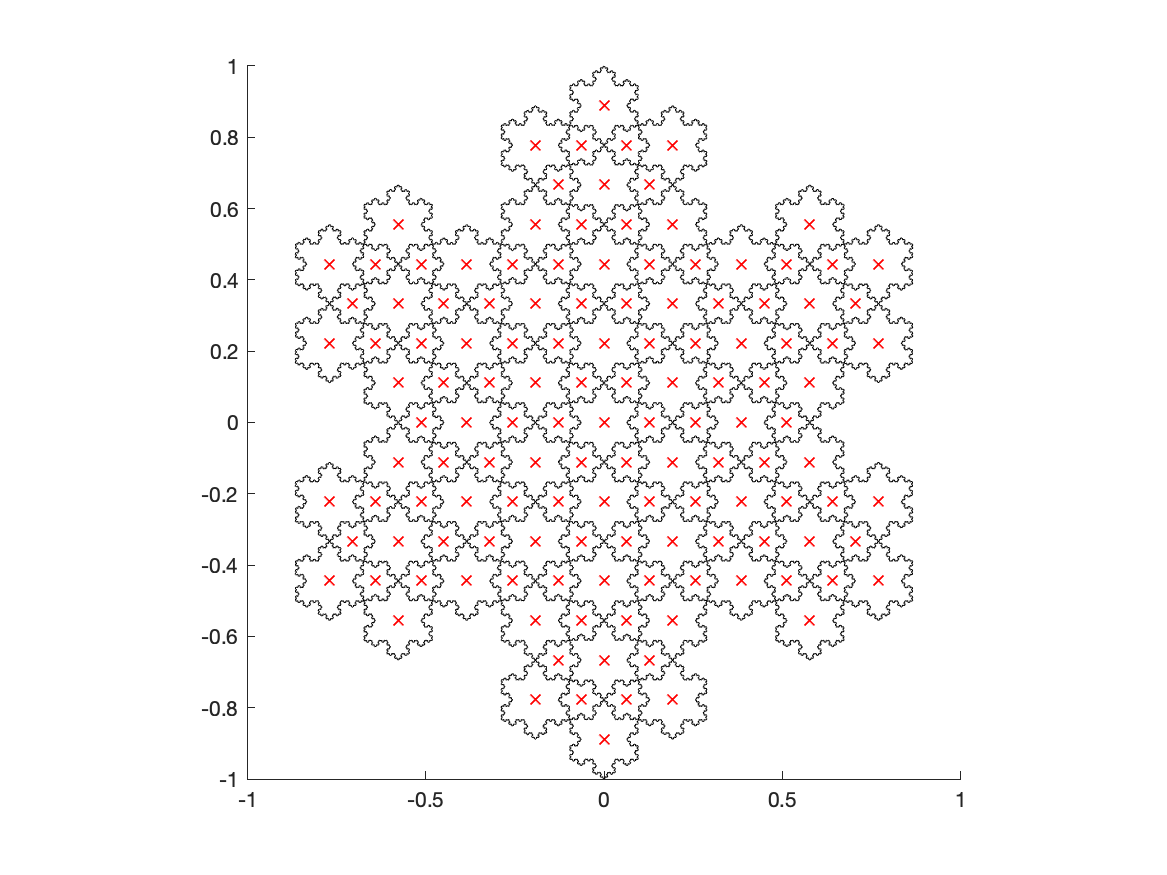}}
\caption{Examples of the \red{partitioning} \eqref{eq:index} corresponding to the index set $\cI=L_h(\Gamma)$ defined by \eqref{eq:Lh_def} in the case where $\Gamma\subset\Rnotsn^2$ is the Koch snowflake. 
The barycentres (defined by \eqref{eq:mid_def}) of the self-similar \red{subset}s in the \red{partitioning} are indicated with red crosses.
The Koch snowflake can be written as a non-\red{homogeneous} IFS attractor with $M=7$, $s_m(x,y)=\frac13(x,y)+\frac23(\cos\alpha_m,\sin\alpha_m)$ with $\alpha_m=\frac{(2m-1)\pi}6$ for $m=1,\ldots,6$, and $s_7(x,y)=(\frac12x-\frac1{2\sqrt3}y,\frac1{2\sqrt3}x+\frac12y )$ \red{(so that $\rho_7=1/\sqrt{3}$).}
It satisfies the OSC \eqref{oscfirst} with $O$ equal to the interior of $\Gamma$.}
\label{fig:Koch_decomp}
\end{figure}

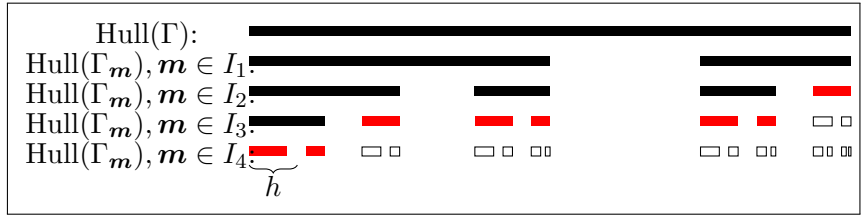
\begin{figure}[tb!]
\centering
\fbox{\begin{tikzpicture}[scale=8]
\def\hh{.015}\def\gap{-.05}
\fill(0,0)rectangle(1,\hh);
\fill(0,\gap)rectangle(1/2,\gap+\hh);
\fill(3/4,\gap)rectangle(1,\gap+\hh);
\fill(0,2*\gap)rectangle(1/4,2*\gap+\hh);
\fill(3/8,2*\gap)rectangle(1/2,2*\gap+\hh);
\fill(3/4,2*\gap)rectangle(3/4+1/8,2*\gap+\hh);
\fill[red](15/16,2*\gap)rectangle(1,2*\gap+\hh);
\fill(0,3*\gap)rectangle(1/8,3*\gap+\hh);
\fill[red](3/16,3*\gap)rectangle(1/4,3*\gap+\hh);
\fill[red](3/8,3*\gap)rectangle(3/8+1/16,3*\gap+\hh);
\fill[red](15/32,3*\gap)rectangle(1/2,3*\gap+\hh);
\fill[red](3/4,3*\gap)rectangle(3/4+1/16,3*\gap+\hh);
\fill[red](3/4+3/32,3*\gap)rectangle(3/4+1/8,3*\gap+\hh);
\draw[](3/4+3/16,3*\gap)rectangle(3/4+3/16+1/32,3*\gap+\hh);
\draw[](63/64,3*\gap)rectangle(1,3*\gap+\hh);
\fill[red](0,4*\gap)rectangle(1/16,4*\gap+\hh);
\fill[red](3/32,4*\gap)rectangle(1/8,4*\gap+\hh);
\draw[](3/16,4*\gap)rectangle(3/16+1/32,4*\gap+\hh);
\draw[](15/64,4*\gap)rectangle(1/4,4*\gap+\hh);
\draw[](3/8,4*\gap)rectangle(3/8+1/32,4*\gap+\hh);
\draw[](3/8+3/64,4*\gap)rectangle(3/8+1/16,4*\gap+\hh);
\draw[](3/8+3/32,4*\gap)rectangle(3/8+3/32+1/64,4*\gap+\hh);
\draw[](63/128,4*\gap)rectangle(1/2,4*\gap+\hh);
\draw[](3/4+0,4*\gap)rectangle(3/4+1/32,4*\gap+\hh);
\draw[](3/4+3/64,4*\gap)rectangle(3/4+1/16,4*\gap+\hh);
\draw[](3/4+3/32,4*\gap)rectangle(3/4+3/32+1/64,4*\gap+\hh);
\draw[](3/4+15/128,4*\gap)rectangle(3/4+1/8,4*\gap+\hh);
\draw[](3/4+3/16,4*\gap)rectangle(3/4+3/16+1/64,4*\gap+\hh);
\draw[](3/4+3/16+3/128,4*\gap)rectangle(3/4+3/16+1/32,4*\gap+\hh);
\draw[](3/4+3/16+3/64,4*\gap)rectangle(3/4+3/16+3/64+1/128,4*\gap+\hh);
\draw[](3/4+63/256,4*\gap)rectangle(3/4+1/4,4*\gap+\hh);
\draw(-.17,0)node{$\Hull(\Gamma)$:};
\draw(-.18,\gap)node{$\Hull(\Gamma_{\bm}),\bm\in I_1$: };
\draw(-.18,2*\gap)node{$\Hull(\Gamma_{\bm}),\bm\in I_2$: };
\draw(-.18,3*\gap)node{$\Hull(\Gamma_{\bm}),\bm\in I_3$: };
\draw(-.18,4*\gap)node{$\Hull(\Gamma_{\bm}),\bm\in I_4$:};
\draw[decorate,decoration={brace,amplitude=3}](.08,4.3*\gap)--(0,4.3*\gap); \draw(.04,5*\gap)node{$h$};
\end{tikzpicture}}
\caption{A depiction of the index set $L_h(\Gamma)$ and of $\UnionHull$ for the non-\red{homogeneous} Cantor set $\Gamma\subset\Rnotsn$ with $s_1(x)=\frac12x$, $s_2(x)=\frac14x+\frac34$.
Each segment represents the convex hull of some $\Gamma_{\bm}$.
Given $h\in[1/16,1/8)$, the multi-indices $\bm\in L_h\GG$ are those corresponding to the red segments: these are the \red{segment}s of length $\le h$ whose parent \red{segment} has length $>h$.
The set $\UnionHull$ is the union of the red segments.
}
\label{fig:Lh} 
\end{figure}

\section{Barycentre rule for regular integrals}\label{sec:smooth_quad}

We now present our composite barycentre rules for the evaluation of regular integrals of the form \eqref{eq:single} and \eqref{eq:double}. 

\subsection{Single integrals}
\label{sec:single}
We first consider the single integral \eqref{eq:single}. Given a partitioning \eqref{eq:index} of $\Gamma$ into self-similar \red{subsets}, our quadrature nodes are the barycentres of the \red{subsets} (for an illustration in the case of the Koch snowflake see Figure \ref{fig:Koch_decomp}), computed with respect to the measure $\red{\mu}$, and the weights are the measures of the \red{subsets}.

\begin{defn}[Barycentre rule for single integrals]
\label{def:single}
Let $\Gamma$ \red{and $\mu$} be as in Subsections \ref{sec:IFS} \red{and \ref{sec:scaling}}, let $\cI\subset\cup_{\ell\in\Nnotsn}I_\ell$ be an index set satisfying \eqref{eq:index}, and let $f:\Gamma\to\Cnotsn$ be continuous. 
Then for the approximation of the integral 
\begin{equation*}%
I_\Gamma[f]:=\int_\Gamma f(x)\,\rd\red{\mu}(x)
\end{equation*}
we define the \textit{barycentre rule}
\begin{equation}
\label{eq:BaryQuad}
Q_\Gamma[f]:=\sum_{\bm\in \cI}w_{\bm}f(x_{\bm}),
\end{equation}
where, for $\bm\red{=(m_1,\ldots,m_l)}\in \cI$, 
\begin{equation}\label{eq:mid_def}
x_{\bm}  :=\frac{\int_{\Gamma_{\bm}}x\,\rd\red{\mu}(x)}{\int_{\Gamma_{\bm}}\rd\red{\mu}(x)}
=\frac{\int_{\Gamma_{\bm}}x\,\rd\red{\mu}(x)}{\red{\mu}(\Gamma_{\bm})}
\end{equation}
and
\begin{equation}\label{eq:weights_def}
w_{\bm}:=\red{\mu}(\Gamma_{\bm}) 
.
\end{equation}
The number of weights and nodes in this approximation is $|\cI|$.
\end{defn}

\begin{rem}
While it always holds that $x_{\bm}\in\Hull(\Gamma_{\bm})$ (by the supporting hyperplane theorem), it does not in general hold that $x_{\bm}\in\Gamma_{\bm}$ (the middle-third Cantor set provides a counterexample). 
\end{rem}

\red{The weights $w_{\bm}$ in \eqref{eq:weights_def} can be computed using  \eqref{eq:sim_measure_general} 
as
\begin{align}
\label{eq:weights_general}
w_{\bm}= \left(\prod_{i=1}^{\ell}\red{p}_{m_i}\right)\red{\mu}(\Gamma).
\end{align}
}
While the barycentres $x_{\bm}$ are defined a priori in terms of integrals with respect to $\red{\mu}$, the following result shows how they can be computed using only information about the similarities $(s_1,\ldots,s_M)$. \red{This result coincides with the first step in the recursive procedure described in \cite[2.5.2]{kunze2011fractal} and \cite[Section 2]{Ma:96} for the calculation of moments of invariant measures.} 

\begin{prop}\label{th:barychuckle}
The barycentres $x_{\bm}$ defined by \eqref{eq:mid_def} can be evaluated as
\begin{align}
\label{eq:xm_formula}
x_{\bm} = s_{\bm}(x_\Gamma), \quad \bm\in \cI,
\end{align}
where $x_\Gamma:=\frac{\int_{\Gamma}x\,\rd\red{\mu}(x)}{\int_{\Gamma}\rd\red{\mu}(x)}=\int_{\Gamma}x\,\rd\red{\mu}(x)$ is the barycentre of $\Gamma$, which can be evaluated as
\begin{align}
\label{eq:xGamma_formula}
x_\Gamma = \Bigg(I - \sum_{m=1}^M\red{p_m}\rho_mA_{m}\Bigg)^{-1}\bigg(\sum_{m=1}^M\red{p_m}\delta_m\bigg),
\end{align}
where $I$ is the $n\times n$ identity matrix and $\rho_m$, $A_m$ and $\delta_m$, $m=1,\ldots,M$, are as in \eqref{eq:affine}.
\end{prop}

\begin{proof}
To prove \eqref{eq:xm_formula} we first consider the case where $s_{\bm}=s_m$ for some $m\in\{1,\ldots,M\}$, for which by \eqref{eq:affine}, \eqref{eq:cov_general} and \eqref{eq:sim_measure_general} 
we have 
	\begin{align}
	x_m 
&\red{:=\frac{\int_{\Gamma_{m}}x\,\rd\mu(x)}{\int_{\Gamma_{m}}\rd\mu(x)}	}
	= 
	\frac{\red{p_m}\int_{\Gamma} s_m(x) \,\rd\red{\mu}(x)}{\red{p_m}\int_{\Gamma}\rd \red{\mu}(x)}
	= \frac{ \rho_mA_m\int_{\Gamma}x \,\rd\red{\mu}(x)}{\int_{\Gamma}\rd \red{\mu}(x)} + \frac{\delta_m\int_{\Gamma}  \,\rd\red{\mu}(x)}{\int_{\Gamma}\rd \red{\mu}(x)}\nonumber\\
	&= \rho_mA_m x_\Gamma + \delta_m 
	= s_m(x_\Gamma).\label{eq:barymap}
	\end{align}
The general result follows by induction on the length of the vector index $\bm$.

To prove \eqref{eq:xGamma_formula} we note that, by \eqref{eq:mid_def} and \eqref{eq:sim_measure_general},
	\[
	x_\Gamma= \frac{\sum_{m=1}^M\int_{\Gamma_m}x\,\rd\red{\mu}(x)}{\red{\mu}(\Gamma)}= \sum_{m=1}^M\frac{\red{\mu}(\Gamma_{m})x_m}{\red{\mu}(\Gamma)} = \sum_{m=1}^M\red{p_m}x_m = \sum_{m=1}^M \red{p_m} s_m(x_\Gamma).
	\]
Using \eqref{eq:barymap}, this can be written as
\begin{align*}
\bigg(I- \sum_{m=1}^M\red{p_m}\rho_mA_m\bigg) x_\Gamma = \sum_{m=1}^M\red{p_m}\delta_m,
\end{align*}
from which \eqref{eq:xGamma_formula} follows by matrix inversion. The invertibility of the matrix $\big(I- \sum_{m=1}^M\red{p_m}\rho_mA_m\big)$ follows from the fact that $\|A_m\|_2=1$ and $0<\rho_m<1$ for each $m=1,\ldots,M$, so that, recalling \eqref{eq:weightsum},
\[\left\|\sum_{m=1}^M\red{p_m}\rho_mA_m\right\|_2\leq \sum_{m=1}^M\red{p_m}\rho_m\|A_m \|_2 = \sum_{m=1}^M\red{p_m}\rho_m<\sum_{m=1}^M\red{p_m} = 1.\]
\end{proof}

\begin{rem}
\label{rem:HdGamma}
Evaluation of the quadrature weights $w_{\bm}$ defined in \eqref{eq:weights_def} requires knowledge of $\red{\mu}(\Gamma)$. If $\red{\mu}(\Gamma)$ is unknown then the quadrature rule \eqref{eq:BaryQuad} can only be evaluated up to the unknown factor $\red{\mu}(\Gamma)$. As mentioned in Section \ref{sec:intro}, \red{even in the special case $\mu=\cH^d|_\Gamma$} the exact value of $\cH^d(\Gamma)$ is known only in certain special cases, and, to our knowledge, only for examples where $d\leq 1$. In the Hausdorff BEM application that motivates this paper, this is unproblematic as 
one can simply work with an appropriately normalised measure (see the discussion in Section \ref{sec:intro} and \cite{HausdorffBEM}).
However, for completeness we comment briefly on the current state of knowledge regarding $\cH^d(\Gamma)$.
The best-studied examples are Cantor-type sets in $\Rnotsn$ ($n=1$). Important early work in this area includes that of Marion \cite{marion1986mesure,marion1987mesures} and Falconer \cite{Fal85}, where it was proved that $\cH^d(\Gamma)=1$ for a large class of Cantor-type sets including the classical Cantor sets defined by \eqref{eq:CantorDef} \cite[Thm.~1.14--1.15]{Fal85}; for more recent related results see e.g.\ \cite{AySt:99} and \cite{Zuberman2019}.
For $n>1$ it appears that the exact value of the Hausdorff measure of even the simplest IFS attractors is known
only for $d\le1$, see e.g.\ \cite{XiongZhou2005}, where it is proved that for the Cantor dust defined in Example (I) of Table \ref{t:disjointness}, $\cH^d(\Gamma)=2^{d/2}$ for $0<\rho\leq 1/4$ ($0<d\leq 1$) \cite[Cor.~1]{XiongZhou2005}; see also the earlier paper \cite{zhou1999hausdorff} where the case $\rho=1/4$ was considered.
For $n>1$ and $d>1$ it appears that only approximate results are available. For instance, when $\Gamma\subset\Rnotsn^2$ is the Sierpi\'{n}ski triangle, it is known that $0.77\leq\cH^d(\Gamma)\leq0.81794$ \cite{Mo:09}. 
One complication in the case $n>1$ is that even if $\Gamma$ is a Cartesian product of lower-dimensional IFS attractors, as is the case \am{\st{e.g.\ }} for the Cantor dust in Example (I) of Table \ref{t:disjointness}, the measure of $\Gamma$ cannot be computed as the product of lower-dimensional measures, since for sets $\Gamma_1$ and $\Gamma_2$ of dimension $d_1$ and $d_2$ respectively, in general we do not have $\cH^{d_1}(\Gamma_1)\times\cH^{d_2}(\Gamma_2) = \cH^{{d_1} + {d_2}}(\Gamma_1\times\Gamma_2)$ \cite[Proposition~7.1]{Fal}. 
\end{rem}

\subsection{Double integrals}
\label{sec:double}
Double integrals of the form \eqref{eq:double} can be treated by iterating the barycentre rule in the obvious way. 

\begin{defn}[Barycentre rule for double integrals]
\label{def:double}
Let $\Gamma\subset \Rnotsn^n$ and $\Gamma'\subset \Rnotsn^{n'}$ be as in Subsection \ref{sec:IFS} \red{(possibly with different Hausdorff dimensions), 
and let $\mu$ and $\mu'$ be invariant measures on $\Gamma$ and $\Gamma'$ respectively, as in Subsection \ref{sec:scaling}.}  
Let $\cI$ and $\cI'$ be index sets satisfying \eqref{eq:index} for $\red{(\Gamma,\mu)}$ and $\red{(\Gamma',\mu')}$ respectively, and let $f:\Gamma\times\Gamma'\to\Cnotsn$ be continuous. 
Then for the approximation of the iterated integral 
\begin{equation*}%
I_{\Gamma,\Gamma'}[f]:=\int_\Gamma\int_{\Gamma'} f(x,y)\,\rd\red{\mu'}(y)\,\rd\red{\mu}(x)
\end{equation*}
we define the \textit{iterated barycentre rule}
\begin{equation}\label{eq:QGG}
Q_{\Gamma,\Gamma'}[f]:= \sum_{\bm\in \cI} \sum_{\bm'\in \cI'} w_{\bm} w'_{\bm'}f(x_{\bm},x'_{\bm'}),
\end{equation}
where, for $\bm\in\cI$, $x_{\bm}$ and $w_{\bm}$ are defined by \eqref{eq:mid_def} and \eqref{eq:weights_def}, and, for $\bm'\in\cI'$, $x'_{\bm'}$ and $w'_{\bm}$ are defined by the analogous formulas involving $\Gamma'$ \red{and $\mu'$}.
\end{defn}

\subsection{Error estimates}
\label{sec:Errors}

When the integrands are sufficiently smooth, error estimates for the quadrature rules in Definitions \ref{def:single} and \ref{def:double} can be derived by standard Taylor series arguments. The result for single integrals (Definition \ref{def:single}) is presented in Theorem \ref{th:MidLip1} below. 
Before stating the theorem, we introduce some notation. Given a set $E\subset\Rnotsn^n$ and a function $f:E\to\Cnotsn$ we define %
\begin{align*}%
\cL_{0,E}[f]:=\sup_{x\neq y\in E}\frac{|f(x)-f(y)|}{|x-y|}.
\end{align*}
If $f$ is differentiable in an open set $\Omega\supset E$, we denote its gradient by $\nabla f:\Omega\to \Cnotsn^n$, and define
\begin{align*}%
\cL_{1,E}[f]:=\sup_{x\neq y\in E}\frac{|\nabla f(x)-\nabla f(y)|}{|x-y|}.
\end{align*}
Note that we are allowing the possibility that $\cL_{0,E}[f]$ and $\cL_{1,E}[f]$ are infinite. 
If $f$ is twice differentiable in $\Omega$, we denote its Hessian by $H\!f:\Omega\to \Cnotsn^{n\times n}$.
For $\alpha\in \Nnotsn_0^n$, $D^\alpha$ denotes standard multi-index notation for partial derivatives. 
\red{Finally, given $h>0$ we define
\begin{align}
\label{eq:LhHull}
\UnionHull:=\bigcup_{\bm\in L_h(\Gamma)}\Hull(\Gamma_{\bm}).
\end{align}}

\begin{thm}\label{th:MidLip1}
Let $\Gamma$ \red{and $\mu$} be as in Subsections \ref{sec:IFS} \red{and \ref{sec:scaling}}. Let $h>0$, and let $\UnionHull$ be as in \eqref{eq:LhHull}. 
Suppose that $f:\UnionHull\to\Cnotsn$, and let $Q_\Gamma^h$ denote the barycentre rule of Definition \ref{def:single} with $\cI=L_h(\Gamma)$. 
Then, \red{with $E$ denoting $|I_\Gamma[f]-Q_\Gamma^h[f]|$}, 
\begin{enumerate}[(i)]
\item %
$
\red{E}\leq h \red{\mu}(\Gamma)\max_{\bm\in L_h(\Gamma)}\cL_{0,\Hull(\Gamma_{\bm})}[f]$.  
\item If $f$ is differentiable in an open set $\Omega\supset \UnionHull$ then 
\begin{align*}
\red{E}
&\leq h \red{\mu}(\Gamma)\sup_{x\in \UnionHull}|\nabla f(x)|\\
&\leq \sqrt{n}h\red{\mu}(\Gamma)\sup_{x\in \UnionHull}\max_{\substack{\alpha\in\Nnotsn_0^n\\|\alpha|=1}}|D^\alpha f(x)|
\end{align*}
and
\[
\red{E}
\leq h^2 \red{\mu}(\Gamma) \max_{\bm\in L_h(\Gamma)}\cL_{1,\Hull(\Gamma_{\bm})}[f].\]  
\item If $f$ is twice differentiable in an open set $\Omega\supset \UnionHull$ then 
\begin{align*}
\red{E}
&\leq \frac{h^2}{2}\red{\mu}(\Gamma)\sup_{x\in\UnionHull}\|H\!f(x)\|_2\\
&\leq \frac{nh^2}{2}\red{\mu}(\Gamma) \sup_{x\in\UnionHull}\max_{\substack{\alpha\in\Nnotsn_0^n\\|\alpha|=2}}|D^\alpha f(x)|.
\end{align*}
\end{enumerate}
\end{thm}
\begin{proof}
(i) Elementary estimation, combined with \eqref{eq:index}, gives
\begin{align*}
\label{}
|I_\Gamma[f]-Q_\Gamma^h[f]| &\leq \sum_{\bm\in L_h(\Gamma)}\int_{\Gamma_{\bm}}|f(x)-f(x_{\bm})|\,\rd\red{\mu}(x)\\
&\leq \sum_{\bm\in L_h(\Gamma)} \red{\mu}(\Gamma_{\bm}) \cL_{0,\Hull(\Gamma_{\bm})}[f]\diam(\Gamma_{\bm})\notag\\
&\leq h \red{\mu}(\Gamma)\max_{\bm\in L_h(\Gamma)}\cL_{0,\Hull(\Gamma_{\bm})}[f].
\end{align*}

(ii) The first bound follows from part (i) and the fact that, by the mean value theorem, $\cL_{0,\Hull(\Gamma_{\bm})}[f]$ $\leq \sup_{x\in \Hull(\Gamma_{\bm})} |\nabla f(x)|$. 
For the second bound we also apply the mean value theorem, noting for each $\bm\in L_h(\Gamma)$ and $x\in \Gamma_m$ there exists a point $\xi_{\bm}(x)$ on the segment between $x$ and $x_{\bm}$ such that $f(x)-f(x_{\bm})=\nabla f(\xi_{\bm}(x))^T(x-x_{\bm})$. Hence
\begin{align*}
\label{}
\int_{\Gamma_{\bm}} f(x)-f(x_{\bm})\,\rd\red{\mu}(x) 
&= \int_{\Gamma_{\bm}}\nabla f(x_{\bm})^T(x-x_{\bm})\,\rd\red{\mu}(x) \\
& \qquad + \int_{\Gamma_{\bm}}\big(\nabla f(\xi_{\bm}(x))-\nabla f(x_{\bm})\big)^T(x-x_{\bm})\,\rd\red{\mu}(x).
\end{align*}
The first integral on the right-hand side vanishes by the definition of $x_{\bm}$, with the result that, again using \eqref{eq:index}, %
\begin{align*}
\label{}
|I_\Gamma[f]-Q_\Gamma^h[f]| &\leq 
\sum_{\bm\in L_h(\Gamma)} \red{\mu}(\Gamma_{\bm}) \cL_{1,\Hull(\Gamma_{\bm})}[f]\diam(\Gamma_{\bm})^2\\
&\leq h^2 \red{\mu}(\Gamma)\max_{\bm\in L_h(\Gamma)}\cL_{1,\Hull(\Gamma_{\bm})}[f].
\end{align*}

(iii) By Taylor's theorem, for each $\bm\in L_h(\Gamma)$ and $x\in \Gamma_m$ there exists a point $\xi_{\bm}(x)$ on the segment between $x$ and $x_{\bm}$ such that $f(x)-f(x_{\bm})= \nabla f(x_{\bm})^T(x-x_{\bm}) + \frac{1}{2}(x-x_{\bm})^TH\!f(\xi_{\bm}(x))(x-x_{\bm})$. Again, the linear term integrates to zero, so %
\begin{align*}
\label{}
\left|\int_{\Gamma_{\bm}} f(x)-f(x_{\bm})\,\rd\red{\mu}(x)\right| 
\leq \frac{1}{2}\red{\mu}(\Gamma_{\bm}) \sup_{x\in\Hull(\Gamma_{\bm})}\|H\!f(x)\|_2
\diam(\Gamma_{\bm})^2,
\end{align*}
and the result follows, by bounding $\diam(\Gamma_{\bm})\leq h$, summing over $\bm$ and using \eqref{eq:index}.
\end{proof}

We now consider the double integral case (Definition \ref{def:double}). 
Higher order iterated integrals over the product of arbitrarily many IFS attractors can be analysed similarly, but are not considered here.

\begin{thm}\label{th:MidLip2}
Let $\Gamma\subset \Rnotsn^n$ and $\Gamma'\subset \Rnotsn^{n'}$ be as in Subsection \ref{sec:IFS}, \red{
and let $\mu$ and $\mu'$ be invariant measures on $\Gamma$ and $\Gamma'$ respectively, as in Subsection \ref{sec:scaling}.} Let $h>0$, and let $\UnionHull$ be as in \eqref{eq:LhHull} and $\UnionHull'$ be as in \eqref{eq:LhHull} with $\Gamma$ replaced by $\Gamma'$. 
Suppose that $f:\UnionHull\times \UnionHull'\to\Cnotsn$, and let $Q_{\Gamma,\Gamma'}^h$ denote the barycentre rule of Definition \ref{def:double} with $\cI=L_h(\Gamma)$ and $\cI'=L_h(\Gamma')$. 
Then, \red{with $E$ denoting $|I_{\Gamma,\Gamma'}[f]-Q_{\Gamma,\Gamma'}^h[f]|$}: 
\begin{enumerate}[(i)]
\item %
$
\red{E}
\leq \sqrt{2}\:h 
\red{\mu}(\Gamma)\red{\mu'}(\Gamma')
\max_{(\bm,\bm')\in L_h(\Gamma)\times L_h(\Gamma')}
\cL_{0,\Hull(\Gamma_{\bm})\times\Hull(\Gamma'_{\bm'})}[f].
$
\item If $f$ is differentiable in an open set $\Omega\supset \UnionHull\times \UnionHull'$ 
then 
\begin{align*}
\label{}
\red{E}
&\leq \sqrt{2}\:h 
\red{\mu}(\Gamma)\red{\mu'}(\Gamma')
\sup_{(x,x')\in\UnionHull\times \UnionHull'}|\nabla f(x,x')|\\
&\leq \sqrt{2}\sqrt{n+n'}\;h 
\red{\mu}(\Gamma)\red{\mu'}(\Gamma')
\sup_{(x,x')\in\UnionHull\times \UnionHull'}\max_{\substack{\alpha\in\Nnotsn_0^{n+n'}\\|\alpha|=1}}|D^\alpha f(x,x')| 
\end{align*}
and
\begin{align*}
\red{E}
&\leq 2h^2 
\red{\mu}(\Gamma)\red{\mu'}(\Gamma')
\max_{(\bm,\bm')\in L_h(\Gamma)\times L_h(\Gamma')}
\cL_{1,\Hull(\Gamma_{\bm})\times\Hull(\Gamma'_{\bm'})}[f].\end{align*}
\item\label{pt:2D} If $f$ is twice differentiable in an open set $\Omega\supset \UnionHull\times \UnionHull'$ then 
\begin{align*}
\label{}
\red{E}
&\leq h^2 
\red{\mu}(\Gamma)\red{\mu'}(\Gamma')
\sup_{(x,x')\in\UnionHull\times \UnionHull'}\|H\!f(x,x')\|_2 \\
&\leq (n+n')h^2  \sup_{(x,x')\in\UnionHull\times \UnionHull'}\max_{\substack{\alpha\in\Nnotsn_0^{n+n'}\\|\alpha|=2}}|D^\alpha f(x,x')|.
\end{align*}
\end{enumerate}
\end{thm}
\begin{proof}
Follows similar arguments to those used to prove Theorem \ref{th:MidLip1}, in the setting of $\Rnotsn^{n+n'}$. The extra factors of $\sqrt2$ and $2$ arise because $\diam(E\times E')\leq \sqrt{\diam(E)^2+\diam(E')^2}$ for $E\subset\Rnotsn^n$ and $E'\subset\Rnotsn^{n'}$.
\end{proof}

\begin{rem}
\label{rem:midpoint_choice}
The fact that $x_{\bm}$ is the barycentre of $\Gamma_{\bm}$ only enters in the proof of the $O(h^2)$ estimates in Theorems \ref{th:MidLip1} and \ref{th:MidLip2}. The $O(h)$ estimates remain true if the barycentre $x_{\bm}$ is replaced by any other point of $\Hull(\Gamma_{\bm})$. 
\end{rem}

\begin{rem}\label{rem:ErrVsCost}
The error bounds in Theorem~\ref{th:MidLip1} are written in terms of the discretization parameter $h$, which measures the diameter of the portion of $\Gamma$ on which the integrand is approximated by a constant value.
In practice it is useful to estimate the error in terms of the computational effort of the quadrature rule, writing the error bounds in terms of the number of quadrature points (and thus of the evaluations of the integrand) $N:=|L_h\GG|$.

\red{We shall do this in the special case where $\Gamma$ is homogeneous.%
}
Then $L_h\GG=I_\ell$ for $\ell\in\Nnotsn$ as in \eqref{eq:rho_ell}, so that $N=M^\ell$, and from \eqref{eq:rho_ell} and \eqref{eq:dfirst}, which gives $\rho^d=1/M$, we have
\begin{align}
\label{eq:HBound}
h<\rho^{\ell-1}\diam\GG=\frac{\diam\GG}\rho\;\frac1{M^{\ell/d}}=\frac{\diam\GG}\rho\; N^{-1/d}
=\diam\GG\left(\frac{N}{M}\right)^{-1/d}.
\end{align}
We can substitute either of the last two expressions in place of $h$ in the right-hand sides of the error bounds in Theorem~\ref{th:MidLip1}, obtaining $O(N^{-1/d})$ and $O(N^{-2/d})$ convergence rates, with respect to increasing $N$.
We observe that the lower the Hausdorff dimension $d$ of $\Gamma$, the faster the convergence of the quadrature rule with respect to increasing $N$, \red{all other factors being equal.}

In the double-integral case, we can substitute either of the expressions 
\begin{align*} &\min\bigg\{\frac{\diam\GG}\rho N^{-1/d},\frac{\diam(\Gamma')}{\rho'} (N')^{-1/d'}\bigg\} 
\\&= \min\bigg\{\diam\GG \left(\frac{N}{M}\right)^{-1/d},\diam\GG \left(\frac{N'}{M'}\right)^{-1/d'}\bigg\},\end{align*}
where $N':=|L_h(\Gamma')|$, in place of $h$ in the error estimates of Theorem~\ref{th:MidLip2}. In this case, the computational cost, measured as the number of integrand evaluations, is the product $NN'$.
\end{rem}

\red{The barycentre rule \eqref{eq:BaryQuad} depends on the choice of the index set $\cI$.
The choice $\cI=L_h(\Gamma)$ stipulated in Section~\ref{sec:Errors}, with $L_h(\Gamma)$ as in \eqref{eq:Lh_def}, allocates the quadrature nodes $x_\bm$ only according to the diameter of the associated subsets $\Gamma_\bm\subset\Gamma$, and is motivated by the use of a Taylor-polynomial technique to bound the quadrature error.
This choice is most effective for homogeneous IFSs and Hausdorff measures, where the fraction of the measure associated to each quadrature node is the same. 
For general IFSs and invariant measures, we expect that more sophisticated choices of the index set $\cI$, taking into account the measure of the subsets in the induced partition as well as their diameter, may be more efficient. However, we leave further discussion of this issue to future work.
}

\section{Evaluation of singular integrals}\label{sec:Phi_t}

We now turn to the case where the integrands in \eqref{eq:single} and \eqref{eq:double} are singular. In this section, we show how two classes of singular integrals involving the function $\Phi_t$ (defined in \eqref{eq:Phit}) can be expressed in terms of regular integrals, using the scaling properties from Subsection \ref{sec:scaling} and certain homogeneity properties of $\Phi_t$. 
\red{This will allow us to derive and analyse quadrature rules for the singular integrals based on the barycentre rule described in the previous section.}

\red{\subsection{Homogeneity and bounds on derivatives of \texorpdfstring{$\Phi_t$}{Phi-t}}}
\red{Key to our analysis will be the fact that,} for any $x\neq y\in\Rnotsn^n$ and $\rho>0$,
\begin{equation}\label{eq:PhitObvs}
\Phi_t(\rho x, \rho y)=
\left\{
\begin{array}{cc}
\log\rho + \Phi_t (x,y),&\quad t=0,\\
\rho^{-t}\Phi_t(x,y),&\quad t>0.
\end{array}
\right.
\end{equation}
We note also that $\Phi_t(x,y)$ is smooth as a function of both $x$ and $y$, away from the diagonal $x=y$, on which it is singular. 
More precisely, we shall need the following result concerning derivatives of $\Phi_t$. Here the assumption that the multi-index $\alpha\in \Nnotsn_0^{2n}$ means that $\alpha$ could correspond to differentiation with respect to the components of either $x$ or $y$, 
and
\[a_t:=\begin{cases}2,&t=0,\\t(t+2),&t>0.\end{cases}\] 
\begin{lem}\label{lem:W2bd}
Let $\alpha\in\Nnotsn_0^{2n}$ be a multi-index with $|\alpha|=2$. Then for $x\neq y$
\begin{equation*}%
\left|\rD^{\alpha}\Phi_t(x,y)\right|
\leq
{\dfrac{a_t}{|x-y|^{t+2}}.}
\end{equation*}
\end{lem}
\begin{proof}
Let $\alpha=\beta+\gamma$ for some multi-indices $\beta,\gamma\in\Nnotsn_0^{2n}$
with $|\beta|=|\gamma|=1$. With $r(x,y):=|x-y|$ for $x,y\in\Rnotsn^n$, for any twice differentiable $F:\Rnotsn_+\to\Cnotsn$ we have 
that for $x\neq y$ 
\begin{equation}
\begin{aligned}%
\rD^{\beta}\big(F(r(x,y))\big) &= F'\big(r(x,y)\big)\big(\rD^\beta r(x,y)\big),\\	
\rD^{\alpha}\big(F(r(x,y))\big) &= 
F''\big(r(x,y)\big)\big(\rD^\beta r(x,y)\big)\big(\rD^\gamma r(x,y)\big)+ F'\big(r(x,y)\big)\big(\rD^{\alpha}r(x,y)\big).
\end{aligned}
\label{eq:FDB}
\end{equation}
From $|x_i-y_i|\leq r(x,y)$ and the values of the following partial derivatives for $i,j=1,\ldots,n$:
\begin{align*}
&\frac{\partial r(x,y)}{\partial x_i} = - \frac{\partial r(x,y)}{\partial y_i}= \frac{(x_i-y_i)}{r(x,y)}, 
\\
&\frac{\partial^2 r(x,y)}{\partial x_i\partial x_j} 
=-\frac{\partial^2 r(x,y)}{\partial x_i\partial y_j}
=\frac{-(x_i-y_i)(x_j-y_j)}{r(x,y)^3},
\qquad i\ne j,\\
&\frac{\partial^2 r(x,y)}{\partial x_i^2} = -\frac{\partial^2 r(x,y)}{\partial x_i\partial y_i}   
=\frac{1}{r(x,y)} - \frac{(x_i-y_i)^2}{r(x,y)^3},
\end{align*}
we obtain
$$
|\rD^\beta r(x,y)|\leq1, \qquad
|\rD^\gamma r(x,y)|\leq1, \qquad
|\rD^\alpha r(x,y)|\leq\dfrac1{r(x,y)}.
$$
Inserting these in \eqref{eq:FDB} gives, for all $\alpha,\beta\in\Nnotsn_0^{2n}$ with $|\beta|=1$ and $|\alpha|=2$,
\begin{align}\label{eq:DF}
\begin{split}
\big|\rD^\beta\big(F(r(x,y))\big)\big|&\le\big|F'\big(r(x,y)\big)\big|,\\
\big|\rD^\alpha\big(F(r(x,y))\big)\big|&\le\big|F''\big(r(x,y)\big)\big|+\frac{\big|F'\big(r(x,y)\big)\big|}{r(x,y)}.
\end{split}
\end{align}
Now recall that $\Phi_t(x,y)=\tilde{\Phi}_t(r(x,y))$, where 
	\[
	\tilde\Phi_t(r) = \left\{
	\begin{array}{cc}
	\log r,&t=0,\\
	\dfrac{1}{r^{t}},&t>0.
	\end{array}
	\right.
	\]
Elementary computations show that
	\begin{align}
	\label{eq:Phit_Derivs}
	\tilde\Phi'_t(r) = \left\{
	\begin{array}{cc}
	\dfrac{1}{r},&t=0,\\
	-\dfrac{t}{r^{t+1}},&t>0,
	\end{array}
	\right.
	\qquad\text{and}\qquad
	\tilde\Phi''_t(r) = \left\{
	\begin{array}{cc}
	-\dfrac{1}{r^2},&t=0,\\
	\dfrac{t(t+1)}{r^{t+2}},&t>0.
	\end{array}
	\right.
	\end{align}
Inserting these results into \eqref{eq:DF} (with $F=\tilde{\Phi}_t$) gives the claimed result. 
\end{proof}

\subsection{Single integrals}%
\label{sec:singlesingular}
We consider first the evaluation of the single integral
\begin{align}
\label{eq:singint_2}
I_\Gamma[\Phi_t(\cdot, \eta)] = \int_\Gamma \Phi_t(x,\eta)\,\rd\red{\mu}(x),
\end{align}
for $\eta\in\Rnotsn^n$. 
When $\eta\not\in\Gamma$ the integral is regular, and combining Theorem \ref{th:MidLip1}(iii) and Lemma \ref{lem:W2bd} gives the following error estimate for the barycentre rule. 
\begin{prop}
\label{prop:PhitSingleRegular}
Let $\Gamma$ \red{and $\mu$} be as in Subsections \ref{sec:IFS} \red{and \ref{sec:scaling}}. 
Let $h>0$ and let $\UnionHull$ be as in \eqref{eq:LhHull}. 
Let $\eta\in\Rnotsn^n\setminus \UnionHull$. 
Then 
\begin{align*}
\label{}
\left|I_{\Gamma}[\Phi_t(\cdot,\eta)]
-Q^h_{\Gamma}[\Phi_t(\cdot,\eta)]\right| 
&\leq 
{\frac{na_th^2\red{\mu}(\Gamma)
}{2\dist(\eta,\UnionHull)^{t+2}}.}
\end{align*}
\end{prop}
When $\eta\in\Gamma$ the integral is singular, but convergent \red{for sufficiently small $t$. In the case $\mu=\cH^d|_\Gamma$ we have convergence for $0\leq t<d\leq n$, for any $\eta\in \Gamma$ (see Corollary \ref{cor:HausdorffIntegrability}). For a general invariant measure the situation is more complicated and the integrability threshold depends on $\eta$ (see the discussion in Section \ref{sec:IntegrabilityGeneral}).  
In the special case where $\eta$ is the fixed point of one of the contracting similarities $s_m$ defining $\Gamma$, we have convergence for $0\leq t<t_m$ (see Lemma \ref{lem:IntegrabilityFixedPoint}), where
\begin{align}
\label{eq:tmDef}
t_m:=\frac{\log p_m}{\log\rho_m}.
\end{align}
Furthermore, in this case the singular integral can be written in terms of regular integrals, as the following result shows. 
We remind the reader that if $\mu=\cH^d|_\Gamma$ then $p_m=\rho_m^d$ and $t_m=d$.}
\begin{thm}\label{th:Jsingthm}
Let $\Gamma$ \red{and $\mu$} be as in Subsections \ref{sec:IFS} \red{and \ref{sec:scaling}}. 
Fix $m\in\{1,\ldots,M\}$ and let $\eta_{m}$ denote the fixed point of the contracting similarity $s_{m}$, i.e.\ the unique point $\eta_{m}\in \Gamma$ such that $s_m(\eta_m)=\eta_m$. Suppose that $\eta_m\not\in \Gamma_{m'}$ for any $m'\in\{1,\ldots,M\}$, $m'\neq m$. (This holds, for instance if 
$\Gamma$ is disjoint in the sense of \eqref{eq:Rdef}.) 
Then \red{the singular integral \eqref{eq:singint_2} is finite for $0\leq t<t_m$, where $t_m$ is defined in \eqref{eq:tmDef}, and it} can be represented in terms of regular integrals, as:
\[
I_\Gamma[\Phi_t(\cdot, \eta_{m})]= 
\left\{
\begin{array}{ll}
\displaystyle(1-\red{p_m})^{-1}\Bigg(\red{p_m}\red{\mu}(\Gamma)\log\rho_{m} +\sum_{\substack{m'=1\\m'\neq{m}}}^M
I_{\Gamma_{m'}}[\Phi_t(\cdot,\eta_m)]\Bigg),&\quad t=0,\\
\displaystyle(1-\red{p_m}\rho^{-t}_{m})^{-1}\sum_{\substack{m'=1\\m'\neq{m}}}^M 
I_{\Gamma_{m'}}[\Phi_t(\cdot,\eta_m)],
&\quad 
t\in(0,\red{t_m}).
\end{array}
\right.
\]
\end{thm}

\begin{proof}
\red{The integrability result is proved in Lemma \ref{lem:IntegrabilityFixedPoint}. To prove the claimed decomposition,} we first split the integral, writing
\begin{equation}\label{eq:Jsing_comps}
I_{\Gamma}[\Phi_t(\cdot ,\eta_{m})]= \underbrace{I_{\Gamma_m}[\Phi_t(\cdot ,\eta_{m})]}_{\text{singular integral}}+ \sum_{\substack{m'=1\\m'\neq{m}}}^M 
\underbrace{I_{\Gamma_{m'}}[\Phi_t(\cdot ,\eta_{m})]}_{\text{regular integral}}.
\end{equation}
Focusing on the singular term, by \eqref{eq:cov_general} and the fact that $s_m(\eta_m)=\eta_m$ we can write
\begin{align*}
I_{\Gamma_m}[\Phi_t(\cdot ,\eta_{m})] 
&=\int_{\Gamma_{m}}\Phi_t(x,\eta_{m})\,\rd\red{\mu}(x)\\
&= \red{p_m}\int_{\Gamma}\Phi_t(s_m(x),\eta_{m})\,\rd\red{\mu}(x)\\
&= \red{p_m}\int_{\Gamma}\Phi_t(s_m(x),s_m(\eta_{m}))\,\rd\red{\mu}(x)\\
&= \red{p_m}\int_{\Gamma}\Phi_t(\rho_m A_mx+\delta_m,\rho_m A_m \eta_{m}+\delta_m)\,\rd\red{\mu}(x)\\
&= \red{p_m}\int_{\Gamma}\Phi_t(\rho_mx,\rho_m\eta_{m})\,\rd\red{\mu}(x),
\end{align*}
using the fact that $\Phi_t$ is translation and rotation invariant. 
Then, applying \eqref{eq:PhitObvs} gives 
\begin{align}
I_{\Gamma_m}[\Phi_t(\cdot ,\eta_{m})] =\left\{
\begin{array}{ll}
\red{p_m}\left(\red{\mu}(\Gamma)\log\rho_{m}+I_{\Gamma}[\Phi_t(\cdot ,\eta_{m})]\right),&t=0,\\
\red{p_m}\rho_{m}^{-t}I_{\Gamma}[\Phi_t(\cdot ,\eta_{m})] ,&t\in(0,\red{t_m}).
\end{array}\label{eq:singsing_rescale}
\right.
\end{align}
Substituting \eqref{eq:singsing_rescale} into \eqref{eq:Jsing_comps} and solving for $I_{\Gamma}[\Phi_t(\cdot ,\eta_{m})]$, we obtain the result.
\end{proof}

Theorem \ref{th:Jsingthm} can be combined with any quadrature rule capable of evaluating the regular integrals $I_{\Gamma_{m'}}[\Phi_t(\cdot ,\eta_{m})]$, $m'\neq m$, to produce a quadrature rule for evaluating the singular integral \eqref{eq:singint_2} when $\eta=\eta_m$. In particular, given $h>0$, 
applying the barycentre rule of Definition \ref{def:single} 
with $\Gamma$ replaced by $\Gamma_{m'}$ 
and $\cI=L_h(\Gamma_{m'})$, for each $m'\neq m$,  
produces the following quadrature rule:
\begin{align}
\label{eq:SingleSingQuad}
Q^h_{\Gamma,t,m}
:= 
\left\{
\begin{array}{ll}
\displaystyle(1-\red{p_m})^{-1}\Bigg(\red{p_m}\red{\mu}(\Gamma)\log(\rho_{m}) +\sum_{\substack{m'=1\\m'\neq{m}}}^M
Q^h_{\Gamma_{m'}}[\Phi_t(\cdot, \eta_{m})]
\Bigg),&t=0,\\
\displaystyle(1-\red{p_m}\rho^{-t}_{m})^{-1}\sum_{\substack{m'=1\\m'\neq{m}}}^M 
Q^h_{\Gamma_{m'}}[\Phi_t(\cdot, \eta_{m})],&\hspace{-.5cm}t\in(0,\red{t_m}).
\end{array}
\right.
\end{align}

This quadrature formula could be used for the implementation of a collocation-type discretisation of the integral equations in \cite{HausdorffBEM}, with the collocation nodes chosen as fixed points of the \red{self-similar subsets} of the IFS attractor \red{used as the BEM elements} in \cite{HausdorffBEM}.
This discretisation is not investigated in \cite{HausdorffBEM}.

\begin{cor}\label{cor:Jsingconvthm}
Let $\Gamma$, $m$, $\eta_m$ and $\red{t_m}$ be as in Theorem \ref{th:Jsingthm}. 
Let $0<h<\diam\GG$, and suppose that $\eta_m\not\in\Hull(\Gamma_{\bm'})$ for each $\bm'\in L_h(\Gamma)$ such that $m'_1\neq m$. Then the quadrature rule defined by \eqref{eq:SingleSingQuad} for the integral \eqref{eq:singint_2} with $\eta=\eta_m$ satisfies the error estimate
\[ 
\left|I_\Gamma[\Phi_t(\cdot, \eta_{m})] - 
Q^h_{\Gamma,t,m}
\right|\leq 
{\frac{na_th^2\red{\mu}(\Gamma)}{2(1-\red{p_m}\rho^{-t}_{m})R_{m,h}^{t+2}}, \qquad \red{0\leq t<t_m,}}
\]
where
\begin{align}
\label{eq:RmhDef}
R_{m,h}:=\min_{\substack{\bm'\in L_h(\Gamma)\\ m_1'\neq m}}\dist\big(\eta_m,\Hull(\Gamma_{\bm'})\big)
>0.
\end{align}
\end{cor}
\begin{proof}
For $\red{0\leq t<t_m}$ we have 
from Theorem \ref{th:Jsingthm} and equation \eqref{eq:SingleSingQuad} that
\begin{align*}
\label{}
&\left|I_\Gamma[\Phi_t(\cdot, \eta_{m})] - 
Q^h_{\Gamma,t,m}
\right|  \\
&\qquad\leq 
\displaystyle(1-\red{p_m}\rho^{-t}_{m})^{-1}
\sum_{\substack{m'=1\\m'\neq{m}}}^M 
\left|
I_{\Gamma_{m'}}[\Phi_t(\cdot,\eta_{m})]
-Q^h_{\Gamma_{m'}}[\Phi_t(\cdot,\eta_{m})]
\right|,
\end{align*}
and the result follows by applying Proposition \ref{prop:PhitSingleRegular} to each term in the sum, 
and recalling 
\eqref{eq:index} \red{to see that 
$\sum_{\substack{m'=1\\m'\neq{m}}}^M 
\mu(\Gamma_{m'})\leq \mu(\Gamma)$.
} 
\end{proof}

The separation parameter $R_{m,h}$ introduced in \eqref{eq:RmhDef} is used only in the statement of Corollary~\ref{cor:Jsingconvthm}, and is compared to other related parameters in Remark~\ref{rem:Rvalues}.
The relative error and the behaviour of the bound in the limit $(d-t)\to0$ are analysed for the case of \red{homogeneous} IFSs \red{with $\mu=\cH^d|_\Gamma$} in Subsection \ref{s:ErrVsCostSing}. %

\subsection{Double integrals}\label{sec:nonoverlap}

We now consider the evaluation of the double integral
\begin{align}
I_{\Gamma,\Gamma'}[\Phi_t]=\int_{\Gamma}\int_{\Gamma'}\Phi_t(x,y)\,\rd\red{\mu'}(y)\,\rd\red{\mu}(x).
\label{eq:double_int}
\end{align}
When $\Gamma$ and $\Gamma'$ are disjoint the integral is regular, and combining Theorem \ref{th:MidLip2}(iii) and Lemma \ref{lem:W2bd} gives the following error estimate for the barycentre rule. 
\begin{prop}
\label{prop:PhitDoubleRegular}
\red{Let $\Gamma,\Gamma'\subset\Rnotsn^{n}$ and $\mu,\mu'$ be as in Subsections \ref{sec:IFS} and \ref{sec:scaling}.} 
Let $h>0$ and let $\UnionHull$ be as in \eqref{eq:LhHull} and $\UnionHull'$ be as in \eqref{eq:LhHull} with $\Gamma$ replaced by $\Gamma'$. Suppose that $\UnionHull\cap\UnionHull'=\emptyset$. Then
\begin{align*}
\label{}
\left|I_{\Gamma,\Gamma'}[\Phi_t]
-Q^h_{\Gamma,\Gamma'}[\Phi_t]\right| 
&\leq \frac{2na_th^2\red{\mu}(\Gamma)\red{\mu'}(\Gamma')}{\dist(\UnionHull,\UnionHull')^{t+2}}.
\end{align*}
\end{prop}
When $\Gamma$ and $\Gamma'$ are not disjoint \eqref{eq:double_int} is a singular integral, \red{which converges only for sufficiently small $t$. Suppose for simplicity that $\Gamma'=\Gamma$. Then in the case $\mu=\mu'=\cH^d|_\Gamma$ we have convergence for $0\leq t<d\leq n$ (see Corollary \ref{cor:HausdorffIntegrability}). 
For a more general pair of invariant measures $\mu$ and $\mu'$ on $\Gamma$, with respective (possibly different) weights/probabilities $(p_1,\ldots,p_M)$ and $(p_1',\ldots,p_M')$, if $\Gamma$ is disjoint then the integral converges for $0\leq t<t_*$ (see Lemma \ref{lem:IntegrabilityDouble}), where $t_{*}$ is the unique positive solution of 
\begin{align}
\label{eq:tstardef}
\sum_{m=1}^M p_mp_m' \rho_m^{-t_*}=1.
\end{align}
In the disjoint case the singular integral (when it converges) can be written purely in terms of regular integrals, as the following result shows. This was noted previously for the case of Cantor sets in e.g.~\cite{bessis1987mellin}. 
We remind the reader that if $\mu=\mu'=\cH^d|_\Gamma$ then $p_m=p_m'=\rho_m^d$ and $t_*=d$.}

\begin{thm}\label{th:sing_reform}
Let $\Gamma\subset\Rnotsn^n$ be as in Subsection \ref{sec:IFS}, 
\red{and let $\mu,\mu'$ be as in Subsection \ref{sec:scaling}. }
Suppose that $\Gamma$ is disjoint in the sense of \eqref{eq:Rdef}. 
\red{Then the singular double integral \eqref{eq:double_int} with $\Gamma'=\Gamma$ converges for $0\leq t<t_*$, where $t_*$ is the unique positive solution of \eqref{eq:tstardef}, and it can be represented in terms of regular integrals, as: }
\begin{align*}
&I_{\Gamma,\Gamma}[\Phi_t]\\
&=\left\{
\begin{array}{ll}
\displaystyle \bigg(1-\sum_{m=1}^M\red{p_mp_m'}\bigg)^{-1}
\sum_{m=1}^M\Bigg(\red{p_mp_m'}
\red{\mu}(\Gamma)\red{\mu'}(\Gamma)
\log(\rho_m) + \sum_{\substack{m'=1\\ m'\neq m}}
I_{\Gamma_m,\Gamma_{m'}}[\Phi_t]
\Bigg),\hfill& t=0,\\
\displaystyle\bigg(1-\sum_{m=1}^M\red{p_mp_m'}\rho^{-t}_m\bigg)^{-1}\sum_{m=1}^M\sum_{\substack{m'=1\\ m'\neq m}}^M
I_{\Gamma_m,\Gamma_{m'}}[\Phi_t],
\quad&\hspace{-.5cm} t\in(0,t_*).
\end{array}
\right.
\end{align*}
\end{thm}

\begin{proof}
\red{The integrability result is proved in Lemma \ref{lem:IntegrabilityDouble}. To prove the claimed decomposition,} as in the single integral case we begin by splitting the integral, writing
\begin{align}
I_{\Gamma,\Gamma}[\Phi_t]
=\sum_{m=1}^M
\underbrace{I_{\Gamma_m,\Gamma_m}[\Phi_t]
}_{\text{singular integral}}
+
\sum_{m=1}^M\sum_{\substack{m'=1\\m'\neq m}}^M \underbrace{I_{\Gamma_m,\Gamma_{m'}}[\Phi_t]}_{\text{regular integral}}
\label{eq:sing_smooth_rep}.
\end{align}
By applying \eqref{eq:cov_general} \red{(for both $\mu$ and $\mu'$)}, \eqref{eq:PhitObvs}, and the fact that $\Phi_t$ is translation and rotation invariant, the singular integrals in \eqref{eq:sing_smooth_rep} can be written as
\begin{align*}
I_{\Gamma_m,\Gamma_m}[\Phi_t]
&=\int_{{\Gamma_m}}\int_{{\Gamma_m}}\Phi_t\left(x,y\right)\rd\red{\mu'}(y)\rd\red{\mu}(x)
\\
&={\red{p_mp_m'}}\int_{\Gamma}\int_{\Gamma}\Phi_t\left(s_m(x),s_m(y)\right)\rd\red{\mu'}(y)\rd\red{\mu}(x)\\
&={\red{p_mp_m'}}\int_{\Gamma}\int_{\Gamma}\Phi_t\left(\rho_mx,\rho_my\right)\rd\red{\mu'}(y)\rd\red{\mu}(x)\\
&=
\begin{cases}
\displaystyle
{\red{p_mp_m'}}\left(\red{\mu}(\Gamma)\red{\mu'}(\Gamma)\log(\rho_m)+I_{\Gamma,\Gamma}[\Phi_t]\right),&\quad t=0,\\
{\red{p_mp_m'}\rho_m^{-t}}I_{\Gamma,\Gamma}[\Phi_t],&\quad t\in(0,\red{t_*}).
\end{cases}
\end{align*}
Substituting this expression into \eqref{eq:sing_smooth_rep} and solving for $I_{\Gamma,\Gamma}[\Phi_t]$ gives the claimed result. 
\end{proof}

By combining Theorem \ref{th:sing_reform} with a suitable quadrature rule for evaluating the regular integrals $I_{\Gamma_m,\Gamma_{m'}}[\Phi_t]$, $m'\neq m$, we can obtain a quadrature rule for evaluating the singular integral \eqref{eq:double_int}. In particular, given $h>0$, applying the barycentre rule of Definition \ref{def:double} 
with $\Gamma,\Gamma'$ replaced by $\Gamma_m,\Gamma_{m'}$ and with $\cI=L_h(\Gamma_m)$ and $\cI'=L_h(\Gamma_{m'})$ for each pair $(m,m')$ such that $m'\neq m$ produces the following quadrature rule:
\begin{align}
&Q^h_{\Gamma,\Gamma,t}:=
\nonumber\\&
\left\{
\begin{array}{ll}
 &\displaystyle\bigg(1-\sum_{m=1}^M\red{p_mp_m'}\bigg)^{-1}\sum_{m=1}^M
\Bigg(\red{p_mp_m'}\red{\mu}(\Gamma)\red{\mu'}(\Gamma)\log(\rho_m) + \sum_{\substack{m'=1\\ m'\neq m}}
Q^h_{\Gamma_m,\Gamma_{m'}}[\Phi_t]
\Bigg),\\
& \hfill t=0,\\
\displaystyle&\displaystyle\bigg(1-\sum_{m=1}^M\red{p_mp_m'}\rho^{-t}_m\bigg)^{-1}\sum_{m=1}^M\sum_{\substack{m'=1\\ m'\neq m}}^M
Q^h_{\Gamma_m,\Gamma_{m'}}[\Phi_t],\hfill t\in(0,\red{t_*}).
\end{array}
\right.\label{eq:non_overlap_quad_gen}
\end{align}

\begin{cor}\label{cor:Phi_t}
Let $\Gamma$, \red{$\mu$, $\mu'$} and $t$ be as in Theorem \ref{th:sing_reform}. 
Let $0<h<\diam\GG$, and assume that $\Hull(\Gamma_{\bm})\cap\Hull(\Gamma_{\bm'})=\emptyset$ for all $\bm,\bm'\in L_h(\Gamma)$ such that $m_1\neq m_1'$.
Then the quadrature rule defined by \eqref{eq:non_overlap_quad_gen} for the integral \eqref{eq:double_int} satisfies the error estimate 
\[ 
\left|I_{\Gamma,\Gamma}[\Phi_t] - Q^h_{\Gamma,\Gamma,t}\right| \leq 
{\dfrac{2na_t h^2\red{\mu}(\Gamma)\red{\mu'}(\Gamma)}{\displaystyle\bigg(1-\sum_{m=1}^M\red{p_mp_m'}\rho^{-t}_m\bigg)R_{\Gamma,\Hull,h}^{t+2}},}
\]
where 
\begin{align}
\label{eq:RGammahDef}
R_{\Gamma,\Hull,h}:=\min_{\substack{\bm,\bm'\in L_h(\Gamma)\\m_1\neq m_1'}}
\dist\big(\Hull(\Gamma_{\bm}),\Hull(\Gamma_{\bm'})\big) \geq R_{\Gamma,\Hull}.
\end{align}
\end{cor}
\begin{proof}
The proof is similar to that of Corollary \ref{cor:Jsingconvthm}. 
For $0\leq t<\red{t_*}$ we have
\[ 
\left|I_{\Gamma,\Gamma}[\Phi_t] - 
Q^h_{\Gamma,\Gamma,t}
\right|  = 
\bigg(1-\sum_{m=1}^M\red{p_mp_m'}\rho^{-t}_m\bigg)^{-1}\sum_{m=1}^M\\
\sum_{\substack{m'=1\\m'\neq{m}}}^M 
\left|I_{\Gamma_m,\Gamma_{m'}}[\Phi_t] - Q^h_{\Gamma_m,\Gamma_{m'}}[\Phi_t]\right|,
\]
and the result follows by applying Proposition \ref{prop:PhitDoubleRegular} to each term in the sum.
\end{proof}

\begin{rem}\label{rem:Rvalues}
So far we have introduced four different parameters quantifying the distance between \red{self-similar subsets} of an IFS attractor $\Gamma$:
\begin{itemize}
\item $R_\Gamma$ in \eqref{eq:Rdef}, which measures the minimum distance between level-1 \red{subsets},
\item $R_{\Gamma,\Hull}$ in \eqref{eq:RHulldef}, which measures the minimum distance between the convex hulls of level-1 \red{subsets},
\item $R_{m,h}$ in 
\eqref{eq:RmhDef}, 
which measures the minimum distance between a fixed point $\eta_m$ of $s_m$ and the convex hulls of \red{subsets} of $\Gamma\setminus \Gamma_m$ of diameter approximately $h$,
\item $R_{\Gamma,\Hull,h}$ in \eqref{eq:RGammahDef}, 
which measures the minimum distance between the convex hulls of pairs of \red{subsets}, taken from different level-1 \red{subset}s, of approximate diameter $h$.
\end{itemize}
Recall that $R_{m,h}$ and $R_{\Gamma,\Hull,h}$ are defined only for $0<h<\diam\GG$.
They satisfy the inequalities
$$0\le R_{\Gamma,\Hull}\le R_{\Gamma,\Hull,h}\le R_\Gamma,\qquad
R_{\Gamma,\Hull,h}\le\min_{m=1,\ldots,M} R_{m,h}.$$
In particular, {Corollaries~\ref{cor:Jsingconvthm} and \ref{cor:Phi_t} show how} $R_{m,h}$ and $R_{\Gamma,\Hull,h}$ quantify the expected deterioration of the quadrature accuracy due to the vicinity of the integrand singularity, for single and double integrals 
($I_\Gamma[\Phi_t(\cdot,\eta_m)]$ and $I_{\Gamma,\Gamma}[\Phi_t]$), respectively. 

Table \ref{t:Rvalues} shows the values of these parameters for the four examples in Figure~\ref{fig:disjointness}.
The values of $R_{m,h}$ and $R_{\Gamma,\Hull,h}$ are valid for all sufficiently small $h$ (e.g.\ $R_{5,h}=0$ for (III) and $h\ge\sqrt2/3$). 

We are not aware of any IFS attractor satisfying the open set condition with $R_{m,h}=0$ for some $m$, i.e.\ with $\eta_m\in\Hull(\Gamma_{\bm'})$ for $\bm'\in L_h\GG$ and $m_1'\ne m$.
\end{rem}

\begin{table}[htb]
\centering
\begin{tabular}{|c|l|l|l|l|}\hline
&$R_\Gamma$&$R_{\Gamma,\Hull}$&$\min_{m=1,\ldots,M}R_{m,h}$&$R_{\Gamma,\Hull,h}$\\\hline
(I)&$1-2\rho$&$1-2\rho$&$1-\rho$\hfill($m=1,2,3,4$)&$1-2\rho$\\
(II)&$\ge\frac{1-\rho-3\rho^2}{2\sqrt3}$ 
&$0$&
$\sqrt{\frac13-\rho+\rho^2-\rho^3+\rho^4}$ \;($m=4$)
&$\ge\frac{1-\rho-3\rho^2}{2\sqrt3}$ 
\\
(III)&$\sqrt2/27$&$0$&$5\sqrt2/117$\hfill($m=5$)&$\sqrt2/27$
\\
(IV)&$0$&$0$&$1/(3\sqrt2)$\hfill($m=5$)&$0$\\\hline
\end{tabular}
\caption{The values of the separation parameters described in Remark~\ref{rem:Rvalues} for the IFS attractors of Figure~\ref{fig:disjointness}, for sufficiently small $h$.}
\label{t:Rvalues}
\end{table}

\subsection{Relative errors and dependence on \texorpdfstring{$N$}N for \red{homogeneous} IFSs \red{and Hausdorff measure}}\label{s:ErrVsCostSing}

In this section we show how the error bounds we derived for the singular integrals in Subsection \ref{sec:singlesingular} and Subsection \ref{sec:nonoverlap} can be written in terms of the quantity $N:=|L_h(\Gamma)|$, as was discussed for the regular integrals of Subsection \ref{sec:Errors} in Remark~\ref{rem:ErrVsCost}. As well as allowing us to determine the dependence of our error bounds on the computational cost of the quadrature rules, 
this also allows us to clarify the limiting behaviour of our bounds as $(d-t)\to 0$.  
To this purpose, we now consider relative errors, and restrict our attention in this section to \red{homogeneous} IFSs \red{and the case $\mu=\mu'=\cH^d|_\Gamma$}.

For a \red{homogeneous} IFS, we have $N=|I_\ell|=M^\ell$ for $\ell$ as in \eqref{eq:rho_ell}.
The number of evaluations of the integrand $\Phi_t$ required for the computation of the quadrature formulas is
$(M-1)M^{\ell-1}=\frac{M-1}M N$ for the single-integral formula \eqref{eq:SingleSingQuad} and
$(M-1)M^{2\ell-1}=\frac{M-1}M N^2$ for the double-integral formula \eqref{eq:non_overlap_quad_gen}.

Let us consider first the case $t>0$. From \eqref{eq:AntonioModified} we have the lower bounds:%
$$
\begin{aligned}
&I_\Gamma[\Phi_t(\cdot,\eta)]\ge\frac{\tilde{c}_1d}{d-t}\big(\diam\GG\big)^{d-t},\\
&I_{\Gamma,\Gamma}[\Phi_t]\ge \frac{\tilde{c}_1d}{d-t}\big(\diam\GG\big)^{d-t}\cH^d\GG,
\end{aligned} \qquad \eta\in\Gamma, \quad 0<t<d.
$$
Recall that $\tilde{c}_1>0$, defined in \eqref{eq:dset2}, is an intrinsic parameter of the $d$-set $\Gamma$, independent of its characterization as an IFS attractor.
Then, using that \red{for a homogeneous IFS and the case $\mu=\mu'=\cH^d|_\Gamma$ we have} $\red{\sum_{m=1}^Mp_m^2\rho_m^{-t}=}\sum_{m=1}^M\rho^{2d-t}=%
\rho^{d-t}$ from $M\rho^d=1$ in \eqref{eq:d2}, Corollaries \ref{cor:Jsingconvthm} and \ref{cor:Phi_t} imply the following relative error estimates:
\begin{align}\label{eq:RelErr}
\frac{\left|I_\Gamma[\Phi_t(\cdot, \eta_{m})] - Q^h_{\Gamma,t,m}\right|}
{\big|I_\Gamma[\Phi_t(\cdot,\eta_m)]\big|}
\le \frac{\cE_t}{R_{m,h}^{t+2}},  %
\qquad\quad
\frac{\left|I_{\Gamma,\Gamma}[\Phi_t] - Q^h_{\Gamma,\Gamma,t}\right|}{\big|I_{\Gamma,\Gamma}[\Phi_t]\big|}
\le \frac{4\,\cE_t}{R_{\Gamma,\Hull,h}^{t+2}},
\end{align}
where
$$
\cE_t:=\frac{na_th^2\cH^d(\Gamma)(d-t)}{2\tilde{c}_1d\big(\diam\GG\big)^{d-t}(1-\rho^{d-t})}.
$$
To bound $\cE_t$ we first note that, for $0<z,\rho<1$, we have $\rho^z\le 1-z(1-\rho)$ (by comparison of an affine and a convex function of $z$ that coincide for $z=0$ and $z=1$), 
so that
\begin{align}
\label{eq:bound1}
\frac{d-t}{1-\rho^{d-t}}\le\frac{d-t}{1-1+(d-t)(1-\rho)}=\frac1{1-\rho}. 
\end{align}
Moreover, from $|\log\rho|=\log\frac1\rho\le\frac1\rho-1=\frac{1-\rho}\rho$ we have 
$|\log\rho|(\frac1{1-\rho}-1)=|\log\rho|(\frac\rho{1-\rho})\le1$ and hence
\begin{align}
\label{eq:bound2}
\frac{|\log\rho|}{1-\rho}\le 1+|\log\rho|.
\end{align}
Then, using also $d=\frac{\log M}{|\log\rho|}$, $\rho=M^{-1/d}$ and $h\le\frac{\diam\GG}{\rho\, N^{1/d}}$ (see \eqref{eq:HBound}),
we can bound %
\begin{align}
\cE_t&\overset{\eqref{eq:bound1}}\le\frac{na_t\,h^2\, \cH^d(\Gamma)|\log\rho|}{2\tilde{c}_1\log M\big(\diam\GG\big)^{d-t}(1-\rho)} \notag\\
&\overset{\eqref{eq:bound2}}\le\frac{na_t\, h^2\, \cH^d(\Gamma)(1+|\log\rho|)}{2\tilde{c}_1\log M\big(\diam\GG\big)^{d-t}}  \notag\\
&\le\frac{na_t\cH^d(\Gamma)\big(\diam\GG\big)^{2-d+t}\; (1+|\log\rho|)}{2\tilde{c}_1(\log M)\rho^2}  N^{-2/d}\notag\\
&=\frac{na_t\cH^d(\Gamma)\big(\diam\GG\big)^{2-d+t}}{2\tilde{c}_1\log M} \;\Big(1+\frac{\log M}d\Big)\Big(\frac NM\Big)^{-2/d}\notag\\
&\le\frac{na_t\cH^d(\Gamma)\big(\diam\GG\big)^{2-d+t}}{2\tilde{c}_1\log2} \;\Big(1+\frac1d\Big) \Big(\frac NM\Big)^{-2/d}.
\label{eq:EtBound}
\end{align}
Combined with \eqref{eq:RelErr}, this reveals the dependence of the relative error on the computational cost, through the parameter $N$; specifically, the errors are $O(N^{-2/d})$ as $N\to\infty$. 
The above bound can also be written in terms of the refinement level $\ell$ 
using $\frac NM=M^{\ell-1}$, giving exponential convergence at the rate $M^{-2\ell/d}$ as $\ell\to\infty$.

While the bounds on the absolute errors in Corollaries \ref{cor:Jsingconvthm} and \ref{cor:Phi_t} blow up in the limit $t\nearrow d$ (with $N$ fixed), the bounds \eqref{eq:RelErr} and \eqref{eq:EtBound} show that the corresponding relative errors are bounded in this limit, because the integrals being approximated also blow up at the same rate.
Similarly, for a sequence of IFSs with $d\searrow t>0$ (and with $M$ constant, $\tilde{c}_1$ and $R_{\Gamma,\Hull}$ uniformly bounded away from zero, and $\cH^d\GG$ and $\diam\GG$ uniformly bounded above), the absolute errors blow up while the relative errors are uniformly bounded.
Furthermore, the same is true (again for fixed $N$) %
in the case where $d\searrow 0$ and $t\searrow 0$ with $d>t$, since the algebraic growth of the $\frac1d$ term in \eqref{eq:EtBound} is controlled by the exponential decay of the factor $(\frac NM)^{-2/d}$ (provided $\ell\ge2$).

In the case $t=0$ the logarithmic function $\Phi_0$ changes sign, so it is not in general possible to bound its integrals from below.
Thus we assume that $\diam\GG\le1$.
Under this assumption, for all $\eta\in\Gamma$, \eqref{eq:AntonioModified} gives
\begin{align*}
\big|I_\Gamma[\Phi_0(\cdot,\eta)&]\big|
=-I_\Gamma[\Phi_0(\cdot,\eta)]\\
&\ge \tilde{c}_1 d\int_0^{\diam\GG}-r^{d-1}\log r\,\rd r= \tilde{c}_1\big(\diam\GG\big)^d\Big(\frac1d-\log \diam\GG\Big),\\
\big|I_{\Gamma,\Gamma}[\Phi_0]\big|
&=-I_{\Gamma,\Gamma}[\Phi_0]
\ge \tilde{c}_1\big(\diam\GG\big)^d\Big(\frac1d-\log\diam\GG\Big)\cH^d\GG.
\end{align*}
Proceeding as above and using $a_0=2$, $\log\diam\GG\le0$, and $\frac1{1-\rho^d}=\frac{M}{M-1}\le2$, the bound \eqref{eq:RelErr} on the relative error extends to the case $t=0$ with
\begin{align*}
\cE_0&:=
\frac{n a_0 h^2\cH^d\GG}{2(1-\rho^d) \, \tilde{c}_1\big(\diam\GG\big)^d(\frac1d-\log\diam\GG)}\\
&\le
\frac{2 n \cH^d\GG\big(\diam\GG\big)^{2-d}}{\tilde{c}_1} d \bigg(\frac NM\bigg)^{-2/d}, 
\end{align*}
which tends to zero as $d\searrow t= 0$.

Regarding sequences of IFS attractors for which $d\searrow 0$, one can show for example that for any $\epsilon>0$ the family of Cantor sets in $\Rnotsn$ defined by \eqref {eq:CantorDef}, for $\rho\in (0,1/2-\epsilon)$, i.e.\ for $d=\log2/\log(1/\rho)\in(0,\log2/\log(1/(1/2-\epsilon)))\subset(0,1)$, have $\diam\GG=\cH^d\GG=1$ and $\tilde{c}_1$ and $R_{\Gamma,\Hull}$ uniformly bounded away from zero.

\section{Application to Galerkin Hausdorff BEM for acoustic scattering}
\label{sec:Phi}

We now apply our previous results to derive and analyse quadrature rules for the evaluation of 
\begin{align}
I_{\Gamma,\Gamma'}[\Phi]=\int_{\Gamma}\int_{\Gamma'}\Phi(x,y)\,\rd\cH^{d'}(y)\,\rd\cH^d(x),
\label{eq:double_int_Phi}
\end{align}
where $\Gamma,\Gamma'\subset\Rnotsn^{n}$ are as in Subsection \ref{sec:IFS} and 
$\Phi(x,y)$ is the fundamental solution of the Helmholtz equation in $\Rnotsn^{n+1}$, 
defined in \eqref{eq:Helmker2d}. 
\red{As \eqref{eq:double_int_Phi} suggests, our focus on this section is on the case $\mu=\cH^d|_\Gamma$, $\mu'=\cH^{d'}|_\Gamma$.} 
As explained in Section \ref{sec:intro}, integrals of the form \eqref{eq:double_int_Phi} arise as the elements of the Galerkin matrix in the ``Hausdorff BEM'' described in \cite{HausdorffBEM}, for acoustic scattering by fractal screens.
We first consider \eqref{eq:double_int_Phi} in the non-singular case where $\Gamma$ and $\Gamma'$ are disjoint, corresponding to the off-diagonal matrix entries in \cite{HausdorffBEM}. Our quadrature rule in this case is the composite barycentre rule, and the main result is Proposition \ref{prop:PhiDoubleRegular}. We then consider \eqref{eq:double_int_Phi} in the singular case where $\Gamma=\Gamma'$, corresponding to the diagonal matrix entries in \cite{HausdorffBEM}. 
Our quadrature rule for this case is defined in \eqref{eq:reduced} and \eqref{eq:PhiQuad}, and the main result is Theorem \ref{thm:PhiDoubleSingular}.

Before proceeding with the analysis we note the following regularity estimate on $\Phi$. Here and henceforth $a\lesssim b$ means $a\leq C b$ for some constant $C>0$, independent of $\Gamma$, $\Gamma'$, $h$ and $k$, which may change from occurrence to occurrence. 
\begin{lem}\label{lem:Phi_D2}
For all $\alpha\in \Nnotsn_0^{2n}$ with $|\alpha|=2$ and all $x\neq y$
\begin{equation*}%
\left|\rD^{\alpha}\Phi(x,y)\right|
\lesssim
\dfrac{(1+(k|x-y|)^{n/2+1})}{|x-y|^{n+1}}.
\end{equation*}
\end{lem}
\begin{proof}
The proof is analogous to that of Lemma \ref{lem:W2bd}. We first note that $\Phi(x,y)=\tilde{\Phi}(r(x,y))$, where
	\[
	\tilde\Phi(r) = \left\{
	\begin{array}{cc}
	\dfrac{\ri}{4}H_0^{(1)}(kr),&n=1,\\
	\dfrac{\re^{\ri kr}}{4\pi r},&n=2,
	\end{array}
	\right.
	\]
and by standard calculations (e.g.\ \cite[10.6.2--3]{DLMF}) we find that
\begin{align}
\label{eq:Phi_Derivs}
	\tilde\Phi'(r) = \left\{
	\begin{array}{cc}
	-\dfrac{\ri k}{4}H_1^{(1)}(kr),\\
	\dfrac{(\ri k r-1)\re^{\ri kr}}{4\pi r^2},%
	\end{array}
	\right.
\quad\text{and }
	\tilde\Phi''(r) = \left\{
	\begin{array}{cc}
	\dfrac{\ri k^2}{4}\left(\dfrac{1}{kr}H_1^{(1)}(kr)-H_0^{(1)}(kr)\right),&n=1,\\
	\dfrac{(2-2\ri kr -(kr)^2)\re^{\ri kr}}{4\pi r^3},&n=2.
	\end{array}
	\right.
\end{align}
Inserting these results into \eqref{eq:DF} (with $F=\tilde{\Phi}$) gives the claimed result, after application of %
the following standard bounds (which follow from results in \cite[Section~10]{DLMF}): %
\begin{equation}
 |H_0^{{(1)}}(z)|\lesssim\left\{
	\begin{array}{cc}
	(1+|\log z|),&0<z\leq 1,\\
	z^{-1/2},&z>1,
	\end{array}
	\right.
\qquad 	
	|H_1^{{(1)}}(z)|\lesssim\left\{
	\begin{array}{cc}
z^{-1},&0<z\leq 1,\\
	z^{-1/2},&z>1.
	\end{array}
	\right.
\label{eq:HankelBounds}
\end{equation}
\end{proof}

We now consider \eqref{eq:double_int_Phi} in the non-singular case where $\Gamma$ and $\Gamma'$ are disjoint. The following result follows from Theorem \ref{th:MidLip2}(iii) and Lemma \ref{lem:Phi_D2}, and the fact that $(1+(kz)^{n/2+1})/z^{n+1}$ is a positive decreasing function on $(0,\infty)$ for $n=1,2$. 
\begin{prop}
\label{prop:PhiDoubleRegular}
Let $\Gamma\subset\Rnotsn^n$ and $\Gamma'\subset\Rnotsn^{n}$ be as in Subsection \ref{sec:IFS}. 
Let $h>0$, and let $\UnionHull$ be as in \eqref{eq:LhHull} and $\UnionHull'$ be as in \eqref{eq:LhHull} with $\Gamma$ replaced by $\Gamma'$. Suppose that $\UnionHull\cap\UnionHull'=\emptyset$. %
Then
\begin{align*}
\label{}
\left|I_{\Gamma,\Gamma'}[\Phi]
-Q^h_{\Gamma,\Gamma'}[\Phi]\right| 
&\lesssim h^2\cH^{d}(\Gamma)\cH^{d'}(\Gamma')
\dfrac{\left(1+(k\delta)^{n/2+1}\right)}
{\delta^{n+1}}, \qquad n=1,2,
\end{align*}
where $\delta=\dist(\UnionHull,\UnionHull')$.
\end{prop}

We now turn to the singular case where $\Gamma=\Gamma'$ and \eqref{eq:double_int_Phi} becomes
\begin{align}
\label{eq:double_int_Phi_singular}
I_{\Gamma,\Gamma}[\Phi]= \int_{\Gamma}\int_{\Gamma}\Phi(x,y)\,\rd\cH^{d}(y)\,\rd\cH^{d}(x).
\end{align}
For fixed $k>0$, we have (by \cite[(10.8.2)]{DLMF} in the case $n=1$) that
\begin{align}
\label{eq:SingBehaviour}
\Phi(x,y)\sim C_n\Phi_{n-1}(x,y), \qquad |x-y|\to 0, \quad n=1,2,
\end{align}
where 
\[ C_1=-1/(2\pi) \text{ and }C_2=1/(4\pi).\]
Furthermore, the function 
\[\Phi_* := \Phi-C_n\Phi_{n-1}\] 
is continuous across $x=y$ (in fact, Lipschitz continuous), with (see \cite[10.8.2]{DLMF} for the case $n=1$)
\begin{align*}%
\lim_{x\to y}\Phi_*(x,y) = 
\begin{cases}
\dfrac{- \log \left(k/2\right)- \gamma }{2 \pi }+\dfrac{\ri}{4},& n=1,\\[3mm]
\dfrac{\ri k}{4\pi}, &n=2,
\end{cases}
\qquad y\in \Rnotsn^n.
\end{align*}
This motivates a singularity-subtraction approach for evaluating \eqref{eq:double_int_Phi_singular}, using the splitting 
\begin{align}
\label{eq:Splitting}
\Phi = C_n\Phi_{n-1} + \Phi_*,
\end{align}
integrating $\Phi_{n-1}$ using the quadrature rules  from Section \ref{sec:Phi_t}, and $\Phi_*$ using the composite barycentre rule. 
However, the application of \eqref{eq:Splitting} is complicated by the fact that while \eqref{eq:Splitting} is designed to deal efficiently with the singular behaviour \eqref{eq:SingBehaviour}, it is not well adapted to the oscillatory behaviour of $\Phi$ as $k|x-y|\to \infty$. 
To deal with this systematically, we introduce a parameter $\cosc>0$, proportional to the maximum number of wavelengths there can be across $\Gamma$ for us to consider the integral non-oscillatory. We then define our quadrature rule for \eqref{eq:double_int_Phi_singular} differently depending on whether $k\diam(\Gamma)\leq \cosc$ or $k\diam(\Gamma)> \cosc$. Note that the non-oscillatory regime $k\diam(\Gamma)\leq \cosc$ is the one relevant for the BEM application in \cite{HausdorffBEM}, since the diameter of the BEM elements in \cite{HausdorffBEM} needs to be small compared to the wavelength in order to achieve acceptable approximation error.

\begin{defn}[Singularity-subtraction quadrature rule $Q^h_{\Gamma,\Gamma,\Phi}$] 
When $k\diam(\Gamma)\leq \cosc$ we apply \eqref{eq:Splitting} directly to \eqref{eq:double_int_Phi_singular},
approximating
\begin{equation}\label{eq:reduced}
I_{\Gamma,\Gamma}[\Phi]\approx Q^h_{\Gamma,\Gamma,\Phi}:= C_n Q^h_{\Gamma,\Gamma,n-1} + Q^h_{\Gamma,\Gamma}[\Phi_*],
	\end{equation}
where $Q^h_{\Gamma,\Gamma,n-1}$ is defined as in \eqref{eq:non_overlap_quad_gen}. 

When $k\diam(\Gamma)>\cosc$ we first partition $\Gamma$ into \red{subset}s of diameter at most $h_*:=\cosc/k$ before applying \eqref{eq:Splitting} only to the singular terms in the resulting decomposition of \eqref{eq:double_int_Phi_singular}, approximating
\begin{align}
I_{\Gamma,\Gamma}[\Phi]\approx Q^h_{\Gamma,\Gamma,\Phi}%
&:=
 \sum_{\bm\in L_{h_*}(\Gamma)} \bigg( C_n
Q^h_{\Gamma_{\bm},\Gamma_{\bm},n-1}
+ Q^h_{\Gamma_{\bm},\Gamma_{\bm}}[\Phi_*]\bigg)\nonumber
 \\&\qquad+ \sum_{\bm\in L_{h_*}(\Gamma)}\sum_{\substack{\bm'\in L_{h_*}(\Gamma)\\ \bm'\neq\bm}} Q^h_{\Gamma_{\bm},\Gamma_{\bm'}}[\Phi],\label{eq:PhiQuad}
\end{align}
where $Q^h_{\Gamma_{\bm},\Gamma_{\bm},n-1}$ is defined as in \eqref{eq:non_overlap_quad_gen} with $\Gamma$ replaced by $\Gamma_{\bm}$. 
\end{defn}

For the error analysis of \eqref{eq:reduced}, we recall that an error estimate for $Q^h_{\Gamma,\Gamma,n-1}$ was presented in Corollary~\ref{cor:Phi_t}, so it remains to derive an error estimate for $Q^h_{\Gamma,\Gamma}[\Phi_*]$. Naively applying Theorem~\ref{th:MidLip1} would result in an $O(h)$ estimate for $Q^h_{\Gamma,\Gamma}[\Phi_*]$, because while $\Phi_*$ is Lipschitz continuous across $x=y$, its derivative is not Lipschitz. An $O(h^2)$ estimate (matching that for $Q^h_{\Gamma,\Gamma,n-1}$ provided by Corollary~\ref{cor:Phi_t}) can be obtained via a first-principles analysis, which we present in Proposition~\ref{prop:PhiStarDoubleSingular}. 
We then apply this result to give a full error analysis of both \eqref{eq:reduced} and \eqref{eq:PhiQuad} in Theorem~\ref{thm:PhiDoubleSingular}. 
Our analysis is restricted to the case of a \red{homogeneous} IFS, but we expect that with further non-trivial work a similar analysis could be carried out for the non-\red{homogeneous} case - see the discussion before Theorem \ref{thm:PhiDoubleSingularNonUniform}, where a weaker $O(h)$ estimate is proved for the non-\red{homogeneous} case.

Our arguments will make use of the following bounds on the second-order derivatives of $\Phi_*$. 
\begin{lem}\label{lem:PhiStar_D2}
For all $\alpha\in \Nnotsn_0^{2n}$ with $|\alpha|=2$ and all $x\neq y$
\begin{equation*}%
\left|\rD^{\alpha}\Phi_*(x,y)\right|
\lesssim
\begin{cases}
k^2\big(1+\big|\log (k|x-y|)\big|\big), &n=1, \quad 0<k|x-y|\leq 1, \\
\dfrac{k^2}{(k|x-y|)^{1/2}}, &n=1, \quad |x-y|>1,\\
\dfrac{k^2}{|x-y|}, &n=2.
\end{cases}
\end{equation*}
\end{lem}
\begin{proof}
The proof is analogous to that of Lemmas \ref{lem:W2bd} and \ref{lem:Phi_D2}. We first note that $\Phi_*(x,y)=\tilde{\Phi}_*(r(x,y))$, where $\tilde{\Phi}_*:=\tilde{\Phi}-C_n\tilde{\Phi}_{n-1}$, so that combining \eqref{eq:Phit_Derivs} and \eqref{eq:Phi_Derivs} gives
\[
	\tilde\Phi_*'(r) = \left\{	\begin{array}{cc}
	\!\!-\dfrac{\ri k}{4}\left(H_1^{(1)}(kr)+\dfrac{2\ri}{\pi kr}\right),&n=1,%
	\\
	\dfrac{(\ri k r-1)\re^{\ri kr}+1}{4\pi r^2},&n=2,%
	\end{array}	\right.%
\]
and
\[
\tilde\Phi_*''(r) = \left\{
\begin{array}{cc}
\!\dfrac{\ri k^2}{4}\left(\dfrac{1}{kr}H_1^{(1)}(kr)-H_0^{(1)}(kr)+\dfrac{2\ri}{\pi (kr)^2}\right),&n=1,\\
\dfrac{(2-2\ri kr -(kr)^2)\re^{\ri kr}-2}{4\pi r^3},&n=2.
\end{array}
\right.
\]
Using the series expansions for the exponential and the Hankel functions (see \cite[10.8.1]{DLMF}) and the bounds \eqref{eq:HankelBounds}, one finds that
\begin{align}
\label{eq:PhiStarEstimates1}
	|\tilde\Phi_*'(r)| \lesssim \left\{
	\begin{array}{cl}
	k^2r(1+|\log kr|), %
	\\
	\dfrac{k}{(kr)^{1/2}}, %
	\\
	\dfrac{k^2}{1+kr},%
	\end{array}
	\right.
\quad%
	|\tilde\Phi_*''(r)| \lesssim \left\{
	\begin{array}{cl}
k^2(1+|\log kr|),  &n=1,\,0<kr\leq 1, \\
\dfrac{k^2}{(kr)^{1/2}}, &n=1, \,kr>1,\\
	\dfrac{k^3}{1+kr},&n=2.
	\end{array}
	\right.
\end{align}
Inserting these bounds into \eqref{eq:DF} (with $F=\tilde{\Phi}_*$) gives the result.
\end{proof}

In what follows we shall also make use of the fact that 
\begin{equation}\label{eq:log4}
\text{if}\quad 0<x_1\leq x_2\leq x_3<\infty\quad \text{then} \quad
|\log x_2| \leq |\log x_1|+ |\log x_3|, 
\end{equation}
\red{and that, in the case $\mu=\cH^d|_\Gamma$, for $\bm\in L_h(\Gamma)$ it follows from \eqref{eq:sim_measure} that 
\begin{align}
\label{eq:HdGammam_Est}
\cH^d(\Gamma_{\bm})
\leq  \left(\frac{h}{\diam(\Gamma)}\right)^d \cH^d(\Gamma).
\end{align}}

We now consider the approximation of $I_{\Gamma,\Gamma}[\Phi_*]$. 
\begin{prop}
\label{prop:PhiStarDoubleSingular}
Let $\Gamma\subset\Rnotsn^n$ be as in Subsection \ref{sec:IFS}, with $d=\dimH(\Gamma)>n-1$. Suppose that $\Gamma$ is \red{homogeneous} in the sense of Subsection \ref{sec:FurtherAssumptions}, with contraction factor $\rho\in(0,1)$. Suppose also that $\Gamma$ is hull-disjoint in the sense of \eqref{eq:RHulldef}. 
Let $k>0$ and suppose that $k\diam(\Gamma)\leq \cosc$. Let $0<h\leq \diam(\Gamma)$. Then 
\begin{align}
\label{eq:PhiStarEstimateHIndep}
\Big|I_{\Gamma,\Gamma}[\Phi_*]
-Q^h_{\Gamma,\Gamma}[\Phi_*]\Big| 
& \lesssim  ck^2h^2\big(\cH^d(\Gamma)\big)^2
\le ck^2\big(\diam(\Gamma)\cH^d(\Gamma)\big)^2\left(\frac{N}{M}\right)^{-2/d},
\end{align}
where the constant implied in $\lesssim$ depends only on $\cosc$ and
\begin{align*}\label{}
c = \begin{cases}
\dfrac{|\log{\rho}|}{\log{M}} +\big|\log(kR_{\Gamma,\Hull})\big|= 
\dfrac1d +\big|\log(kR_{\Gamma,\Hull})\big|, 
& n=1,\\[3mm]
\dfrac{\rho M}{R_{\Gamma,\Hull}(\rho M-1)}=\dfrac{1}{R_{\Gamma,\Hull}(1- M^{1/d-1})}, & n=2.
\end{cases} 
\end{align*}
\end{prop}
\begin{proof}
We first note that
\begin{align*}
\label{}
\left|I_{\Gamma,\Gamma}[\Phi_*]
-Q^h_{\Gamma,\Gamma}[\Phi_*]\right| 
\leq &
\underbrace{ \sum_{\bm\in L_{h}(\Gamma)} \left|I_{\Gamma_{\bm},\Gamma_{\bm}}[\Phi_*]-Q^h_{\Gamma_{\bm},\Gamma_{\bm}}[\Phi_*] \right|}_{=:S_1}\\
& + \underbrace{\sum_{\bm\in L_{h}(\Gamma)}\sum_{\substack{\bm'\in L_{h}(\Gamma)\\ \bm'\neq\bm}} \left|I_{\Gamma_{\bm},\Gamma_{\bm'}}[\Phi_*]-Q^h_{\Gamma_{\bm},\Gamma_{\bm'}}[\Phi_*] \right|}_{=:S_2}.
\end{align*}
For the analysis of $S_1$ we note that for $\bm\in L_h(\Gamma)$
\begin{align}
\Big|I_{\Gamma_{\bm},\Gamma_{\bm}}[\Phi_*]-Q^h_{\Gamma_{\bm},\Gamma_{\bm}}&[\Phi_*] \Big| \nonumber\\
&=
\left|\int_{\Gamma_{\bm}}\int_{\Gamma_{\bm}} \left(\Phi_*(x,y)-\Phi_*(x_{\bm},x_{\bm})\right)\,\rd\cH^d(y)\rd\cH^d(x)\right|\notag\\
&\leq \int_{\Gamma_{\bm}}\int_{\Gamma_{\bm}} \left|\tilde\Phi_*(r(x,y))-\tilde\Phi_*(0)\right|\,\rd\cH^d(y)\rd\cH^d(x).
\label{eq:TildePhiStarEstimate}
\end{align}
By \eqref{eq:PhiStarEstimates1} we have that, when $kr\leq \cosc$,
\[ 
|\tilde\Phi_*(r)-\tilde\Phi_*(0)| \lesssim \begin{cases}
(kr)^2(1+|\log{kr}|), & n=1,\\[3mm]
k^2r, & n=2.
\end{cases} 
\]
Since both $z^2(1+|\log z|)$ and $z$ are increasing functions of $z\in(0,\infty)$, applying a uniform upper bound on the integrand in 
\eqref{eq:TildePhiStarEstimate} (with $r\leq h$), then summing over $\bm$ and using \eqref{eq:index} and \eqref{eq:HdGammam_Est} gives
\begin{align*}
S_1 
&{\le |\tilde\Phi_*(h)-\tilde\Phi_*(0)|\sum_{\bm\in L_{h}(\Gamma)} \cH^d(\Gamma_{\bm})^2}
\\&{\le |\tilde\Phi_*(h)-\tilde\Phi_*(0)|\left(\frac{h}{\diam(\Gamma)}\right)^d\cH^d(\Gamma)
\sum_{\bm\in L_{h}(\Gamma)} \cH^d(\Gamma_{\bm})}
\\&= {|\tilde\Phi_*(h)-\tilde\Phi_*(0)|\left(\frac{h}{\diam(\Gamma)}\right)^d\cH^d(\Gamma)^2}
\\&\lesssim k^2\left(\frac{h}{\diam(\Gamma)}\right)^d\cH^d(\Gamma)^2
\begin{cases}
h^2(1+|\log{kh}|), & n=1,\\[3mm]
h, & n=2.
\end{cases} 
\end{align*}

For the analysis of $S_2$ we note that if $\bm\in L_h(\Gamma)$ then $\UnionHullbm=\Hull(\Gamma_{\bm})$, so that for $\bm\neq \bm'$ we have, by \eqref{eq:rprod_ineq} and the assumption that $\Gamma$ is hull-disjoint, that
\[\dist\big(\UnionHullbm,\UnionHullbmp\big)=\dist\big(\Hull(\Gamma_{\bm}),\Hull(\Gamma_{\bm'})\big)\geq R_{\Gamma,\Hull}\rho^{\ell_*(\bm,\bm')-1},\]
where $\ell_*(\bm,\bm')$ was defined in \eqref{eq:lstar}. 
Now, since $\Gamma$ is \red{homogeneous} we have that $L_h(\Gamma)=I_\ell$, where $\ell$ satisfies \eqref{eq:rho_ell}. Also, $\cH^d(\Gamma_{\bm})=M^{-\ell}\cH^d(\Gamma)$ for $\bm\in L_h(\Gamma)$.

In the case $n=2$, given $\bm\in L_h(\Gamma)=I_\ell$, by Lemma \ref{lem:PhiStar_D2}, Theorem \ref{th:MidLip2}\eqref{pt:2D} and Lemma \ref{lem:Ik_reorder} we have
\begin{align*}
\sum_{\substack{\bm'\in L_{h}(\Gamma)\\ \bm'\neq\bm}} \Big|I_{\Gamma_{\bm},\Gamma_{\bm'}}[\Phi_*]-&Q^h_{\Gamma_{\bm},\Gamma_{\bm'}}[\Phi_*] \Big|
\lesssim \sum_{\substack{\bm'\in L_{h}(\Gamma)\\ \bm'\neq\bm}} \frac{k^2h^2\cH^d(\Gamma_{\bm})\cH^d(\Gamma_{\bm'})}{R_{\Gamma,\Hull}\rho^{\ell_*(\bm,\bm')-1}}\\
&= \frac{k^2h^2\cH^d(\Gamma_{\bm})\cH^d(\Gamma)M^{-\ell}}{R_{\Gamma,\Hull}}
\sum_{\ell_*=1}^{\ell-1}\frac{M^{\ell-\ell_*}(M-1)}{\rho^{\ell_*-1}}\\
&= \frac{k^2h^2\cH^d(\Gamma_{\bm})\cH^d(\Gamma)}{R_{\Gamma,\Hull}} (M-1)\rho \sum_{\ell_*=1}^{\ell-1}(\rho M)^{-\ell_*}\\
&\leq \frac{k^2h^2\cH^d(\Gamma_{\bm})\cH^d(\Gamma)}{R_{\Gamma,\Hull}} \frac{(M-1)\rho}{(\rho M-1)},
\end{align*}
where we used the fact that $M\rho^d=1$ (by \eqref{eq:dfirst}) and $d>n-1=1$ (an assumption of the theorem) to deduce that $\rho M>1$, which means we can take the summation limit to infinity in the geometric series. 
Finally, summing over $\bm$ and bounding $M-1\leq M$ gives
\[S_2\lesssim \frac{k^2h^2\cH^d(\Gamma)^2}{R_{\Gamma,\Hull}} \frac{\rho M}{(\rho M-1)}, \qquad n=2. \]

In the case $n=1$, given $\bm\in L_h(\Gamma)=I_\ell$, by Lemma \ref{lem:PhiStar_D2}, Theorem \ref{th:MidLip2}\eqref{pt:2D} and Lemma \ref{lem:Ik_reorder} we have
\begin{align*}
\label{}
&\sum_{\substack{\bm'\in L_{h}(\Gamma)\\ \bm'\neq\bm}} \left|I_{\Gamma_{\bm},\Gamma_{\bm'}}[\Phi_*]-Q^h_{\Gamma_{\bm},\Gamma_{\bm'}}[\Phi_*] \right|\\
&\quad\lesssim \sum_{\substack{\bm'\in L_{h}(\Gamma)\\ \bm'\neq\bm}} k^2h^2\cH^d(\Gamma_{\bm})\cH^d(\Gamma_{\bm'})\left(1+\left|\log\left(kR_{\Gamma,\Hull}\rho^{\ell_*(\bm,\bm')-1}\right)\right|\right)\\
&\quad= k^2h^2\cH^d(\Gamma_{\bm})\cH^d(\Gamma)M^{-\ell}
\sum_{\ell_*=1}^{\ell-1}M^{\ell-\ell_*}(M-1)\left(1+\left|\log\left(kR_{\Gamma,\Hull}\rho^{\ell_*-1}\right)\right|\right)\\
&\quad 
\leq k^2h^2\cH^d(\Gamma_{\bm})\cH^d(\Gamma)(M-1)
\sum_{\ell_*=1}^{\ell-1}M^{-\ell_*}\left(1+\left|\log\left(kR_{\Gamma,\Hull}\right)\right| + (\ell_*-1)|\log\rho|\right)\\
&\quad 
\leq k^2h^2\cH^d(\Gamma_{\bm})\cH^d(\Gamma)
\left(1+\left|\log\left(kR_{\Gamma,\Hull}\right)\right| + \frac{|\log\rho|}{M-1}\right)
,
\end{align*}
and then summing over $\bm$, and noting that $1\leq 1/d$ and that $\frac{|\log\rho|}{M-1}\leq\frac{|\log\rho|}{\log M}=\frac1d$ by \eqref{eq:d2} and the fact that $M-1>\log M$, gives
\[S_2\lesssim k^2h^2\big(\cH^d(\Gamma)\big)^2
\left(\frac{1}{d}+\left|\log\left(kR_{\Gamma,\Hull}\right)\right|\right), \qquad n=1.
\]

Combining the estimates for $S_1$ and $S_2$ then gives 
\begin{align*}
&\Big|I_{\Gamma,\Gamma}[\Phi_*]
-Q^h_{\Gamma,\Gamma}[\Phi_*]\Big| %
\\ %
&\qquad\lesssim 
k^2h^2\big(\cH^d(\Gamma)\big)^2
\begin{cases}
\dfrac{h^d(1+|\log{kh}|)}{\diam(\Gamma)^d} + \dfrac{1}{d}+\left|\log\left(kR_{\Gamma,\Hull}\right)\right| , & n=1,\\[3mm]
\dfrac{h^{d-1}}{\diam(\Gamma)^d} + \dfrac{\rho M}{R_{\Gamma,\Hull}(\rho M-1)}, & n=2.
\end{cases} 
\end{align*}
The quantity in braces here is $h$-dependent, but can be bounded uniformly in $h$ to give \eqref{eq:PhiStarEstimateHIndep}. 
For $n=1$ this is achieved by noting that $z^d|\log z|\leq 1/(\re d)$ for $z\in(0,1]$, which, together with $h\leq \diam(\Gamma)$, 
implies that 
\begin{align*}
\dfrac{h^d(1+|\log{kh}|)}{\diam(\Gamma)^d} 
&\le \Big(\frac{h}{\diam(\Gamma)}\Big)^d
\Big(1+\big|\log \big(k\diam\GG\big)\big|+\Big|\log \frac h{\diam\GG}\Big|\Big) \\&
\le 1+\big|\log \big(k\diam\GG\big)\big|+\frac1{\re d}\\
&\lesssim \frac{1}{d}+\big|\log \big(kR_{\Gamma,\Hull}\big)\big|,
\end{align*}
where the final inequality holds by \eqref{eq:log4} with $x_1=kR_{\Gamma,\Hull}$, $x_2=k\diam\GG$, and $x_3=\cosc$, again noting that $1\leq 1/d$ in this case.
For $n=2$, since $h\leq \diam(\Gamma)$ and $d>1$ we can bound
\begin{align*}
\label{}
\frac{h^{d-1}}{\diam(\Gamma)^d} = \left(\frac{h}{\diam(\Gamma)}\right)^{d-1}\frac{1}{\diam(\Gamma)}
\leq \frac{1}{\diam(\Gamma)}
\leq \frac{1}{R_{\Gamma,\Hull}}
\leq \dfrac{\rho M}{R_{\Gamma,\Hull}(\rho M-1)},
\end{align*}
and we can then obtain \eqref{eq:PhiStarEstimateHIndep} by recalling that $\rho^d=1/M$. 
Finally, in converting the $O(h^2)$ estimate to one involving $N=|L_h(\Gamma)|$, we recall that 
$h\leq \diam(\Gamma)(N/M)^{-1/d}$ (see \eqref{eq:HBound}).
\end{proof}

\begin{rem}\label{rem:PhiStarBlowUp}
The error estimate of Proposition~\ref{prop:PhiStarDoubleSingular} for $n=2$ blows up as $\rho\searrow\frac1M$, equivalently as $d\searrow1$, assuming $\cH^d(\Gamma)$, $\diam\GG$, $M$ and $N$ are fixed and $R_{\Gamma,\Hull}$ is bounded away from zero, because of the factor $\frac{\rho M}{\rho M-1}=\frac{1}{1-M^{1/d-1}}$.
For example, for a family of Cantor dusts as in Figure~\ref{fig:disjointness}(I), parametrised by $\rho$ and with $M=4$, this corresponds to the limit $\rho\searrow\frac14$.
Differently from the setting in Subsection \ref{s:ErrVsCostSing}, the relative error is also predicted to blow up in this limit, because the integral being approximated in this case is bounded, since $\Phi_*\in L^\infty(\Gamma\times\Gamma)$.
By contrast, for $n=1$ the estimate in Proposition~\ref{prop:PhiStarDoubleSingular} tends to zero as $d\searrow0$, because for $N>M$ the algebraic growth of the $1/d$ term in $c$ is beaten by the exponential decay of the $(N/M)^{-2/d}$ factor.
\end{rem}

Finally, we can state and prove our main result for the approximation of $I_{\Gamma,\Gamma}[\Phi]$. 

\begin{thm}
\label{thm:PhiDoubleSingular}
Let $\Gamma\subset\Rnotsn^n$ be as in Subsection \ref{sec:IFS}, with $d=\dimH(\Gamma)>n-1$. Suppose that $\Gamma$ is \red{homogeneous} in the sense of Subsection \ref{sec:FurtherAssumptions}, with contraction factor $\rho\in(0,1)$. Suppose also that $\Gamma$ is hull-disjoint in the sense of \eqref{eq:RHulldef}. 
Let $k>0$. For the approximation of the integral \eqref{eq:double_int_Phi_singular} define the quadrature rule $Q^h_{\Gamma,\Gamma,\Phi}$ by \eqref{eq:reduced} (with $h\leq \diam(\Gamma)$) if $k\diam(\Gamma)\leq \cosc$ and by \eqref{eq:PhiQuad} (with $h\leq \cosc/k$) if $k\diam(\Gamma)>\cosc$. 
Then 
\begin{align*}
\left|I_{\Gamma,\Gamma}[\Phi]-Q^h_{\Gamma,\Gamma,\Phi}\right| 
&\lesssim c' h^2\big(\cH^d(\Gamma)\big)^2
\le c'\big(\diam\GG\cH^d(\Gamma)\big)^2 \left(\frac{N}{M}\right)^{-2/d},
\end{align*}
where the constant implied in $\lesssim$ depends only on $\cosc$, and 
\begin{align*}%
c'=\left\{
\begin{array}{l}
\dfrac{\left(\frac{\rho M}{\rho M-1}\right)^{n-1}}{\rho^{n+1}R_{\Gamma,\Hull}^{n+1}}
=\dfrac{M^{(n+1)/d}}{(1-M^{1/d-1})^{n-1}R_{\Gamma,\Hull}^{n+1}}
,\qquad\hfill k\diam(\Gamma)\leq \cosc,\\
\bigg(1+\dfrac{\left(\frac{\rho M}{\rho M-1}\right)^{n-1}}{\rho^{n+1}\big(k\diam\GG\big)^d}\bigg)
\bigg(\dfrac{k\diam(\Gamma)}{R_{\Gamma,\Hull}}\bigg)^{n+1} \\
=\bigg(1+\dfrac{M^{(n+1)/d}}{(1-M^{1/d-1})^{n-1}\big(k\diam\GG\big)^d}\bigg)
\bigg(\dfrac{k\diam(\Gamma)}{R_{\Gamma,\Hull}}\bigg)^{n+1}, \hfill \\
\hfill k\diam(\Gamma)> \cosc.
\end{array}\right.
\end{align*}
\end{thm}

\begin{proof}
By redefining $h_*:=\min\{\cosc/k,\diam(\Gamma)\}$, \eqref{eq:reduced} can be viewed as a special case of \eqref{eq:PhiQuad}, since with $h_*=\diam(\Gamma)$ we have $L_{h_*}(\Gamma)=\{0\}$, in which case the first sum in \eqref{eq:PhiQuad} reduces to \eqref{eq:reduced} and the second sum is absent. We therefore present the proof of \eqref{eq:PhiQuad} and specialise to \eqref{eq:reduced} at the end. 

By the triangle inequality and the splitting $\Phi=C_n\Phi_{n-1}+\Phi_*$, we have:
\begin{align*}
\left|I_{\Gamma,\Gamma}[\Phi]-
Q^h_{\Gamma,\Gamma,\Phi}
\right| 
\!\leq &
\!\underbrace{\sum_{\bm\in L_{h_*}(\Gamma)}\!\! C_n\left|I_{\Gamma_{\bm},\Gamma_{\bm}}[\Phi_{n-1}] -
Q^h_{\Gamma_{\bm},\Gamma_{\bm},n-1}
\right|}_{=:T_1}
\\&+\! \underbrace{\sum_{\bm\in L_{h_*}(\Gamma)}\!\left|I_{\Gamma_{\bm},\Gamma_{\bm}}[\Phi_*]-Q^h_{\Gamma_{\bm},\Gamma_{\bm}}[\Phi_*] \right|}_{=:T_2}\\
& + \underbrace{\sum_{\bm\in L_{h_*}(\Gamma)}\sum_{\substack{\bm'\in L_{h_*}(\Gamma)\\ \bm'\neq\bm}} \left|I_{\Gamma_{\bm},\Gamma_{\bm'}}[\Phi]-Q^h_{\Gamma_{\bm},\Gamma_{\bm'}}[\Phi] \right|}_{=:T_3}.
\end{align*}

To bound $T_1$ above, we first note that by Corollary \ref{cor:Phi_t} (with $\Gamma$ replaced by $\Gamma_{\bm}$ and with $t=n-1$ for $n=1,2$) for any $\bm\in L_{h_*}(\Gamma)$ we have
\begin{align*}\label{}
\left|I_{\Gamma_{\bm},\Gamma_{\bm}}[\Phi_{n-1}] - 
Q^h_{\Gamma_{\bm},\Gamma_{\bm},n-1}\right| 
&\lesssim 
h^2\big(\cH^d(\Gamma_{\bm})\big)^2\left(1-M\rho^{2d+1-n}\right)^{-1}\dfrac{1}{R_{\Gamma_{\bm},\Hull}^{n+1}}\\
&=h^2\big(\cH^d(\Gamma_{\bm})\big)^2\left(1-M^{-1}\rho^{1-n}%
\right)^{-1}\dfrac{1}{R_{\Gamma_{\bm},\Hull}^{n+1}}.
\end{align*}
Let $\ell(h_*)$ be given by \eqref{eq:rho_ell} with $h$ replaced by $h_*$. Noting that $\rho^{\ell(h_*)}>\rho h_*/\diam(\Gamma)$ we have that
\begin{equation}\label{eq:RGmRG}
\dfrac{\rho h_*}{\diam(\Gamma)}R_{\Gamma,\Hull}<
\rho^{\ell(h_*)}R_{\Gamma,\Hull}=R_{\Gamma_{\bm},\Hull}\le\dfrac{h_*}{\diam(\Gamma)}R_{\Gamma,\Hull}.
\end{equation}
Hence, using {\eqref{eq:index} with $\cI=L_{h_*}(\Gamma)$ and \eqref{eq:HdGammam_Est} with $h$ replaced by $h_*$,}
\[ 
T_1 \lesssim 
h^2\big(\cH^d(\Gamma)\big)^2\left(\frac{h_*}{\diam(\Gamma)}\right)^d
\frac{M}{M-\rho^{1-n}}
\left(\dfrac{\diam(\Gamma)}{\rho h_*R_{\Gamma,\Hull}}\right)^{n+1}.
\]

For $T_2$, by Proposition \ref{prop:PhiStarDoubleSingular} we have for any $\bm\in L_{h_*}(\Gamma)$ that
\begin{align*}
&\left|I_{\Gamma_{\bm},\Gamma_{\bm}}[\Phi_*]-Q^h_{\Gamma_{\bm},\Gamma_{\bm}}[\Phi_*] \right| 
\\&\qquad 
\lesssim 
k^2h^2\big(\cH^d(\Gamma_{\bm})\big)^2
\begin{cases}
\frac{|\log\rho|}{\log M} +\Big|\log\big(kR_{\Gamma_{\bm},\Hull}\big)\Big|, & n=1,\\[3mm]
\dfrac{\rho M}{R_{\Gamma_{\bm},\Hull}(\rho M -1)}, & n=2.
\end{cases} 
\end{align*}
Using \eqref{eq:RGmRG}, 
and for $n=1$ the facts that $\frac{|\log\rho|}{\log M}\lesssim|\log\rho|\le\frac1\rho$, and that $kR_{\Gamma_{\bm},\Hull}\leq k\diam\Gamma_{\bm}\leq k h_*\leq \cosc$, so that we can apply \eqref{eq:log4} with $x_1=\rho kh_* R_{\Gamma,\Hull}/\diam(\Gamma_{\bm})$, $x_2=kR_{\Gamma_{\bm},\Hull}$ and $x_3=\cosc$, to find that
\[ 
T_2\lesssim 
k^2h^2\big(\cH^d(\Gamma)\big)^2\left(\frac{h_*}{\diam(\Gamma)}\right)^d
\begin{cases}
\dfrac1{\rho}
+\left|\log\left(\dfrac{\rho kh_*R_{\Gamma,\Hull}}{\diam(\Gamma)}\right)\right|, & n=1,\\[3mm]
\dfrac{M\diam(\Gamma)}{h_*R_{\Gamma_{\bm},\Hull}(\rho M -1)},
& n=2.
\end{cases} 
\]

For $T_3$, by Proposition \ref{prop:PhiDoubleRegular} we have for $\bm,\bm'\in L_{h_*}(\Gamma)$ with $\bm\neq\bm'$ that
\begin{align*}
&\left|I_{\Gamma_{\bm},\Gamma_{\bm'}}[\Phi]
-Q^h_{\Gamma_{\bm},\Gamma_{\bm'}}[\Phi]\right| \\
&\qquad\lesssim h^2\cH^d(\Gamma_{\bm})\cH^{d}(\Gamma_{\bm'})
\dfrac{\Big(1+\big(k\dist(\UnionHullbm,\UnionHullbmp)\big)^{n/2+1}\Big)}
{\dist\big(\UnionHullbm,\UnionHullbmp\big)^{n+1}},
\end{align*}
and by \eqref{eq:rprod_ineq} and \eqref{eq:RGmRG} %
it holds that
\begin{align*}\dist\big(\UnionHullbm,\UnionHullbmp\big)
&\geq \dist\big(\Hull(\Gamma_{\bm}),\Hull(\Gamma_{\bm'})\big)\\
&\geq \rho^{\ell(h_*)-1}R_{\Gamma,\Hull}\\
&>\frac{h_*}{\diam(\Gamma)}R_{\Gamma,\Hull},\end{align*}
so that, 
since $(1+(kz)^{n/2+1})/z^{n+1}$ is positive and decreasing on $(0,\infty)$, and  
$\dfrac{kh_*R_{\Gamma,\Hull}}{\diam(\Gamma)}\leq \cosc$,
 \[ T_3 \lesssim 
 h^2\big(\cH^d(\Gamma)\big)^2
{ \dfrac{\left(1+\left(\dfrac{kh_*R_{\Gamma,\Hull}}{\diam(\Gamma)}\right)^{n/2+1}\right)}
 {\left(\dfrac{h_*R_{\Gamma,\Hull}}{\diam(\Gamma)}\right)^{n+1}}
 }
 \lesssim 
 h^2\big(\cH^d(\Gamma)\big)^2
{ \left(\dfrac{\diam(\Gamma)}{h_*R_{\Gamma,\Hull}}\right)^{n+1}}.
  \]
  
Recalling our redefinition of $h_*:=\min\{\cosc/k,\diam(\Gamma)\}$, the result for the case $k\diam(\Gamma)>\cosc$ then follows by combining the estimates for $T_1$, $T_2$ and $T_3$, and the result for the case $k\diam(\Gamma)\leq\cosc$ follows by noting that in that case $T_3$ is absent. 
We describe in more detail the four possible cases.
\begin{itemize}
\item %
For $k\diam(\Gamma)\le\cosc$ (and thus $h_*=\diam\GG$) and $n=1$ we have, since $M/(M-1)\leq 2$,
\begin{align*}
&\frac{T_1+T_2}{h^2\big(\cH^d(\Gamma)\big)^2}\\
&\lesssim
\!\Big(\frac{h_*}{\diam(\Gamma)}\Big)^d
\bigg(
\left(\dfrac{\diam(\Gamma)}{\rho h_*R_{\Gamma,\Hull}}\right)^2
\!\!+\dfrac{k^2}{\rho}+k^2\left|\log\left(\dfrac{\rho kh_*R_{\Gamma,\Hull}}{\diam(\Gamma)}\right)\right|
\!\bigg)
\\
&= \dfrac{1}{\rho^2 R_{\Gamma,\Hull}^2}+\dfrac{k^2}{\rho}
+k^2\big|\log(\rho k R_{\Gamma,\Hull})\big|
\\
&\lesssim \dfrac{1}{\rho^2 R_{\Gamma,\Hull}^2},
\end{align*}
where the final bound follows 
from the fact that $k\le\frac\cosc{\diam\GG}\lesssim\frac1{R_{\Gamma,\Hull}}$, and because if $\rho k R_{\Gamma,\Hull}\leq 1$ then $k^2|\log(\rho k R_{\Gamma,\Hull})|\leq (\rho  R_{\Gamma,\Hull})^{-2}$, since
\begin{align}
\label{eq:LogEpsBound}
|\log\epsilon|<|\epsilon^{-2}|, \qquad 0<\epsilon\leq 1,
\end{align}
and if $\rho k R_{\Gamma,\Hull}>1$ then $k^2|\log(\rho k R_{\Gamma,\Hull})|\leq k^2 |\log(k \diam(\Gamma))| \leq k^2 |\log(\cosc)|\lesssim k^2 \lesssim (\rho  R_{\Gamma,\Hull})^{-2}$.

\item For $k\diam\GG\le\cosc$ and $n=2$ we have, again using that $k\lesssim 1/R_{\Gamma,\Hull}$,
\begin{align*}
&\frac{T_1+T_2}{h^2\big(\cH^d(\Gamma)\big)^2}\\
&\lesssim
\left(\frac{h_*}{\diam(\Gamma)}\right)^d
\bigg(\frac{M}{\rho^2(\rho M-1)}\left(\dfrac{\diam(\Gamma)}{h_*R_{\Gamma,\Hull}}\right)^3
+ \dfrac{k^2M\diam(\Gamma)}{h_*R_{\Gamma,\Hull}(\rho M-1)}\bigg)
\\
&\lesssim
\frac{M}{\rho^2(\rho M-1)R_{\Gamma,\Hull}^3}
+ \dfrac{k^2M}{R_{\Gamma,\Hull}(\rho M-1)}
\lesssim
\frac{M}{\rho^2(\rho M-1)R_{\Gamma,\Hull}^3}.
\end{align*}
\item For $k\diam(\Gamma)>\cosc$ (and thus $h_*=\cosc/k$) and $n=1$,
\begin{align*}
&\frac{T_1+T_2+T_3}{h^2\big(\cH^d(\Gamma)\big)^2}\\
&\lesssim
\left(\frac{h_*}{\diam(\Gamma)}\right)^d
\bigg[
\left(\dfrac{\diam(\Gamma)}{\rho h_*R_{\Gamma,\Hull}}\right)^2
+\dfrac{k^2}{\rho}+k^2\left|\log\left(\dfrac{\rho kh_*R_{\Gamma,\Hull}}{\diam(\Gamma)}\right)\right|\bigg]
\\&\quad+
\left(\dfrac{\diam(\Gamma)}{h_*R_{\Gamma,\Hull}}\right)^2
\\
&\lesssim
\frac1{\big(k\diam\GG\big)^d}
\bigg[\left(\dfrac{k\diam(\Gamma)}{\rho R_{\Gamma,\Hull}}\right)^2
+\dfrac{k^2}{\rho}
+k^2\left|\log\left(\dfrac{\rho R_{\Gamma,\Hull}}{\diam(\Gamma)}\right)\right|\bigg]
\\&\quad+\left(\dfrac{k\diam(\Gamma)}{R_{\Gamma,\Hull}}\right)^2
\\
&\lesssim
\bigg( 1+ \frac1{\rho^2\big(k\diam\GG\big)^d}\bigg)\left(\dfrac{k\diam(\Gamma)}{R_{\Gamma,\Hull}}\right)^2
,\end{align*}
using $R_{\Gamma,\Hull}<\diam\GG$ and \eqref{eq:LogEpsBound} with 
$\epsilon=\frac{\rho R_{\Gamma,\Hull}}{\diam\GG}<1$.
\item %
Finally, for $k\diam(\Gamma)>\cosc$ and
$n=2$, %
\begin{align*}
&\frac{T_1+T_2+T_3}{h^2\big(\cH^d(\Gamma)\big)^2}\\
&\lesssim
\left(\frac{h_*}{\diam(\Gamma)}\right)^d\bigg[
\frac{M}{\rho^2(\rho M-1)}
\left(\dfrac{\diam(\Gamma)}{h_*R_{\Gamma,\Hull}}\right)^3
+\dfrac{k^2M\diam(\Gamma)}{h_*R_{\Gamma,\Hull}(\rho M-1)}\bigg]
\\&\quad+\left(\dfrac{\diam(\Gamma)}{h_*R_{\Gamma,\Hull}}\right)^3
\\
&\lesssim
\frac1{\big(k\diam\GG\big)^d}\bigg[
\frac{M}{\rho^2(\rho M-1)}
\left(\dfrac{k\diam(\Gamma)}{R_{\Gamma,\Hull}}\right)^3
+ \dfrac{k^3M\diam(\Gamma)}{R_{\Gamma,\Hull}(\rho M-1)}\bigg]
\\&\quad+\left(\dfrac{k\diam(\Gamma)}{R_{\Gamma,\Hull}}\right)^3
\end{align*}
and we obtain the assertion again using the fact that $R_{\Gamma,\Hull}<\diam\GG$. 
\end{itemize}
\end{proof}

\begin{rem}[Number of function evaluations]
\label{rem:FunctionEvaluations}
For a \red{homogeneous} IFS, recall that $N=|L_h(\Gamma)|=|I_\ell|=M^\ell$ for $\ell$ as in \eqref{eq:rho_ell}.
A priori, the quadrature rule \eqref{eq:reduced} requires $(M-1)M^{2\ell-1}=\frac{M-1}{M}N^2$ evaluations of $\Phi_{n-1}$ (see Subsection \ref{s:ErrVsCostSing}) and $M^{2\ell}=N^2$ evaluations of $\Phi_*$. 
If $\ell_*$ is defined by \eqref{eq:rho_ell} with $ h$ replaced by $ h_*$, 
i.e.\ $\ell_*$ is the level of the partition whose elements have diameter approximately $h_*=\cosc/k$, 
then the quadrature rule \eqref{eq:PhiQuad} requires $M^{\ell_*}\frac{M-1}{M}M^{2(\ell-\ell_*)}=\frac{M-1}{M}M^{2\ell-\ell_*}=\frac{M-1}{M^{\ell_*+1}}N^{2}$ evaluations of $\Phi_{n-1}$, $M^{\ell_*}M^{2(\ell-\ell_*)}=M^{2\ell-\ell_*}=\frac{1}{M^{\ell_*}}N^2$ evaluations of $\Phi_*$, and $M^{\ell_*}(M^{\ell_*}-1)M^{2(\ell-\ell_*)}=(M^{\ell_*}-1)M^{2\ell-\ell_*}=(1-\frac{1}{M^{\ell_*}})N^{2}$ evaluations of $\Phi$. 
However, the number of function evaluations can be reduced (by a factor of a half in the limit $h\to 0$) by exploiting the symmetry of $\Phi$, $\Phi_*$, $\Phi_{n-1}$, all of which satisfy $f(x,y)=f(y,x)$. 
\end{rem}

\begin{rem}[Limit behaviour for $d\searrow n-1$]\label{rem:BlowUpFinal}
We consider the behaviour of the estimates of The\-o\-rem \ref{thm:PhiDoubleSingular} as $d\searrow n-1$, assuming $\cH^d(\Gamma)$, $\diam\GG$, $M$ and $N$ are fixed, and $R_{\Gamma,\Hull}$ is bounded away from zero.
This limit corresponds to $\rho\searrow0$ for $n=1$ and $\rho\searrow1/M$ for $n=2$.
For $n=1$ the absolute error tends to zero like $O((N/M^2)^{-2/d})$ as $d\searrow 0$, provided $N>M^2$.
Since in this limit the integral is dominated by the contribution from the singular function $\Phi_0$, the integral of which grows like $1/d$ as $d\searrow 0$ (as shown in Subsection \ref{s:ErrVsCostSing}), the relative error tends to zero like $O(d(N/M^2)^{-2/d})$ as $d\searrow 0$. 
For $n=2$ the estimate for the absolute error grows as $d\searrow 1$, being asymptotically proportional to $1/(1-M^{1/d-1})\sim 1/((d-1)\log M)$ as $d\searrow 1$. However, the relative error is bounded in this limit because, again, the contribution of the singular function ($\Phi_1$ in this case) grows in proportion to $1/(d-1)$ as $d\searrow 1$ (as shown in Subsection \ref{s:ErrVsCostSing}). 
We validate these statements numerically in Section \ref{sec:Numerics} below.
\end{rem}

\begin{rem}[Behaviour for vanishing distance between \red{subsets}]
\label{rem:RGammaFinal}
The estimates of The\-o\-rem~\ref{thm:PhiDoubleSingular} also blow up in the limit   
$R_{\Gamma,\Hull}\searrow0$, specifically like $R_{\Gamma,\Hull}^{-(n+1)}$. However, our numerical investigations in Section \ref{sec:Numerics} suggest that, at least in certain cases, this is overly pessimistic. 
\end{rem}

Proposition \ref{prop:PhiStarDoubleSingular} and Theorem \ref{thm:PhiDoubleSingular} are stated only for the case of a \red{homogeneous} IFS. We expect that with non-trivial further work one should be able to extend the $O(h^2)$ estimates in these results to the non-\red{homogeneous} case, but we defer this to future studies. The main difficulty is obtaining sharp estimates for the sum $S_2$ in the proof of Proposition \ref{prop:PhiStarDoubleSingular}.
While we cannot currently prove $O(h^2)$ estimates for the non-\red{homogeneous} case, we can at least prove weaker $O(h)$ estimates. 
The following is the $O(h)$ analogue of Theorem \ref{thm:PhiDoubleSingular}. The proof, which we do not provide here, essentially follows that of Theorem \ref{thm:PhiDoubleSingular}, but applies lower order estimates, and estimates the sum $S_2$ in the proof of Proposition \ref{prop:PhiStarDoubleSingular} more simply (but less sharply) using a uniform bound over all summands.

\begin{thm}\label{thm:PhiDoubleSingularNonUniform} Let $\Gamma\subset\Rnotsn^n$ satisfy the assumptions of Theorem \ref{thm:PhiDoubleSingular}, except that we no longer assume $\Gamma$ is \red{homogeneous} in the sense of Subsection \ref{sec:FurtherAssumptions}. Then 
\[	\left|I_{\Gamma,\Gamma}[\Phi]-Q^h_{\Gamma,\Gamma,\Phi}\right|
\lesssim c''h\big(\cH^d(\Gamma)\big)^2,
\]
where the constant implied by $\lesssim$ depends only on $\cosc$ and
\begin{align*}
c''=
\left\{\begin{array}{l}
\dfrac{1}{\rho_{\min}^{n}(1-\sum_{m=1}^M\rho_m^{2d-n+1})R_{\Gamma,\Hull}^{n}}
,\hfill \quad k\diam(\Gamma)\leq \cosc,
\\
\bigg(1+\dfrac{1}{\rho_{\min}^n\big(k\diam\GG\big)^d (1-\sum_{m=1}^M\rho_m^{2d-n+1})}\bigg)
\bigg(\dfrac{k\diam(\Gamma)}{R_{\Gamma,\Hull}}\bigg)^{n},\qquad\qquad\\
\hfill \quad k\diam(\Gamma)> \cosc,
\end{array}
\right.
\end{align*}
where $\rho_{\min}=\min_{m\in\{1,\ldots,M\}}\rho_m$.
\end{thm}

\section{Numerical experiments}
\label{sec:Numerics}

In this section we present numerical results complementing our theoretical analysis. 
The code used for our numerical experiments is available at \url{https://github.com/AndrewGibbs/IFSintegrals}, where we provide a Julia-based \cite{julia} implementation of all the quadrature rules presented in this paper. \red{Within this repository, the interactive notebook}  \verb|QuadratureExample.jpynb|
\red{provides an overview of the main steps in our algorithm and examples of usage.
The pseudocode for a simple recursive implementation of the quadrature rule $Q_\Gamma[f]$ \eqref{eq:BaryQuad} for regular single integrals is shown in Algorithms \ref{algo:Qhf}--\ref{algo:LeafCalculator} below.

Estimation of $\diam(\Gamma)$ is a key step in Algorithm \ref{algo:Qhf}. In the numerical experiments which follow, $\diam(\Gamma)$ can be derived analytically. But for more general cases where an analytic derivation is not possible, our implementation estimates $\diam(\Gamma)$ using an algorithm that follows from \cite[Proposition~6]{DuHa:94}.
}

\begin{algorithm2e}[tb!]
\fbox{\parbox{.89\textwidth}{\SetAlgoLined
{\bf Data:  
$(A_m,\rho_m,\delta_m, p_m)_{m=1,\ldots,M}, h, \mu\GG, f$} %
\\[2mm]
$x_\emptyset=x_\Gamma=(I-\sum_{m=1}^Mp_m\rho_mA_m)^{-1}\sum_{m=1}^Mp_m\delta_m$

\tcc*[f]{\eqref{eq:xGamma_formula}: compute barycentre of $\Gamma$ from IFS parameters}

$w_\emptyset=\mu\GG$

Estimate $D_\emptyset=\diam\GG$

Set $\mathcal L_h=\emptyset$

\tcc*[f]{List that will contain quadrature nodes and weights}

Set $M, h, (A_m,\rho_m,\delta_m, p_m)_{m=1,\ldots,M},\mathcal L_h$ as global parameters

$A_\emptyset=I,\quad\rho_\emptyset=1,\quad\delta_\emptyset=0$

\tcc*[f]{$s_\emptyset(x)=\rho_\emptyset A_\emptyset x+\delta_\emptyset$ is identity map}

LeafCalculator($x_\emptyset,w_\emptyset,D_\emptyset,A_\emptyset,\rho_\emptyset,\delta_\emptyset$)
 \tcc*[f]{Populate list $\mathcal L_h$}

$Q_\Gamma[f]=\sum_{(x_\bm,w_\bm)\in\mathcal L_h} w_\bm f(x_\bm)$     \tcc*[f]{Apply rule \eqref{eq:BaryQuad}}
}}
\medskip
\caption{\red{Quadrature $Q_\Gamma^h[f]$ for single regular integral $I_\Gamma[f]$. 
Quadrature nodes and weights are computed using the recursive routine LeafCalculator in Algorithm~\ref{algo:LeafCalculator}.}}
\label{algo:Qhf}
\end{algorithm2e}

\begin{algorithm2e}[tb!]
\fbox{\parbox{.89\textwidth}{
\SetAlgoLined
{\bf LeafCalculator}
\\[2mm]
{\bf Input:} $x_\bm, w_\bm, D_\bm, A_\bm,\rho_\bm,\delta_\bm$
\\
\eIf{ $D_\bm\le h$}{
    append  $(x_\bm, w_\bm)$ to $\mathcal L_h$ list
    \\
    $(x_*,w_*,D_*,\rho_*,A_*,\delta_*)$ = ($x_\bm, w_\bm, D_\bm, A_\bm,\rho_\bm,\delta_\bm$)
    }
    {
\For{$\widetilde m=1,\ldots,M$}{
$x_{(\bm,\widetilde m)}
=s_\bm s_{\widetilde m}s_\bm ^{-1}(x_\bm)\\
\rule{11mm}{0mm}=\rho_{\widetilde m} A_\bm A_{\widetilde m}A_\bm ^{-1} (x_\bm -\delta_\bm)
+ \delta_\bm  +\rho_\bm A_\bm \delta_{\widetilde m}$
\\
$w_{(\bm,\widetilde m)}= p_{\widetilde m} w_\bm$
\\
$D_{(\bm,\widetilde m)}=\rho_{\widetilde m}D_\bm$
\\
$\rho_{(\bm,\widetilde m)}= \rho_{\widetilde m}\rho_\bm$
\\
$A_{(\bm,\widetilde m)}=A_\bm A_{\widetilde m}$
\\
$\delta_{(\bm,\widetilde m)}=\rho_\bm A_\bm\delta_{\widetilde m}+\delta_\bm$
\\
$(x_*,w_*,D_*,\rho_*,A_*,\delta_*)=$ LeafCalculator$(x_{(\bm,\widetilde m)},w_{(\bm,\widetilde m)},$\\
\rule{45mm}{0mm}$D_{(\bm,\widetilde m)},\rho_{(\bm,\widetilde m)},A_{(\bm,\widetilde m)},\delta_{(\bm,\widetilde m)})$
}
}
{\bf Return:} $x_*,w_*,D_*,\rho_*,A_*,\delta_*$
}}
\medskip
\caption{\red{Recursive algorithm for the computation of barycentre $x_\bm$, measure $w_\bm=\mu(\Gamma_\bm)$, diameter $D_\bm$ of $\Gamma_\bm$, and the parameters of $s_\bm:\Gamma\to\Gamma_\bm$.
If $\bm\in L_h$ then $(x_\bm,w_\bm)$ is saved in $\mathcal L_h$.}}
\label{algo:LeafCalculator}
\end{algorithm2e}

We shall focus mainly on the validation of the quadrature rule $Q^h_{\Gamma,\Gamma,\Phi}$ defined by \eqref{eq:reduced} and \eqref{eq:PhiQuad} for the calculation of the singular double integral $I_{\Gamma,\Gamma}[\Phi]$ defined in \eqref{eq:double_int_Phi_singular}, since this is the most challenging integral we consider in the paper, and since it is important for the Hausdorff BEM application of \cite{HausdorffBEM}. However, our numerical results for $Q^h_{\Gamma,\Gamma,\Phi}$ also implicitly validate the quadrature rule $Q^h_{\Gamma,\Gamma,t}$ defined in \eqref{eq:non_overlap_quad_gen} for the integration of the singular function $\Phi_t$, and, more fundamentally, the barycentre rule of Definition \ref{def:double} for regular integrands. 

The definition of $Q^h_{\Gamma,\Gamma,\Phi}$ in Section \ref{sec:Phi} involves a parameter $\cosc>0$, which governs whether the integral $I_{\Gamma,\Gamma}[\Phi]$ is treated as non-oscillatory ($k\diam(\Gamma)\leq \cosc$), in which case \eqref{eq:reduced} is applied, or oscillatory ($k\diam(\Gamma)> \cosc$), in which case \eqref{eq:PhiQuad} is used. 
If $\cosc$ is too small, accuracy will deteriorate because the singularity will not be properly captured, while if $\cosc$ is too large, accuracy will also deteriorate because the splitting \eqref{eq:Splitting} is being used outside of its range of applicability. Our experience, following a detailed numerical investigation, suggests that a value of $\cosc=2\pi$ gives acceptable performance across all the examples we considered, and this is the value of $\cosc$ we use throughout this section. This means we classify the integral $I_{\Gamma,\Gamma}[\Phi]$ to be oscillatory (and use \eqref{eq:PhiQuad} rather than \eqref{eq:reduced}) whenever the diameter of $\Gamma$ is larger than one wavelength. 

\paragraph{Cantor sets.}
We first consider the calculation of $I_{\Gamma,\Gamma}[\Phi]$ in the case where $\Gamma\subset\Rnotsn$ is a Cantor set, defined by \eqref{eq:CantorDef}  for some $\rho\in(0,1/2)$, \red{and $\mu=\cH^d|_\Gamma$.} 
In this case $\Gamma$ is \red{homogeneous}, with $d=\log{2}/\log{(1/\rho)}$ and $\cH^d(\Gamma)=1$ (see e.g.\ \cite[p.~53]{Fal}), and hull-disjoint, with $R_{\Gamma,\Hull,h}=R_{\gamma,\Hull}=R_\Gamma=1-2\rho$, for $0<h<\diam(\Gamma)=1$.
In Figure \ref{fig:CantorSetRhoZero} we plot absolute and relative errors for the quadrature rule $Q^h_{\Gamma,\Gamma,\Phi}$ as a function of $N=|L_h(\Gamma)|=2^\ell$, for $\ell=2,\ldots,9$, $k=5$ and $\rho\in\{1/3, 0.1, 0.01, 0.001\}$. The reference solution $I^{\rm ref}_{\Gamma,\Gamma}[\Phi]$ in each case is computed using the quadrature rule $Q^h_{\Gamma,\Gamma,\Phi}$ with $N=8192$ ($\ell=13$). 
We also plot on the same axes the corresponding theoretical convergence rate $N^{-2/d}$ (which differs for each value of $\rho$) predicted by Theorem \ref{thm:PhiDoubleSingular}. For all choices of $\rho$ we see excellent agreement with the theory. Moreover, both the absolute and relative errors for a given $N$ clearly decrease as $\rho\searrow 0$ (equivalently, $d\searrow 0$), in line with the observations of Remark \ref{rem:BlowUpFinal}.

\begin{figure}[t!]
	\centering
		\subfigure[][
	{Absolute error}
	]
	{\includegraphics[width=.48\linewidth]{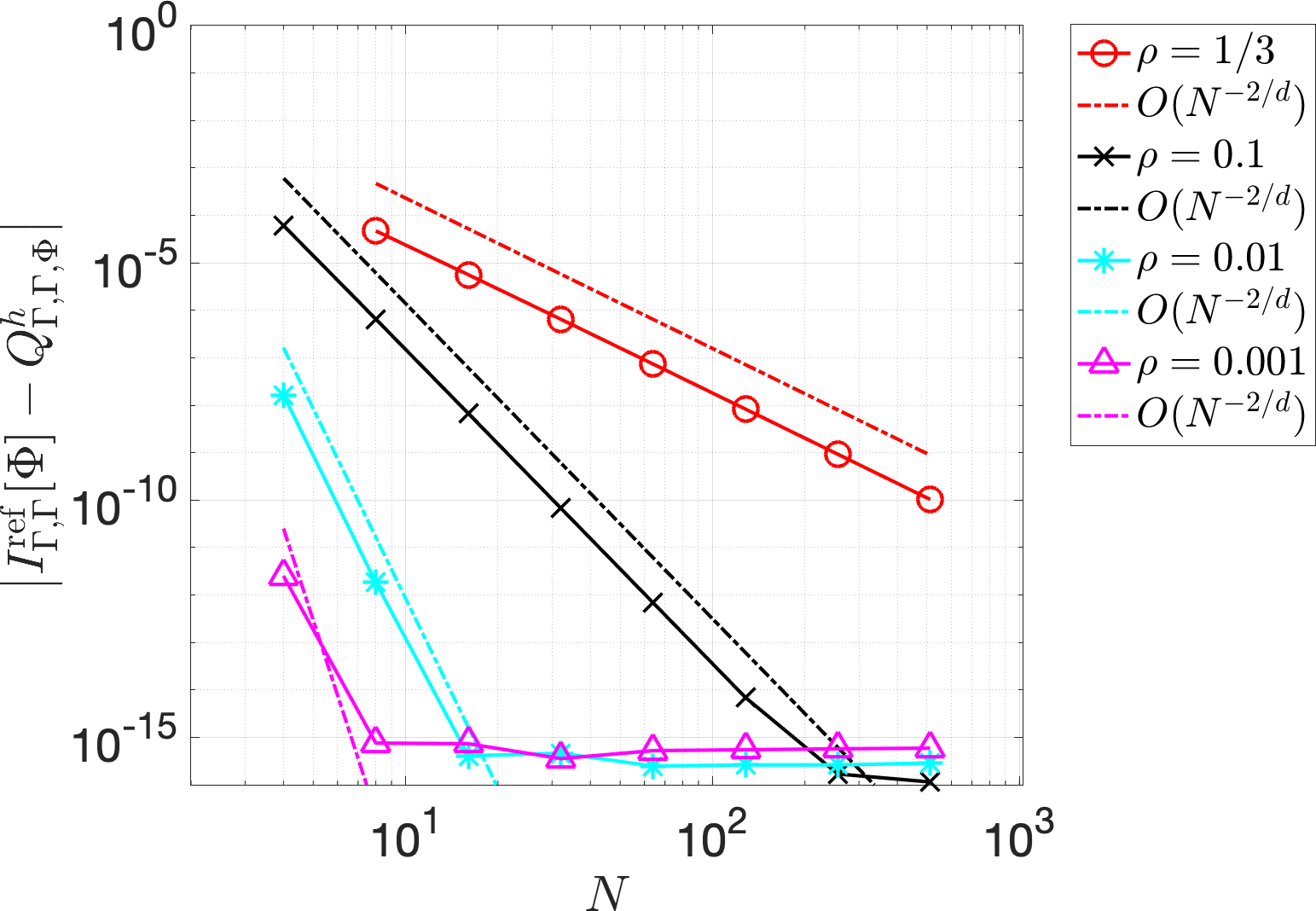}}
\hspace{2mm}
	\subfigure[][
	{Relative error}
	]
	{\includegraphics[width=.48\linewidth]{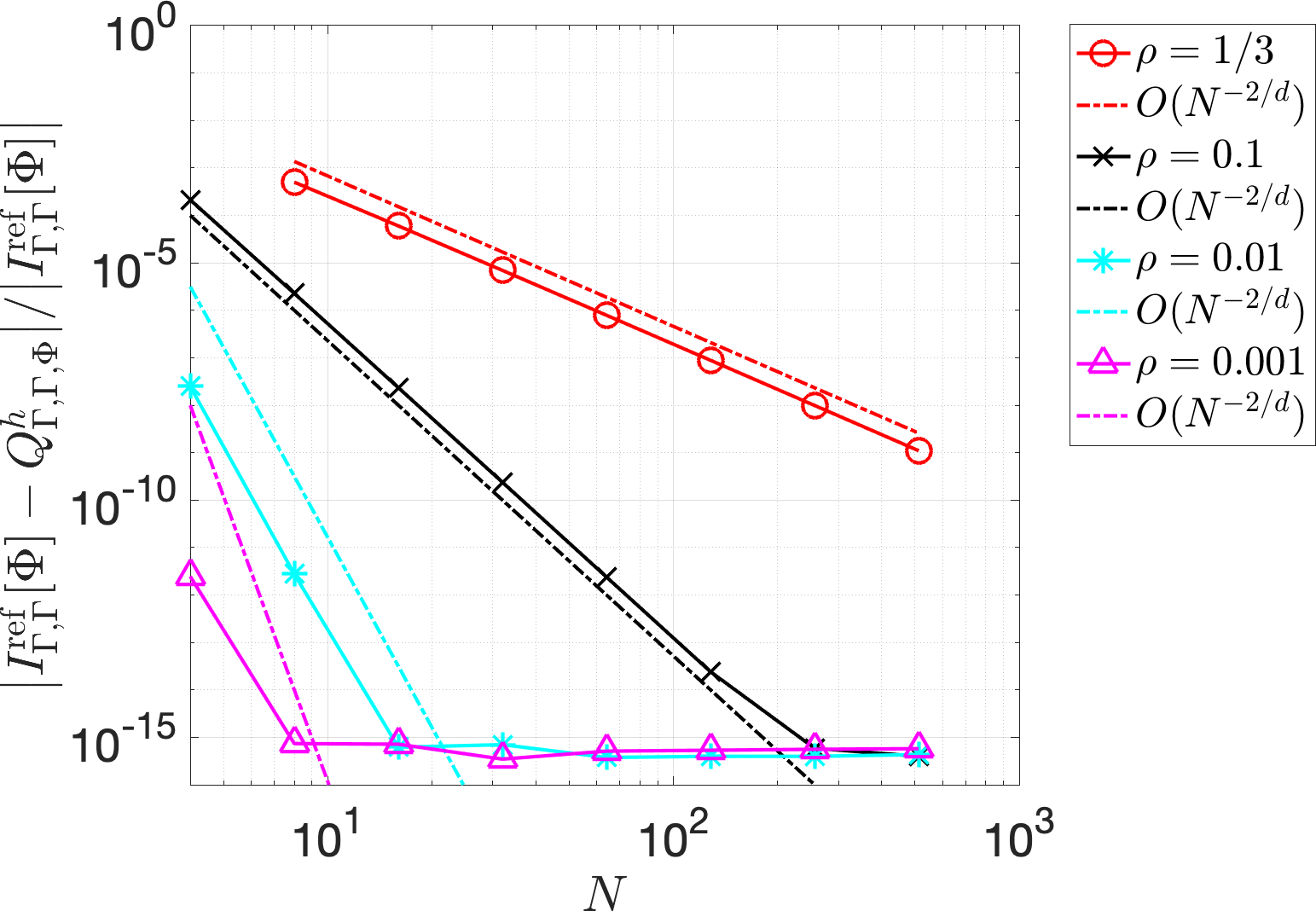}}
	\hspace{5mm}
	\caption{Convergence of $Q_{\Gamma,\Gamma,\Phi}^h$ %
	for a collection of Cantor sets with parameters $\rho$ approaching $0$.}
\label{fig:CantorSetRhoZero}
\end{figure}

\paragraph{Cantor dusts.}
Next we consider the case where $\Gamma$ is a Cantor dust, defined as in the first line of Table \ref{t:disjointness} with $\rho\in (0,1/2)$\red{, and $\mu=\cH^d|_\Gamma$}.
Again, $\Gamma$ is \red{homogeneous}, now with $d=\log4/\log(1/\rho)$, and hull-disjoint, with $R_{\Gamma,\Hull,h}=R_{\gamma,\Hull}=R_\Gamma=1-2\rho$, for $0<h<\diam(\Gamma)=\sqrt2$.
In this case the Hausdorff measure $\cH^d(\Gamma)$ is not known exactly, so the double integral $I_{\Gamma,\Gamma}[\Phi]$ can only be computed up to the unknown factor $\cH^d(\Gamma)^2$.
In Figure \ref{fig:CantorDustRhoQuarter} we present absolute (scaled by $\cH^d(\Gamma)^2$) and relative errors for $k=5$ and $\rho\in\{1/3,0.26,0.251,0.2501\}$, plotted against $N=4^\ell$, $\ell=3,4,5$, along with the corresponding theoretical convergence rate $N^{-2/d}$. 
The reference solution $I^{\rm ref}_{\Gamma,\Gamma}[\Phi]$ in each case is computed using the quadrature rule $Q^h_{\Gamma,\Gamma,\Phi}$ with $N=16384$ ($\ell=7$). 
The behaviour as $N\to\infty$ is clearly consistent with the theoretical convergence rates. Moreover, as $\rho\searrow 1/4$ (equivalently, as $d\searrow 1$), the absolute error grows, while the relative error remains bounded, as predicted in Remark \ref{rem:BlowUpFinal}.

\begin{figure}[t!]
	\centering
		\subfigure[][
	{Absolute error (scaled by $\cH^d(\Gamma)^2$)}
	]
	{\includegraphics[width=.48\linewidth]{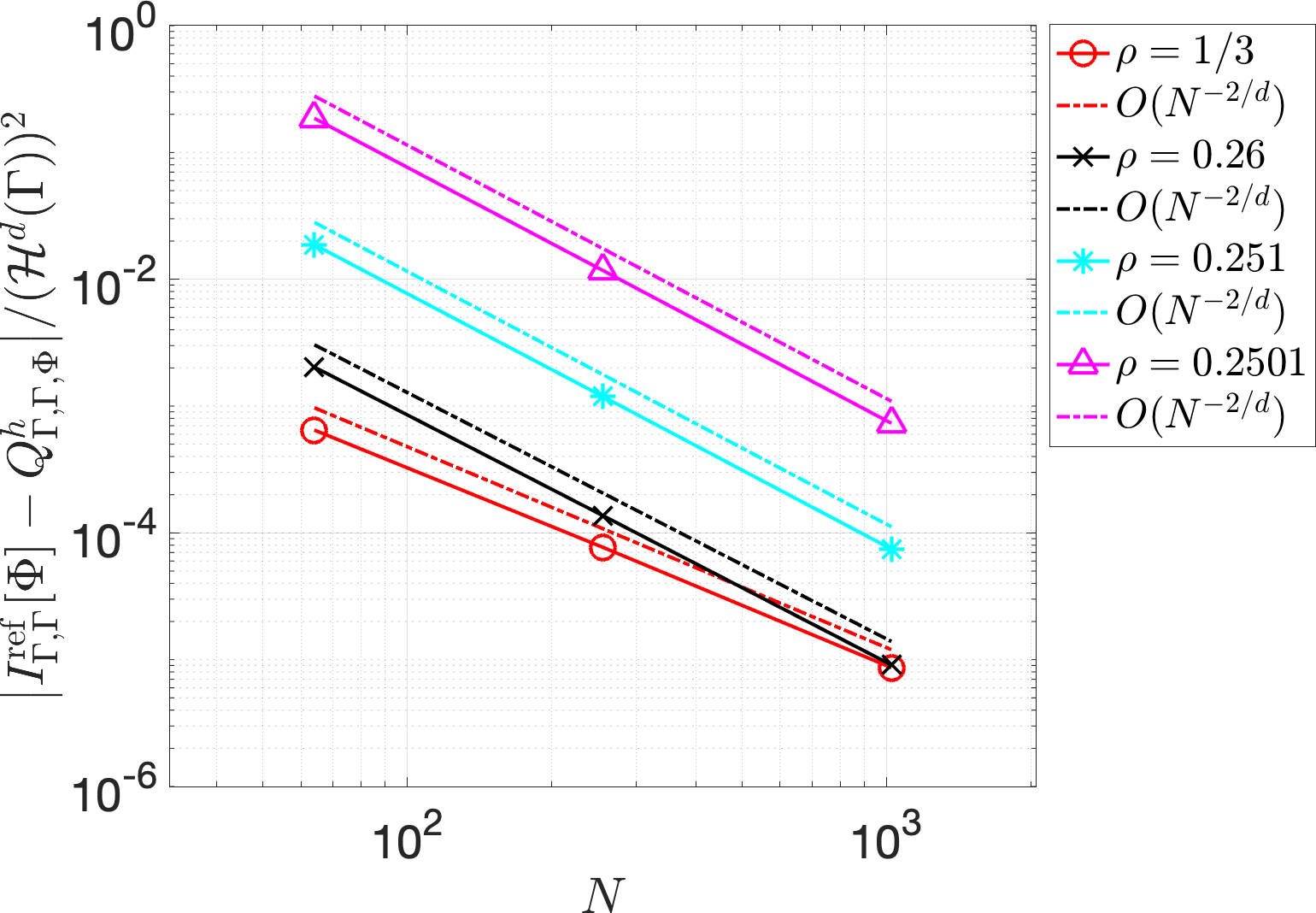}}
	\subfigure[][
	{Relative error}
	]
	{\includegraphics[width=.48\linewidth]{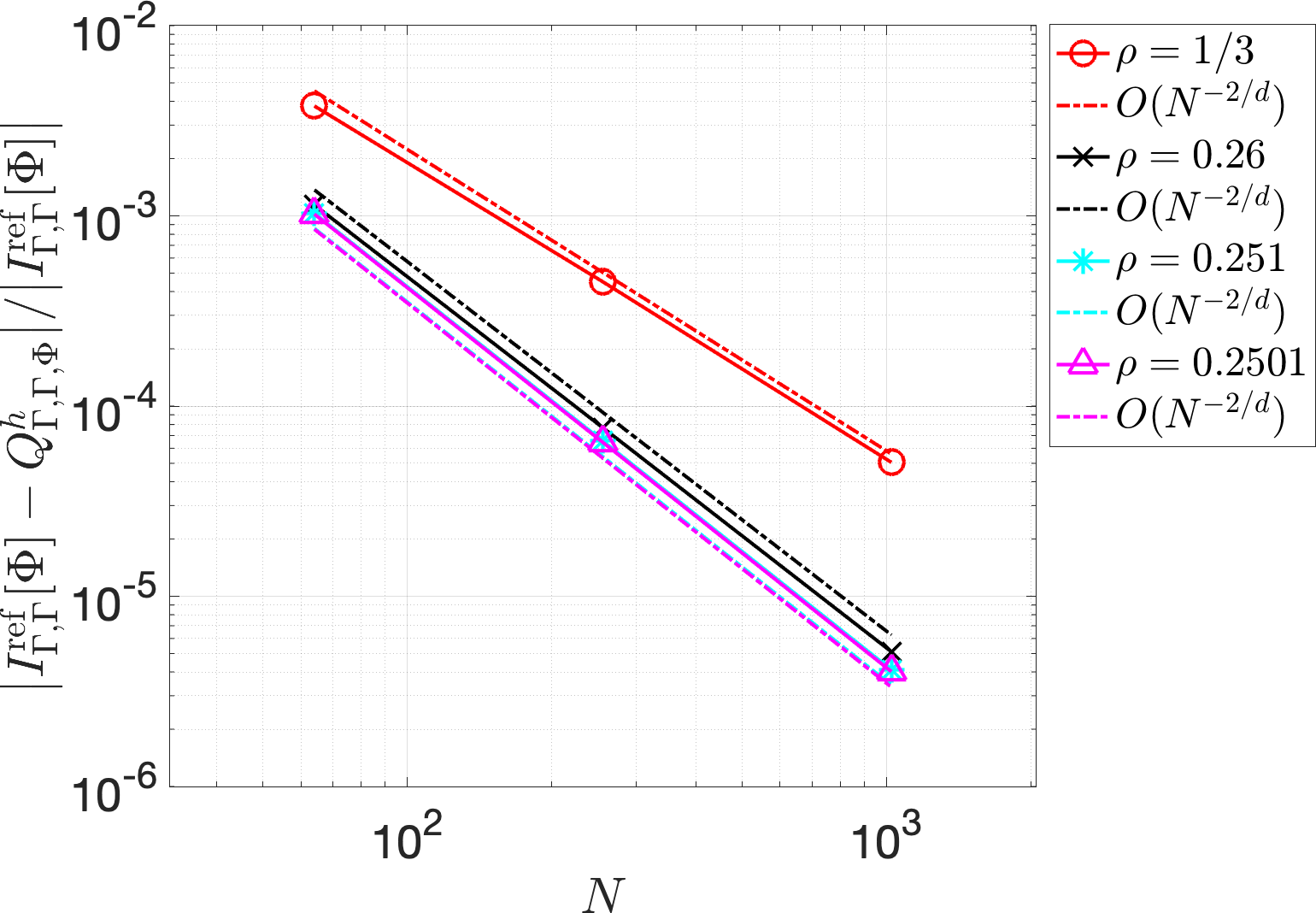}}
	\caption{Convergence of $Q_{\Gamma,\Gamma,\Phi}^h$ %
	for a collection of Cantor dusts with parameters $\rho$ approaching $1/4$.}
\label{fig:CantorDustRhoQuarter}
\end{figure}

\paragraph{Vanishing separation limit.}
Next we consider the behaviour of our quadrature rules as the parameter $R_{\Gamma,\Hull}$ tends to zero. 
In Figure \ref{fig:CantorSetGapShrink}(a) we show absolute errors for $Q_{\Gamma,\Gamma,\Phi}^h$ %
with with $k=5$ at three values of $N$, for $\Gamma\subset\Rnotsn$ a Cantor set, defined by \eqref{eq:CantorDef}, with $\rho=(1-R_\Gamma)/2$, where $R_{\Gamma}=1-2\rho\in\{0.1,0.01,0.001,0.0001,0.00001\}$\red{, and $\mu=\cH^d|_\Gamma$}. 
The reference solution is as for Figure \ref{fig:CantorSetRhoZero}. 
In Remark \ref{rem:RGammaFinal} we observed that as $R_{\Gamma,\Hull}\to 0$ our theory predicts blow-up of the error like $R_{\Gamma,\Hull}^{n+1}$, i.e.\ like $R_{\Gamma}^{2}$ in this case. However, the numerical results suggest that, at least in this case, the theoretical prediction is overly pessimistic, since the error appears to be bounded as $R_{\Gamma,\Hull}\to 0$. In fact, the integral for the non-disjoint case $\rho=1/2$ (so $\Gamma=[0,1]$, $d=1$ and $R_{\Gamma}=0$) can be computed using our method, and the corresponding errors appear to follow the same $N^{-2/d}$ behaviour (in this case, $N^{-2/d}=N^{-2}$ since $d=1$) with respect to increasing $N$ as for the case $0<\rho<1/2$ (see the dashed lines in the figure). To further investigate the non-disjoint case $\rho=1/2$ (with $\Gamma=[0,1]$, $d=1$ and $R_{\Gamma}=0$), in Figure \ref{fig:CantorSetGapShrink}(b) we plot the absolute error in the quadrature rule $Q_{\Gamma,\Gamma,0}^{h}\approx I_{\Gamma,\Gamma}[\Phi_0]$ for this case, for which we have the exact result
\begin{align}
\label{eq:ExactIntegralLog}
I_{\Gamma,\Gamma}[\Phi_0] 
= \int_0^1 \int_0^1 \log{|x-y|} \,\rd\cH^1(y)\,\rd\cH^1(y) 
= \int_0^1 \int_0^1 \log{|x-y|} \,\rd y\,\rd x
= -\frac{3}{2},
\end{align}
where we used the fact that $\cH^1$ coincides with the Lebesgue measure on $\Rnotsn$. 

\begin{figure}[t!]
\centering
\subfigure[][
{
$R_\Gamma=R_{\Gamma,\Hull}\to 0$}
]
{\includegraphics[width=.45\linewidth]{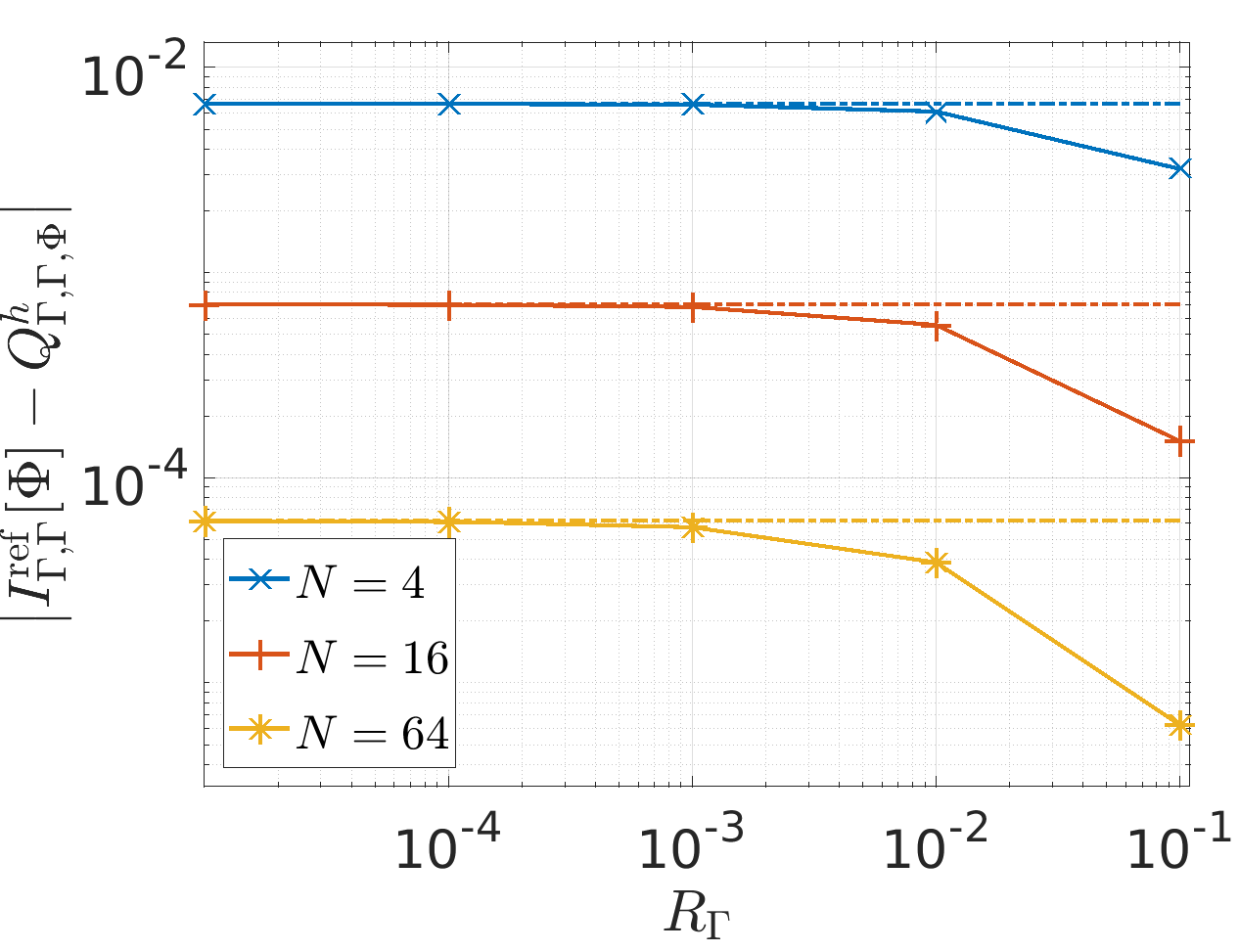}}
\hspace{5mm}
\subfigure[][
{
$R_\Gamma=R_{\Gamma,\Hull}= 0$}
]
{\includegraphics[width=.45\linewidth]{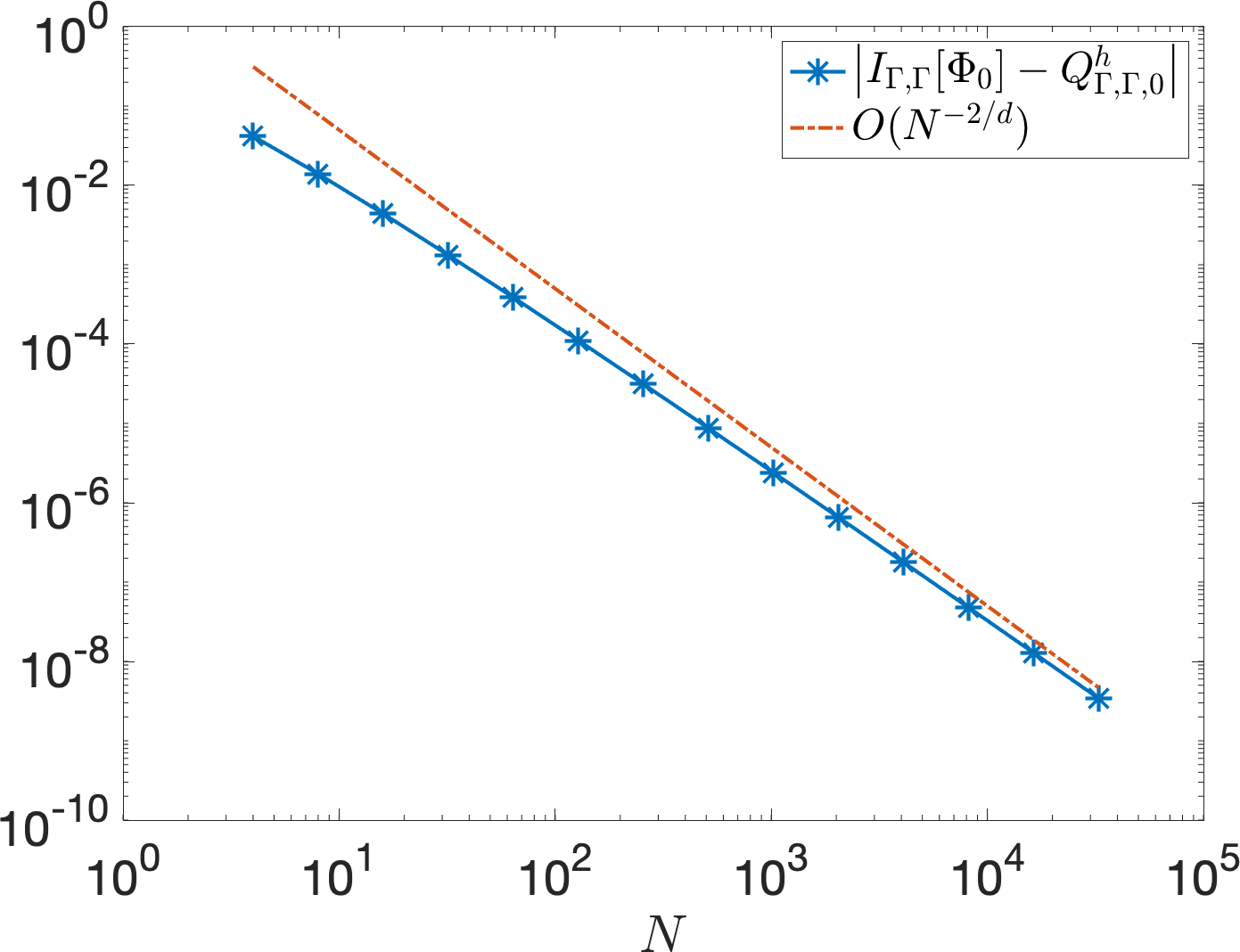}}
\caption{(a) Absolute error for $Q_{\Gamma,\Gamma,\Phi}^{h}$ for a collection of Cantor sets with parameters $\rho$ approaching $1/2$, i.e.\ $R_\Gamma=R_{\Gamma,\Hull}$ approaching $0$, for three different values of $N$ and $k=5$.
The dashed lines indicate the value of $Q_{\Gamma,\Gamma,\Phi}^{h}$ in the case $\rho=1/2$ ($R_\Gamma=R_{\Gamma,\Hull}=0$), in which case $\Gamma=[0,1]$.
\newline
(b) Absolute error for $Q_{\Gamma,\Gamma,0}^{h}$ for a Cantor set with $\rho=1/2$ ($R_\Gamma=R_{\Gamma,\Hull}=0$), i.e.\ $\Gamma=[0,1]$. In this case we have an exact value $I_{\Gamma,\Gamma}[\Phi_0]=-3/2$ with which to compute errors (see \eqref{eq:ExactIntegralLog}).}
\label{fig:CantorSetGapShrink}
\end{figure}

\red{\paragraph{Non-disjoint, non-hull-disjoint and non-homogeneous examples.}}
The results in Figure \ref{fig:CantorSetGapShrink} suggest that our assumption that $\Gamma$ should be hull-disjoint, or even disjoint at all, may not be necessary in Theorem \ref{thm:PhiDoubleSingular}.
To investigate this further we compute $Q_{\Gamma,\Gamma,\Phi}^h$ %
for two non-hull-disjoint examples \red{and $\mu=\cH^d|_\Gamma$}: example (II) in Table \ref{t:disjointness}, which is disjoint but not hull-disjoint, and example (IV) in Table \ref{t:disjointness}, which is not disjoint.
Both attractors are shown in Figure~\ref{fig:disjointness}. 
Absolute errors (scaled by $\cH^d(\Gamma)^2$) for these cases for $k=2$ and a range of $h$ values are presented in Figure \ref{fig:kconv}(a), and for both examples it seems we obtain $O(h^2)$ convergence, even though our theoretical error analysis does not cover these cases. 
Results for the middle-third Cantor dust (example (I) in Table \ref{t:disjointness}) are included in the same figure for reference. 

\begin{figure}[t!]
\centering
\subfigure[]{\includegraphics[width=.49\linewidth]{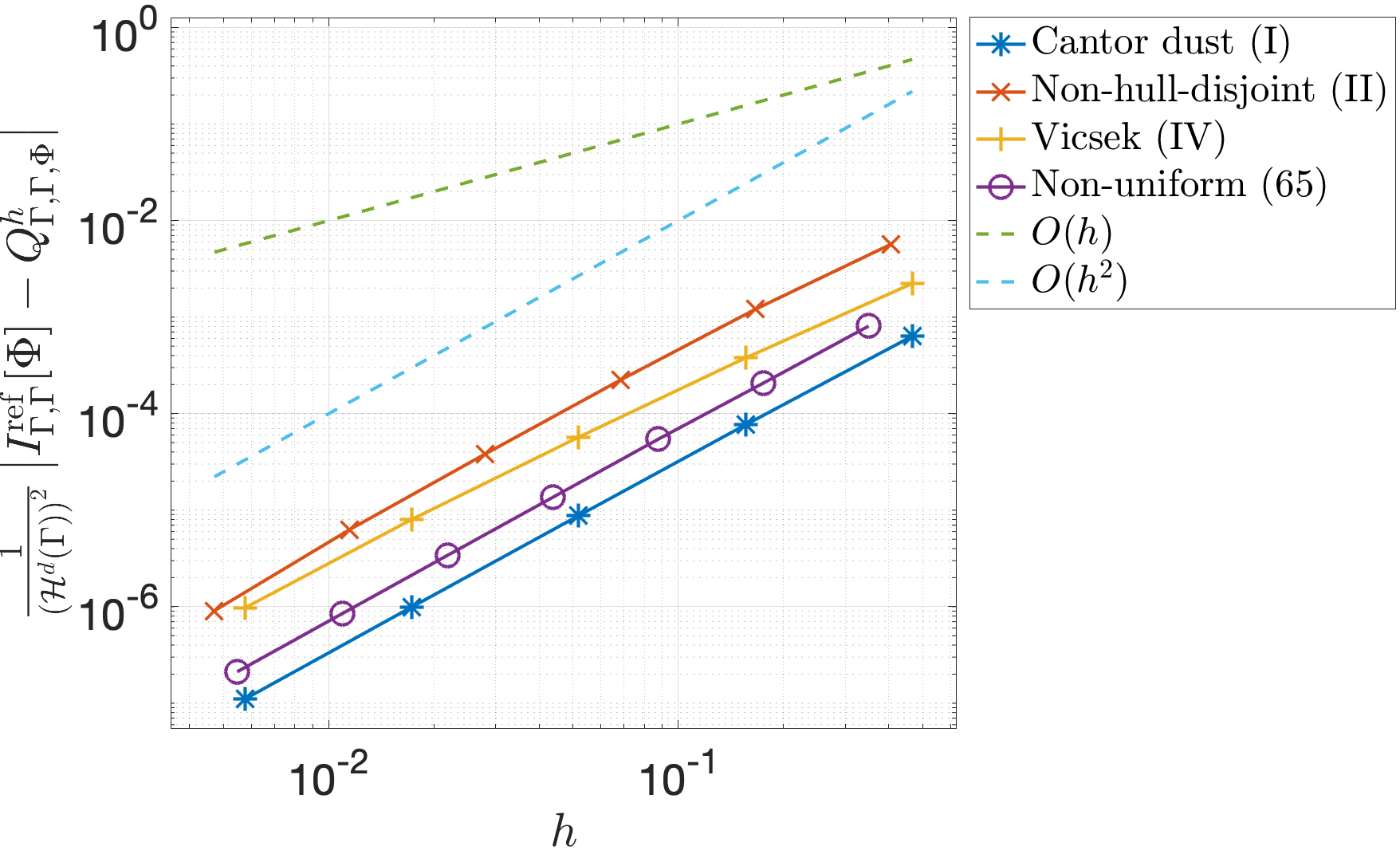} }%
\subfigure[]{\includegraphics[width=0.45\linewidth]{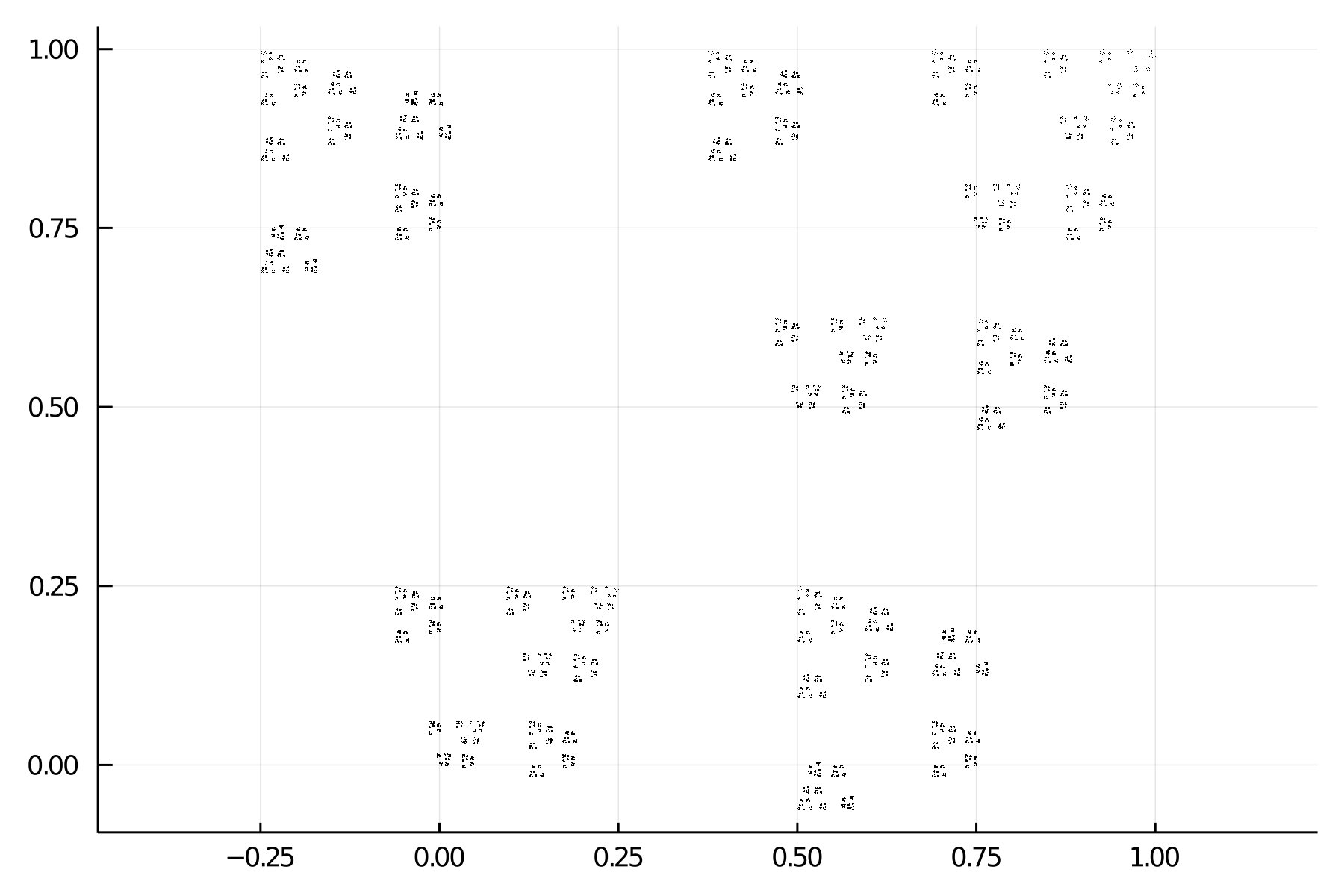}   }
\caption{(a) Convergence of $Q_{\Gamma,\Gamma,\Phi}^h$ %
with $k=2$ for examples (I) (\red{homogeneous} and hull-disjoint), (II) (\red{homogeneous} and disjoint but not hull-disjoint) and (IV) (\red{homogeneous} and not disjoint) from Table \ref{t:disjointness}, along with the attractor of \eqref{eq:wonky} (hull-disjoint but non-\red{homogeneous}). 
The values of $N$ corresponding to (i) the largest $h$ value shown, (ii) the smallest $h$ value shown, and (iii) the reference solution, are 
$(16,4096,65536)$ for (I), 
$(16,16384,65536)$ for (II), 
$(25,15625,78125)$ for (IV), 
and 
$(19,14209,75316)$ for \eqref{eq:wonky}. 
(b) An approximation of the attractor defined by \eqref{eq:wonky}, produced by plotting $s^6(E)$, with $s$ defined by \eqref{eq:fullmap}, with $E$ a set of ten random points in $[0,1]^2$.}
\label{fig:kconv} 
\end{figure}

In Figure \ref{fig:kconv}(a) we also include results for a hull-disjoint but non-\red{homogeneous} IFS with $M=4$ and
\begin{align}
&s_1(x) = \frac{1}{4}x,\quad
s_2(x) = \frac{1}{4}Ax + \Bigg(\begin{matrix}3/4\\0\end{matrix}\Bigg),\quad
s_3(x) = \frac{1}{4}Ax+ \Bigg(\begin{matrix}0\\3/4\end{matrix}\Bigg)
,\quad\nonumber\\
&\qquad\qquad\qquad\qquad\qquad s_4(x) = \frac{1}{2}x+ 
\Bigg(\begin{matrix}1/2\\1/2\end{matrix}\Bigg),\label{eq:wonky}
\end{align}
for the rotation matrix 
$A =\big(\begin{smallmatrix} 0&-1\\1& 0\end{smallmatrix}\big)$.
The attractor $\Gamma$ for this IFS is sketched in Figure \ref{fig:kconv}(b) and has Hausdorff dimension $d = \log((1 + \sqrt{13})/2)/\log2\approx 1.20$ and diameter $\diam(\Gamma)=\sqrt{2}$.
For such non-\red{homogeneous} hull-disjoint cases our current analysis only provides an $O(h)$ convergence result (see Theorem \ref{thm:PhiDoubleSingularNonUniform}). But the results in Figure \ref{fig:kconv} for the IFS \eqref{eq:wonky}
suggest that, at least in this case, our analysis may not be sharp in this respect, since we seem to obtain $O(h^2)$ convergence in practice. We leave further investigation of this to future work.

\red{\paragraph{Comparison against chaos-game quadrature}

In this section we compare the barycentre rule \eqref{eq:BaryQuad} with the ``chaos game'' rule described, e.g., in \cite[eqn~(3.22)--(3.23)]{forte1998chaos} and \cite[Section~6.3.1]{kunze2011fractal}.
This consists of 
(i) choosing some $x_0\in\mathbb{R}^n$ (we take $x_0=x_\Gamma$ in the numerical example below), 
(ii) selecting a realisation of the sequence $\{m_j\}_{j\in\mathbb N}$ of i.i.d.\ random variables taking values in $\{1,\ldots,M\}$ with probabilities $\{p_1,\ldots,p_M\}$, %
(iii) constructing the stochastic sequence $x_j=s_{m_j}(x_{j-1})$ for $j\in\mathbb N$, and 
(iv) approximating the integral of a continuous function $f$ as
$$
Q_\Gamma^{CG}[f]
=\frac1N\sum_{j=1}^N f(x_j)\xrightarrow{N\to\infty}
\int_\Gamma f(x)\rd\mu(x).
$$

We first consider the case where $\Gamma\subset\mathbb{R}^2$ is the Koch snowflake, the attractor of an non-homogeneous non-disjoint IFS with $M=7$ whose parameters were given in Figure \ref{fig:Koch_decomp}. We consider integration of the (smooth) function $f(x)=\cos|x|/(1+|x|^2)$ over $\Gamma$ with respect to the non-Hausdorff invariant measure $\mu$ with $\mu(\Gamma)=1$ and the following randomly chosen weights/probabilities:
\[
(p_1,p_2,p_3,p_4,p_5,p_6,p_7)=(0.052,0.214,0.104,
0.038,0.110,0.194,0.288).
\]
For this non-Hausdorff invariant measure it is instructive to compare how the two quadrature rules deal with the non-uniform way in which the mass of the measure is distributed across $\Gamma$.
In Figure \ref{fig:koch2in1} we plot the nodes and weights for the barycentre rule for the case $h=0.01$ (which corresponds to $N=35839$) alongside those for one realisation of the chaos game rule with the same $N=35839$.
Each node is represented by a small dot, coloured according to the corresponding quadrature weight. 
For the chaos game, the weights are uniform, all being equal to $1/N$, but the nodes are distributed non-uniformly, being concentrated in the regions where the measure has greatest mass. By contrast, the nodes for the barycentre rule are distributed approximately uniformly, but the weights vary according to the measure.

\begin{figure}[t!]
	\centering
	\includegraphics[width=\linewidth]{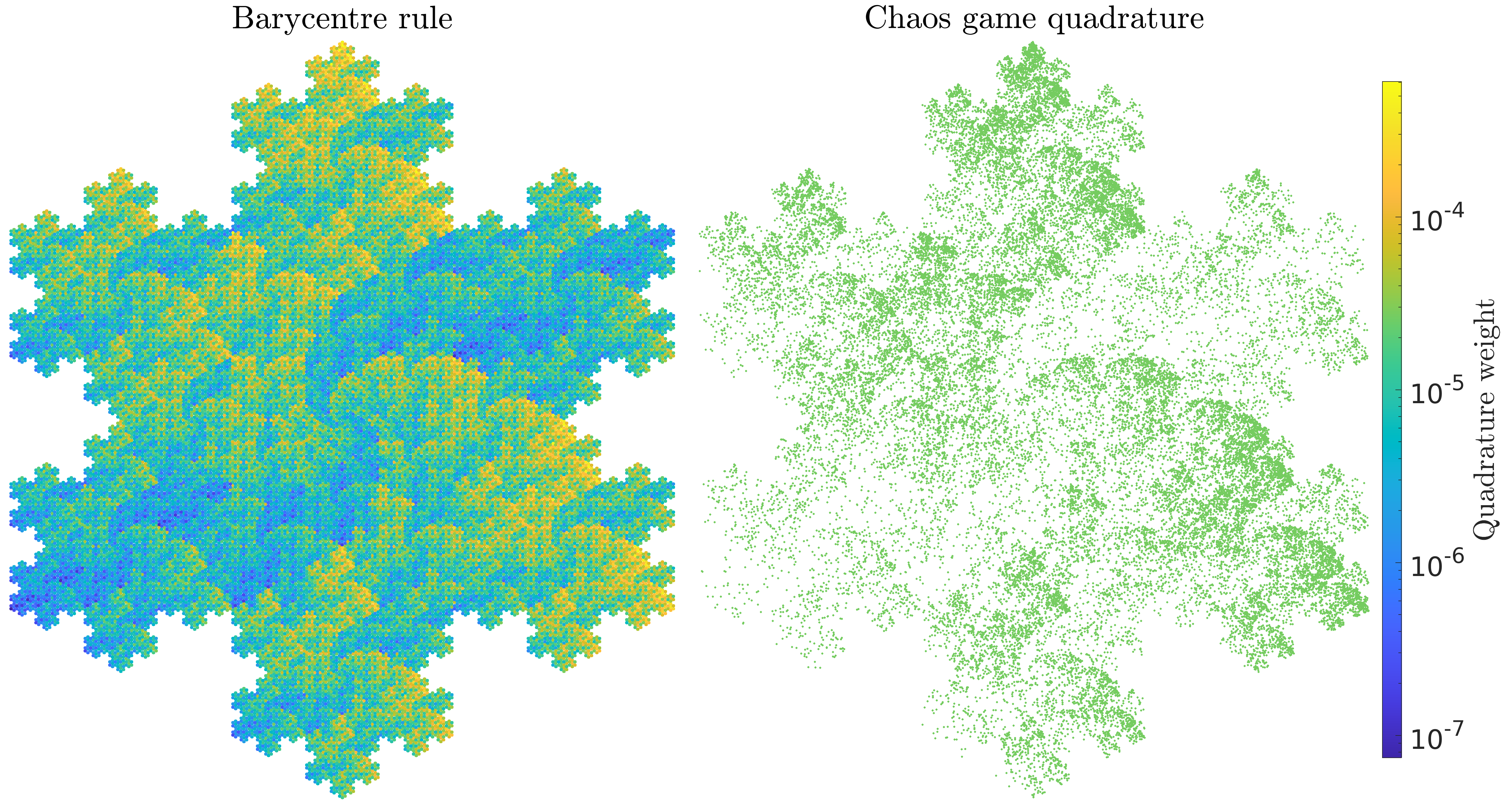}
	\caption{\red{Visual representation of barycentre rule (left) and chaos game quadrature (right) on the Koch snowflake, for a randomly chosen invariant measure. Each quadrature node is represented by a small dot, coloured according to the corresponding quadrature weight.}}
	\label{fig:koch2in1}
\end{figure}

In Figure \ref{fig:barychaoskochconvrand} we plot the relative quadrature errors $|Q_\Gamma[f]-I^{\mathrm{ref}}_\Gamma[f]|/|I^{\mathrm{ref}}_\Gamma[f]|$ and $|Q_\Gamma^{CG}[f]-I^{\mathrm{ref}}_\Gamma[f]|/|I^{\mathrm{ref}}_\Gamma[f]|$ 
against the number $N$ of point evaluations of $f$, for a range of values of $N$ between $463$ and $320503$. Here the reference value $I^{\mathrm{ref}}_\Gamma[f]$ was computed using the barycentre rule with $N=2876335$ (which corresponds to $h=10^{-3}$). For the chaos game rule, the plots show both the individual errors for each of 1000 random realisations (thin blue lines) and the average of these individual errors (thick red line), which represents an approximation to the statistical expectation of the error for the chaos game rule. The error for the barycentre rule clearly decays like $h^2\sim N^{-2/d}=N^{-1}$, 
consistently with Theorem \ref{th:MidLip1}{\it(iii)} and Remark \ref{rem:ErrVsCost}, even if the latter does not directly apply to non-homogeneous IFSs,\footnote{\red{We remark that, although this IFS is not homogeneous, we still have $N=|L_h(\Gamma)|\sim h^{-2}$ for the barycentre rule, since $h/3<\diam(\Gamma_{\bm})\leq h$ and hence $B_{h/(6\sqrt{3})}(x_\bm)\subset\Gamma_\bm\subset B_{h/2}(x_\bm)$ for all $\bm\in L_h(\Gamma)$, and since $\sum_{\bm\in L_h(\Gamma)}|\Gamma_\bm|=|\Gamma|$ ($|\cdot|$ being the Lebesgue measure in $\mathbb R^2$).}} 
while the average error for the chaos game rule decays like ${N^{-1/2}}$, as one expects from a Monte-Carlo-type stochastic method.
So for this problem the barycentre rule clearly outperforms the chaos game rule. 
Comparing the convergence rates $h^2\sim N^{-2/d}$ (for the homogeneous case and the present one) and $N^{-1/2}$, we expect that the advantage provided by the barycentre rule over the chaos game rule is even stronger for the lower-dimensional attractors considered in the previous numerical experiments. 

\begin{figure}[t!]
	\centering
	\includegraphics[width=0.9\linewidth]{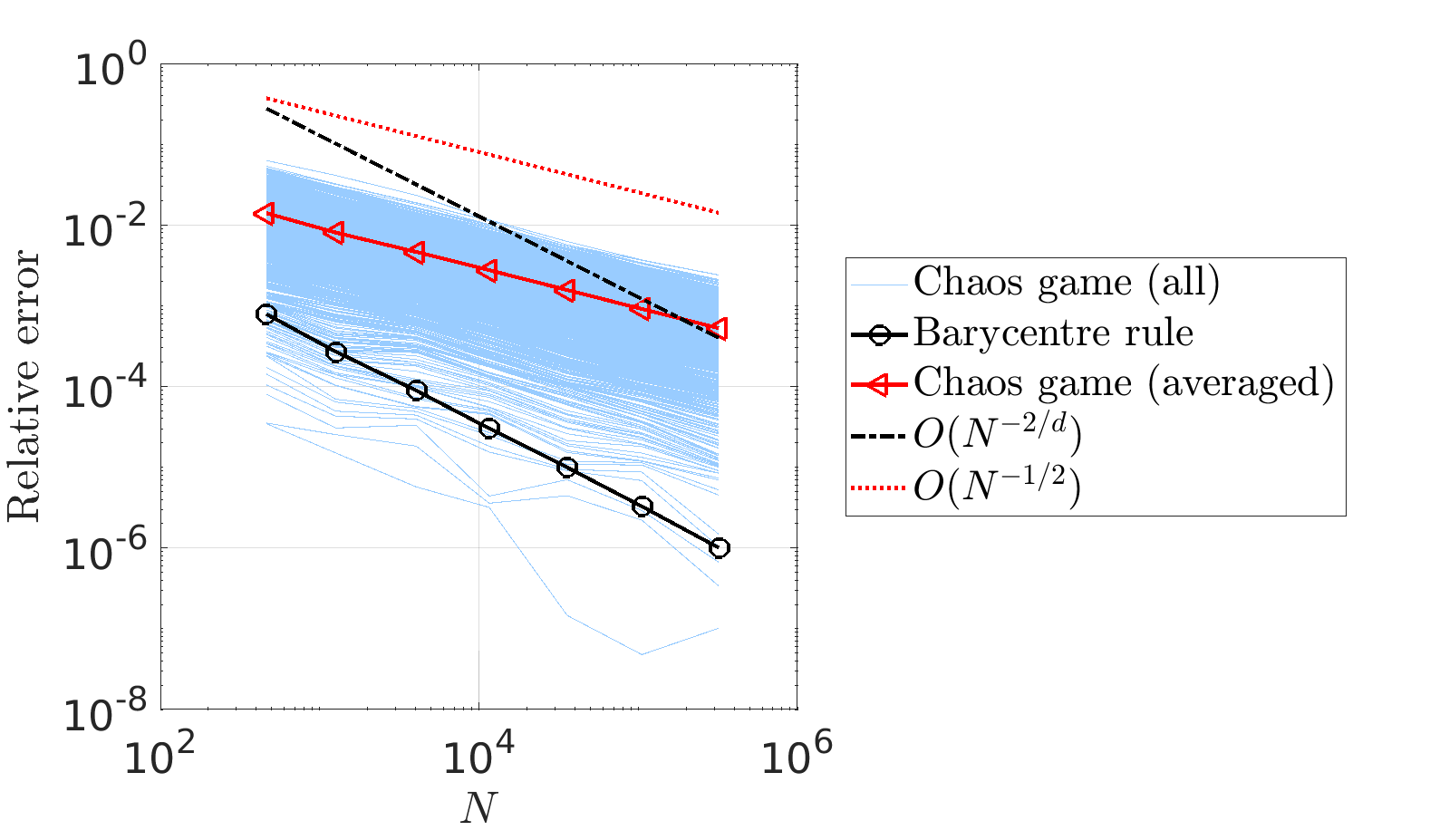}
\caption{Convergence of the barycentre rule and the chaos game rule for integration of a smooth integrand with respect to a non-Hausdorff invariant measure on the Koch snowflake. For the chaos game rule we show both the individual errors for each of 1000 random realisations (thin blue lines) and the average of these individual errors (thick red line).}
	\label{fig:barychaoskochconvrand}
\end{figure}

For higher-dimensional problems we might expect the stochastic approach to become more competitive. To investigate this we consider the case where $\Gamma$ is a high-dimensional Cantor dust. Specifically, we take $\Gamma=(\Gamma_\rho)^6\subset [0,1]^6\subset\mathbb R^6$, a 6-fold Cartesian product of a homogeneous dyadic Cantor set $\Gamma_\rho\subset\mathbb R$ (with contractions $s_1(x)=\rho x$ and $s_2(x)=1-\rho+\rho x$ for some $0<\rho<1/2$) with itself, so that $\Gamma$ is the attractor of a homogeneous, disjoint IFS with $M=2^6=64$ and $d=\dimH\GG=\log{M}/\log(1/\rho)=6\log{2}/\log(1/\rho)$. 
Figure~\ref{fig:ChaosCantor} shows the relative quadrature errors 
for integration of the (smooth) integrand $f(x)=\cos|x|/(1+|x|^2)$ 
with respect to the normalised Hausdorff measure $\mu=\frac1{\cH^d(\Gamma)}\cH^d|_\Gamma$, for three different values of $\rho$, namely $\rho\in\{\frac14, \frac1{2\sqrt2},\frac1{2^{6/5}}
\}$, corresponding to $d=\dimH\GG\in\{3,4,5\}$ respectively.
Both rules are applied with $N\in\{64,64^2,64^3\}=\{64,4096,262144\}$ and the reference value is computed using the barycentre rule with $N=64^4=16777216$. 

For all three values of the dimension $d$, the barycentre rule converges like %
$h^2\sim N^{-2/d}$,   
as predicted by Theorem~\ref{th:MidLip1}{\it(iii)}, while the average error for the chaos game rule converges consistently like $N^{-1/2}$.  
Hence for $d=3$ the barycentre rule converges faster; for $d=4$ the two methods converge at the same rate, and for $d=5$ the chaos game rule converges faster (although we note that in this particular experiment the errors for the barycentre rule were smaller than the expected errors for the chaos game even for $d=5$). 
}

\begin{figure}[t!]
\centering
\includegraphics[width=.99\linewidth, clip, trim=100 0 90 0]{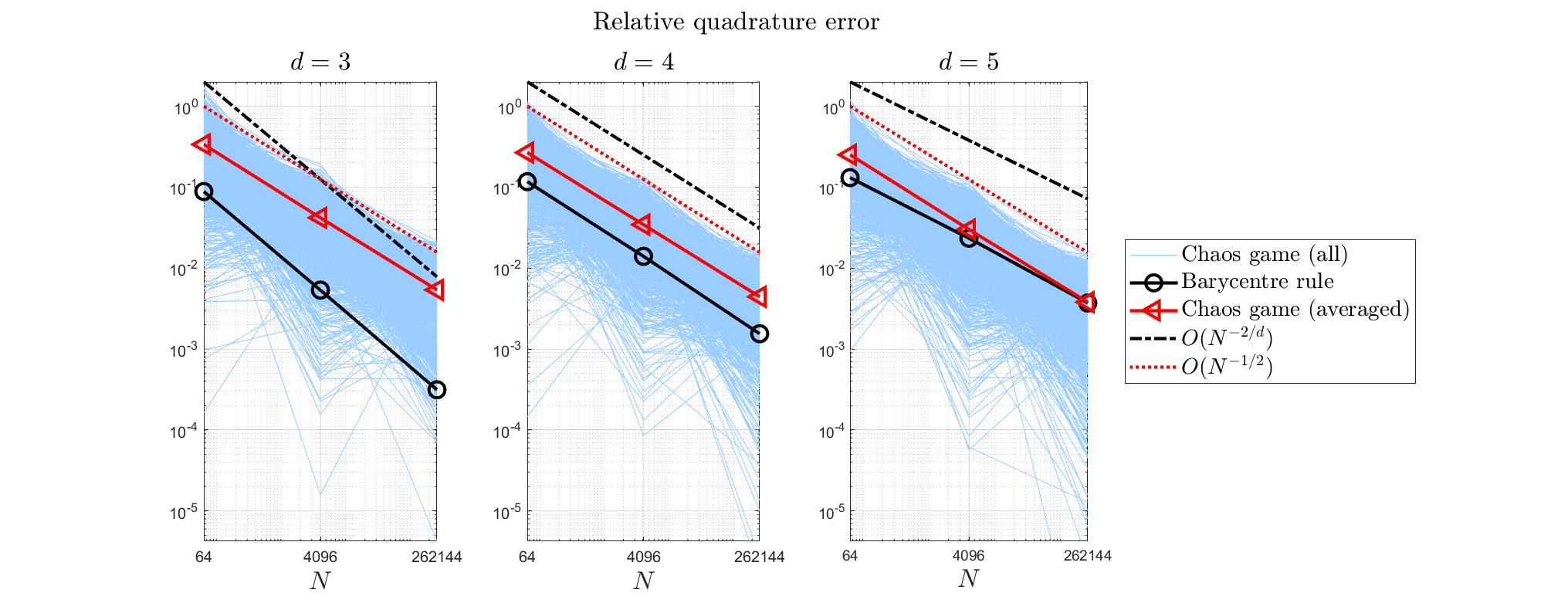}   
\caption{\red{Convergence of the barycentre rule and the chaos game rule for the approximation of a single smooth integral on three Cantor dusts $\Gamma\subset\mathbb R^6$ with $d=\dimH(\Gamma)=3,4,5$.}}
\label{fig:ChaosCantor} 
\end{figure}

\appendix
\red{\section{Integrability of singular functions with respect to invariant measures}

In this appendix we collect some results concerning the integrability of singular functions with respect to invariant measures of the type defined in Section \ref{sec:scaling}. In particular, we study for which $t\geq 0$ the single integral 
$I_\Gamma[\Phi_t(\cdot,\eta)]$ (defined in \eqref{eq:singint_2}) and the double integral $I_{\Gamma,\Gamma}[\Phi_t]$ (defined in \eqref{eq:double_int}) are finite.

\subsection{The Hausdorff measure case \texorpdfstring{$\mu=\cH^d|_\Gamma$}{mu=Hd}}
\label{sec:IntegrabilityHausdorff}
In the case where $\mu=\cH^d|_\Gamma$, everything we need is provided by the following lemma, 
which is adapted from \cite[Lemma 2.13]{CC08}. 
\begin{lem}[{\cite[Lemma 2.13]{CC08}}]
\label{lem:AntonioModified}
Let $0<d\leq n$ and let $\Gamma$ be a compact $d$-set, satisfying \eqref{eq:dset2} for some constants $\tilde{c}_2>\tilde{c}_1>0$. Let $x\in \Gamma$ and let $f:(0,\infty)\to [0,\infty)$ be non-increasing and continuous.
Then
\begin{align}
\label{eq:AntonioModified}
\tilde{c}_1d\int_0^{\diam(\Gamma)} r^{d-1}f(r)\,\rd r \leq 
\int_\Gamma f(|x-y|)\, \rd \cH^d(y) \leq \tilde{c}_2d\int_0^{\diam(\Gamma)} r^{d-1}f(r)\,\rd r.
\end{align}
\end{lem}

\begin{cor}
\label{cor:HausdorffIntegrability}
Let $0<d\leq n$ and let $\Gamma$ be a compact $d$-set. Then, for $\mu=\cH^d|_\Gamma$, and for any $\eta\in \Gamma$, $I_\Gamma[\Phi_t(\cdot,\eta)]$ is finite if and only if $t<d$. For $\mu=\mu'=\cH^d|_\Gamma$, $I_{\Gamma,\Gamma}[\Phi_t]$ is finite if and only if $t<d$. 
\end{cor}

\subsection{General invariant measures}
\label{sec:IntegrabilityGeneral}
For a more general invariant measure $\mu$, as defined in Section \ref{sec:scaling}, the integrability criterion on $t$ for the single integral $I_\Gamma[\Phi_t(\cdot,\eta)]$ depends on the point $\eta$. Given $\eta\in\Gamma$ let
\[
t_\mu(\eta):= \sup\{t\geq 0, I_\Gamma[\Phi_t(\cdot,\eta)]<\infty\}.
\]
The integral $I_\Gamma[\Phi_t(\cdot,\eta)]$ and the threshold $t_\mu(\eta)$ are called ``generalized electrostatic potential'' (``$t$-potential'' in \cite[(4.12)]{Fal}) and ``electrostatic local dimension'', respectively, in \cite[Defns~3 and 5]{mantica2007asymptotic}. 
From \cite[Chap.~8, p.~109]{Mattila95}, which holds for general Radon measures on $\mathbb R^n$, we have that for $t>0$ 
$$
I_\Gamma[\Phi_t(\cdot,\eta)]
=t \int_0^\infty\frac{\mu(B_r(\eta))}{r^{t+1}}\rd r.
$$
From this it follows that if there exist $C,t'>0$ such that $\mu(B_r(\eta))\leq Cr^{t'}$ for small $r$ then $I_\Gamma[\Phi_t(\cdot,\eta)]<\infty$ for all  $t<t'$, i.e.~$t_\mu(\eta)\geq t'$. 
As in \cite[Eqn~(17.15)]{Fal} we define the local dimension of $\mu$ at $\eta\in\mathbb{R}^n$ (when the limit exists) as
$$
\dimloc\mu(\eta):=\lim_{r\to0}\frac{\log\mu(B_r(\eta))}{\log r}.
$$
From the above observations it follows that if $\dimloc\mu(\eta)$ exists then $t_\mu(\eta)= \dimloc\mu(\eta)$.
By \cite[Thm~2]{geronimo1989exact} we have that if $\Gamma$ is disjoint (see also \cite[Thm~7.4]{strichartz1994self} for the general case) then
\[
\dimloc\mu(\eta) =t_{\rm a.e.} := \frac{\sum_{m=1}^M p_m\log p_m}{\sum_{m=1}^M p_m\log \rho_m}, \qquad \mu\text{-a.e. } \eta \in \Gamma.\]
As a consequence, $t_\mu(\eta)= t_{\rm a.e.}$, for $\mu\text{-a.e. } \eta \in \Gamma$. But $t_\mu(\eta)$ is not in general equal to $t_{\rm a.e.}$ on the whole of $\Gamma$. 
By \cite[Thm 17.4]{Fal}, if $\Gamma$ is disjoint then for all points $\eta\in \Gamma$ where the local dimension exists (which we know from the above is $\mu$-a.e.) we have:
$$\min_{m=1,\ldots,M} \frac{\log p_m}{\log\rho_m}
\le \;t_\mu(\eta)=\dimloc\mu(\eta)\;
\le\max_{m=1,\ldots,M} \frac{\log p_m}{\log\rho_m}.
$$
The upper and the lower bounds coincide, i.e.\ $\log p_m/\log\rho_m$ is the same for all $m=1,\ldots,M$, if and only if $p_m=\rho_m^d$, i.e.\ $\mu=\cH^d|_\Gamma$. 
Moreover, the extremal values are attained, as the following lemma shows. 

\begin{lem}
\label{lem:IntegrabilityFixedPoint}
Let $\Gamma$ and $\mu$ be as in Subsections \ref{sec:IFS} and \ref{sec:scaling}. 
Fix $m\in\{1,\ldots,M\}$ and let $\eta_{m}$ denote the fixed point of the contracting similarity $s_{m}$, i.e.\ the unique point $\eta_{m}\in \Gamma$ such that $s_m(\eta_m)=\eta_m$. Suppose that $\eta_m\not\in \Gamma_{m'}$ for any $m'\in\{1,\ldots,M\}$, $m'\neq m$. (This holds, for instance if 
$\Gamma$ is disjoint in the sense of \eqref{eq:Rdef}.) 
Then there exist $C_2>C_1>0$ such that, for all sufficiently small $r>0$, 
\begin{align}
\label{eq:muBR}
C_1r^{t_m} \leq \mu(B_r(\eta_m)\cap \Gamma) \leq C_2r^{t_m},
\end{align}
where 
\[ t_m:=\frac{\log{p_m}}{\log{\rho_m}}.\]
Hence $t_\mu(\eta_m)=\dimloc\mu(\eta)=t_m$. 
\end{lem}
\begin{proof}
Using the fact that $\eta_m$ is the fixed point of $s_m$, and the fact that 
$R_m:=\min_{m'\neq m}\dist(\eta_m,\Gamma_{m'})>0$, one can show that
\[
s^\ell_m\GG\subset B_{\rho_m^\ell\diam\GG}(\eta_m)\cap\Gamma\quad\text{and}\quad
B_{\rho_m^\ell R}(\eta_m)\cap\Gamma\subset s^\ell_m\GG \quad \forall\ell\in\mathbb{N}, \forall R<R_m,
\]
so that
\[ s_m^{j_r+j}(\Gamma)\subset B_r(\eta_m)\cap \Gamma\subset s_m^{j_r+j'}(\Gamma),\]
where $j_r=\left\lfloor\log r/\log \rho_m\right\rfloor$, $j=1-\left\lfloor\log \diam(\Gamma)/\log \rho_m\right\rfloor$ and $j'=-1-\left\lfloor\log R_m/\log \rho_m\right\rfloor$, provided that $r$ is small enough to ensure that $j_r+j'\geq 0$, i.e.\ $\left\lfloor\log r/\log \rho_m\right\rfloor\geq 1+\left\lfloor\log R_m/\log \rho_m\right\rfloor$. Since $j$ and $j'$ are independent of $r$, the bound \eqref{eq:muBR} follows upon applying $\mu$ and recalling \eqref{eq:sim_measure_general}, which implies that $\mu(s_m^\ell(\Gamma))=p_m^\ell\mu(\Gamma)$ for $\ell\in\mathbb{N}_0$. 

From \eqref{eq:muBR} it follows that $\dimloc\mu(\eta_m)$ exists and equals $t_m$, and hence (by our earlier arguments) that $t_\mu(\eta_m)$ takes the same value. 
\end{proof}

We now consider the double integral \[I_{\Gamma,\Gamma}[\Phi_t]=\int_\Gamma\int_\Gamma \Phi_t(x,y)\,\rd \mu'(y)\rd \mu(x),\]
where, for maximum generality, $\mu$ and $\mu'$ are invariant measures on $\Gamma$ with (possibly different) weights/probabilities $(p_1,\ldots,p_M)$ and $(p_1',\ldots,p_M')$ respectively. Define 
\[ t_{\mu,\mu'}:= \sup\{t\geq0,\, I_{\Gamma,\Gamma}[\Phi_t]<\infty\}.\] 
The integral $I_{\Gamma,\Gamma}[\Phi_t]$ and the threshold $t_{\mu,\mu'}$ are called ``generalized electrostatic energy'' (``$t$-energy'' in \cite[(4.13)]{Fal}) and ``electrostatic correlation dimension'', respectively, in \cite[Defns~4 and 6]{mantica2007asymptotic}.

\begin{lem}
\label{lem:IntegrabilityDouble}
Let $\Gamma$, $\mu$ and $\mu'$ be as above, and suppose that $\Gamma$ is disjoint. Then $t_{\mu,\mu'}=t_{*}$, where $t_{*}$ is the unique positive solution of 
\[\sum_{m=1}^M p_mp_m' \rho_m^{-t_*}=1.\]
\end{lem}
\begin{proof}
To see that $t_{\mu,\mu'}\leq t_{*}$, we note that if $0<t<t_{\mu,\mu'}$ then $I_{\Gamma,\Gamma}[\Phi_t]<\infty$ and, arguing as in the proof of Theorem \ref{th:sing_reform}, 
\[ \left(1-\sum_{m=1}^M p_mp_m' \rho_m^{-t}\right)I_{\Gamma,\Gamma}[\Phi_t] = \sum_{m=1}^M\sum_{\substack{m'=1\\m'\neq{m}}}^M I_{\Gamma_m,\Gamma_{m'}}[\Phi_t].\] 
Since $\Phi_t(x,y)>0$ for $x\neq y$ the integrals $I_{\Gamma_m,\Gamma_{m'}}[\Phi_t]$ are all positive, implying that the right-hand side is non-zero, so that the factor $1-\sum_{m=1}^M p_mp_m' \rho_m^{-t}$ cannot vanish, i.e.\ $t\neq t_{*}$. Since this holds for all $0<t<t_{\mu,\mu'}$ we must have $t_{\mu,\mu'}\leq t_{*}$. 

To prove that $t_{\mu,\mu'}\geq t_{*}$ we adopt an argument suggested by K. Falconer \cite{FalPrivateComm}. Suppose that $0<t<t_{*}$. Then 
\[ 0<\lambda_t:=\sum_{m=1}^M p_{m}p'_{m}\rho_{m}^{-t}<1,\]
and we can write,  where $I_0=\{0\}$ and $(0,\bm)$ stands for $\bm$,
\begin{align*}
\label{}
I_{\Gamma,\Gamma}[\Phi_t] 
&= \sum_{\ell=0}^\infty \sum_{\bm\in I_{\ell}}\sum_{m=1}^M\sum_{\substack{m'=1\\m'\neq{m}}}^M \int_{\Gamma_{(\bm,m)}}\int_{\Gamma_{(\bm,m')}} |x-y|^{-t}\,\rd \mu'(y)\rd\mu(x)\\
&\leq \sum_{\ell=0}^\infty \sum_{\bm\in I_{\ell}}\sum_{m=1}^M\sum_{\substack{m'=1\\m'\neq{m}}}^M \mu(\Gamma_{(\bm,m)})\mu'(\Gamma_{(\bm,m')})\left(R_\Gamma\prod_{j=1}^\ell \rho_{m_j}\right)^{-t}\\
&\leq \sum_{\ell=0}^\infty \sum_{\bm\in I_{\ell}} \mu(\Gamma_{\bm})\mu'(\Gamma_{\bm}) \left(R_\Gamma\prod_{j=1}^\ell \rho_{m_j}\right)^{-t}\\
&= R_\Gamma^{-t}{\mu\GG\mu'\GG}\sum_{\ell=0}^\infty \sum_{\bm\in I_{\ell}} \prod_{j=1}^\ell p_{m_j}\prod_{j=1}^\ell p'_{m_j}\prod_{j=1}^\ell \rho_{m_j}^{-t} \\
&= R_\Gamma^{-t}{\mu\GG\mu'\GG}\sum_{\ell=0}^\infty \sum_{\bm\in I_{\ell}} \prod_{j=1}^\ell p_{m_j}p'_{m_j}\rho_{m_j}^{-t} \\
&= R_\Gamma^{-t}{\mu\GG\mu'\GG}\sum_{\ell=0}^\infty \left(\sum_{m=1}^M p_{m}p'_{m}\rho_{m}^{-t}\right)^\ell \\
&= R_\Gamma^{-t}{\mu\GG\mu'\GG}\sum_{\ell=0}^\infty \lambda_t^\ell,
\end{align*}
which is finite since $0<\lambda_t<1$. Hence $I_{\Gamma,\Gamma}[\Phi_t]<\infty$ for all $0<t<t_{*}$, which implies that $t_{\mu,\mu'}\geq t_{*}$. 
\end{proof}
}

\section*{Acknowledgements}

The authors acknowledge support from EPSRC grants EP/S01375X/1 (DH) and EP/V053868/1 (DH and AG), 
from PRIN project ``NA-FROM-PDEs'' and from MIUR through the ``Dipartimenti di Eccellenza'' Programme (2018-2022) -- Dept.\ of Mathematics, University of Pavia (AM),
and thank Ant\'onio Caetano, Simon Chandler-Wilde, Kenneth Falconer, Uta Freiberg, Giorgio Mantica, and the two anonymous reviewers for helpful discussions in relation to this work.

\section*{Statements and Declarations}

\subsection*{Competing interests}

The work of AG and DH is supported ``in kind'' (through staff time and equipment use) by the UK Met Office, who are the industrial partner on grant EP/V053868/1. All authors certify that they have no other affiliations with or involvement in any organisation or entity with any financial interest or non-financial interest in the subject matter or materials discussed in this manuscript.

\subsection*{Data availability}
\red{Data sharing is not applicable to this article as no datasets were generated or analysed during the current study.}

\bibliography{fractalQuad}
\bibliographystyle{siam}
\end{document}